\documentclass{siamart250211}  
\usepackage{amsmath}
\usepackage{float,subfigure}

\usepackage[sort&compress]{natbib}
\usepackage{afterpage}
\usepackage{amsfonts}

\usepackage{amssymb}
\usepackage{graphicx}
\usepackage{amssymb}
\usepackage{amsfonts}
\usepackage{url}
\usepackage{epstopdf}
\usepackage{color}
\usepackage{tikz-cd}
\usepackage{bm}
\usepackage{multirow}
\usepackage{rotating}

\newcommand{\bsub}{\begin{subequations}}
\newcommand{\esub}{\end{subequations}$\!$}

\newcommand{\eps}{{\varepsilon}}
\newcommand{\X}{{\bf X}}
\newcommand{\B}{{\mathcal B}}

\newcommand{\kapdif}{\mathcal{Q}}
\newcommand{\orthomat}{\mathbf{Q}}

\newcommand{\teta}{\eta}
\newcommand{\txione}{\xi_1}
\newcommand{\txitwo}{\xi_2}
\newcommand{\tdelta}{\Delta_{\y}}

\newcommand{\HC}{{\mathcal H}}

\newcommand{\highblue}[1]{{\color{blue}#1}}

\newcommand{\M}{\mathcal M}

\newcommand{\pa}{{\partial\Omega}}

\newcommand{\clb}{\color{black}}

\newcommand{\x}{{\bf{x}}}

\newcommand{\vc}{{\bf{C}}}

\newcommand{\vb}{{\bf{a}}}
\newcommand{\y}{{\bf{y}}}
\newcommand{\n}{{\bf{n}}}

\newcommand{\area}{|\Omega|}
\newcommand{\R}{{\mathbb{R}}}

\newcommand{\PT}{{\Gamma}}
\newcommand{\DT}{{\bf p}}

\newtheorem{prop}{Proposition}

\renewcommand\thesection{\arabic{section}}
\renewcommand\thesubsection{\thesection.\arabic{subsection}}
\renewcommand\thesubsubsection{\thesubsection.\arabic{subsubsection}}
\renewcommand\theequation{\thesection.\arabic{equation}}

\graphicspath{{figures/}}

\title{Asymptotic Analysis of the Narrow Escape and Berg-Purcell problems on general three-dimensional domains with reactive boundary patches}

\author{Alan E. Lindsay\thanks{University of Notre Dame, Notre Dame,
    IN, 46556, USA.  Email: {\tt a.lindsay@nd.edu} (corresponding author)},
  \and Andrew J. Bernoff\thanks{Department of Mathematics, Harvey Mudd College, Claremont, CA, 91711, USA.  Email: {\tt bernoff@g.hmc.edu}}, 
  \and Denis
  S. Grebenkov\thanks{Laboratoire de Physique de la Mati\'{e}re
    Condens\'{e}e, CNRS -- Ecole Polytechnique, Institut Polytechnique
    de Paris, 91120 Palaiseau, France. Email:
    {\tt denis.grebenkov@polytechnique.edu}}, \and Jeremy
  G. Hoskins\thanks{Department of Statistics, University of Chicago,
    USA and NSF-Simons National Institute for Theory and Mathematics
    in Biology, Chicago, IL. Email: {\tt jeremyhoskins@uchicago.edu}}
  \and Michael J. Ward\thanks{Department of Mathematics, University of
    British Columbia, Vancouver, B.C., Canada, V6T 1Z2. Email:
    {\tt ward@math.ubc.ca}}}

\date{\today}
\begin{document}

\label{firstpage}
\maketitle

\baselineskip=12pt

\begin{abstract} \clb
We present an asymptotic analysis of two diffusive capture problems in
general smooth closed three-dimensional geometries with multiple small
reactive boundary patches of arbitrary shapes.  (i) The narrow escape
problem seeks to determine the escape rate of Brownian particles from
an enclosed region through small boundary windows.  (ii) The related
Berg-Purcell (or narrow entrance) problem seeks to resolve the capture
rate for signaling molecules diffusing outside the cell and entering
through localized reactions at membrane-bound receptors.  We obtain
matched asymptotic solutions of these two problems and thus address
the long-standing challenge of describing the role that curvature and
local reactivities play in modulating diffusive capture rates.  Our
explicit expansions quantify local effects on diffusive capture
through the sizes, shapes, and reactivities of the patches together
with the principal curvatures of the manifold at each patch.  In turn,
we examine global effects on diffusive capture such as the spatial
configuration of patches on the manifold, as encloded by the
associated surface Neumann Green's function and its regular part.  The
accuracy of our asymptotic formulas is validated against a full
numerical solution for an ellipsoidal domain.  Overall, our results
yield new insights on how geometry and stochasticity combine to shape
the dynamics of various biological processes.
%We present an asymptotic analysis of diffusive capture problems in
%general smooth and closed three dimensional geometries with small
%reactive boundary patches. We obtain matched asymptotic solutions to
%the narrow escape and Berg-Purcell problems for multiple partially
%reactive and non-overlapping boundary sites. The narrow escape problem
%seeks to determine the escape rate of Brownian particles from an
%enclosed region through small boundary windows. The related
%Berg-Purcell problem seeks to resolve the reaction rate of a cell with
%diffusing extracellular signaling molecules through localized
%reactions at membrane bound receptors.  Our asymptotic solution in the
%scenario of a general three-dimensional manifold addresses the
%long-standing problem of describing the role that curvature and
%reaction rates play in modulating diffusive capture rates. Our
%explicit expansions quantify the effect of the role that the size and
%reactivity of the sites together with their geometric configuration on
%the manifold. We identify local geometric effects in terms of the
%principal curvatures of the manifold at each site while global
%geometric effects are described through the regular part of the
%associated Green's function. Results from our asymptotic analysis are
%validated against full numerical results for an ellipsoidal-shaped
%domain. Overall, our results yield new insights on how geometry and
%stochasticity combine to shape the dynamics of biological processes.
%\vspace*{0.1cm}
\end{abstract}


% MSC 35B25, 35J05, 35P05, 58J50, 92C05

\section{Introduction}\label{sec:into}

The event where a diffusing particle hits a threshold or arrives at a
small reaction site is a key milestone in the completion of numerous
biological phenomena.  Examples include the binding of the T-cell
receptor to a ligand on an opposing cell, the entry of a virus into a
cell or the trafficking of intra-cellular molecules through nuclear
pore complexes
\cite{ZSchuss2013,Wei2011,benichou2014first,schuss2012narrow,RH,app10186543,redner2001guide,FPPA2014,NewbyBressloff2013,Bressloff2024,Grebenkov2023}.

An ongoing and challenging avenue of research is to describe how the
interaction of diffusive processes with complex geometries affects the
timescale of such processes incling shielding, confinement and
boundary induced motion. The interaction of diffusive particles with
curved manifolds is of great importance in understanding cellular
functions where morphology plays a significant role
\cite{Endres2025,Cavanagh2020,GomezCheviakov2015,Curvature2023}.

In this paper, we describe the solution by matched asymptotic analysis
of two related problems in diffusive transport, the {\em Narrow Escape
Problem} (NEP)
\cite{SchussSinger2007,schuss2012narrow,HolcmanReview2014,chen,Lagache2017,carillo2026}
and the {\em Berg-Purcell (BP) problem}
\cite{bergp,LLM2020,HANDY20212237,LWB2017}.  Before introducing the
precise mathematical {\clb formulation} of these problems and our main
results, we outline the geometric setting which is common to both.

Let $\M \subset {\mathbb R}^3$ be a three-dimensional simply-connected
bounded domain,
%that has many small partially reactive
%sites embedded in its boundary $\partial\M$. We assume that the
{\clb whose smooth reflecting} boundary $\partial\M$ hosts the union
$\partial {\mathcal M}_a=\cup_{i=1}^{N} \partial {\mathcal M}_i$ of
$N$ small {\clb partially} reactive patches $\partial {\mathcal M}_i$.
%on an otherwise reflecting boundary.  
Each boundary patch $\partial {\mathcal M}_i$ has a length-scale
$L_i$, and a {\clb constant} reactivity $\B_i>0$.  {\clb More
  specifically, we denote $2L_i$ to be the diameter of the orthogonal
  projection of the patch $\partial\Omega_i$ when mapped onto the
  tangent plane.}  The geometry of each patch is assumed to be
simply-connected with a smooth boundary, but with an otherwise
arbitrary shape.  Within this setting, the NEP is posed on the
interior of $\M$ while the BP problem is solved in the exterior region
$\R^3\setminus\M$.  Owing to differences that arise from this
geometric consideration, we outline the two problems separately.

\begin{figure}[htbp]
\centering
\includegraphics[width=0.9\textwidth]{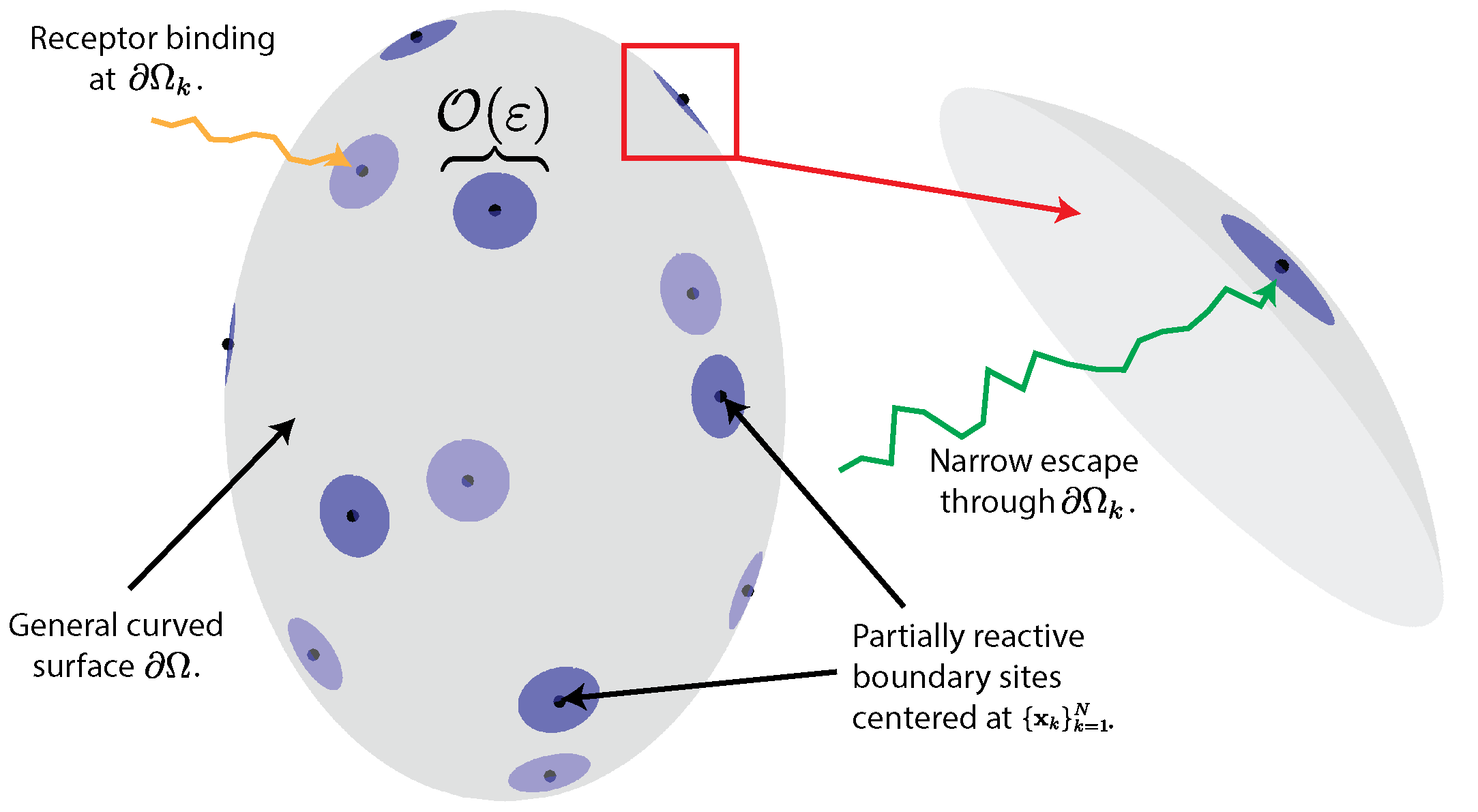}
\caption{
Three-dimensional diffusion interior ($\x\in\Omega$) and exterior
($\x\in\mathbb{R}^3\setminus\Omega$) to a general closed and smooth
curved surface $\partial\Omega$.  The Berg-Purcell problem
(\ref{berg_intro:ssp0}) {\clb solves the exterior diffusion problem
and} seeks the {\clb capture} rate of ligands on a cell surface with
localized reaction receptors $\{\partial\Omega_k\}_{k=1}^N$.  The
narrow escape problem (\ref{mfpt_intro:ssp}) {\clb solves the interior
diffusion problem and} seeks the escape rate of diffusing molecules in
$\Omega$ through small boundary windows
$\{\partial\Omega_k\}_{k=1}^N$.  The non-overlapping {\clb reactive
patches} are centered at points $\{\x_k \}_{k=1}^N $ with spatial
extent $\varepsilon\ll1$.}
\label{fig:intro_schem}
\end{figure}

\noindent \underline{The Berg-Purcell (BP) problem:} A foundational model of
diffusion-mediated signaling, introduced by Howard Berg and Edward
Purcell \cite{bergp}, posits that interactions between cells and the
extra-cellular medium is negotiated through binding of ligands to
membrane-bound surface receptors.  In this context, the steady-state
concentration ${\mathcal U}$ of diffusing ligands satisfies the mixed
Neumann-Robin boundary value problem (BVP)
\bsub \label{berg_intro:ssp0}
\begin{align}
  \Delta_{\X} {\mathcal U} & = 0 \,, \quad \X \in \R^3\backslash {\mathcal M} \,,
  \\  \label{berg_intro:RC}
   D\partial_{n} {\mathcal U} + \B_i {\mathcal U} & = 0\,, \quad \X \in
     \partial {\mathcal M}_i \,, \quad i=1,\ldots,N \,, \\
  \partial_{n} {\mathcal U} & = 0 \,, \quad \X \in \partial {\mathcal M}_r=
                              \partial {\mathcal M}\backslash\cup_{i=1}^{N}
                              \partial{\mathcal M}_i \,,\\
  {\mathcal U} & \sim  {\mathcal U}_{\rm \inf}\left(1 -\frac{{\mathcal C}_{\rm T}}
                 {|\X|} +    {\mathcal O}(|\X|^{-2})\right)
  \,,  \quad \mbox{as} \quad |\X|\to \infty \,, \label{berg_intro:ssp0_c}
\end{align}
\esub 
where $D>0$ is the constant diffusivity, ${\mathcal U}_{\rm \inf}$ is
the constant concentration imposed at infinity, $\Delta_{\X}$ is the
Laplacian in the dimensional coordinate $\X$, and $\partial_n$ is the
outward normal derivative to $\partial {\mathcal M}$, directed into
${\mathcal M}$.  We emphasize that the original work \cite{bergp}
dealt with perfectly reactive patches ($\B_i = \infty$), in which case
the Robin condition (\ref{berg_intro:RC}) is reduced to the Dirichlet
condition ${\mathcal U} = 0$; various physical and probabilistic
interpretations of the Robin boundary condition (\ref{berg_intro:RC})
are discussed in \cite{Grebenkov20,Grebenkov2023}.  The monopole
coefficient ${\mathcal C}_{\rm T}$ in the far-field
(\ref{berg_intro:ssp0_c}) is referred to as the {\em capacitance} of
the heterogeneous partially reactive boundary $\partial \M$.  The key
quantity of interest is the combined flux through the whole cell surface, given
by
\begin{equation}\label{eq:j_origin}
  J\equiv  - D \int\limits_{\partial {\mathcal M}} \partial_n {\mathcal U}\, dS\,
  =  
    4\pi D \, {\mathcal U}_{\rm \inf} {\clb \mathcal C}_{\rm T} \,.
\end{equation}
In arriving at \eqref{eq:j_origin}, we have used the divergence theorem to
relate the total flux $J$ to the far field behavior
(\ref{berg_intro:ssp0_c}). {\clb Our analysis will provide an asymptotic expansion for the total flux $J$ in the small patch diameter limit.}

To analyze (\ref{berg_intro:ssp0}), we introduce the dimensionless
variables
\begin{equation}\label{intro_bp:scalings}
    \x = \frac{\X}{R_{\star}}, \quad L =\max\limits_{i} \{L_i\}, \quad
    \eps = \frac{L}{R_{\star}}, \quad a_i=\frac{L_i}{L}, \quad b_i =
    \frac{L \B_i}{D}, \quad  u = \frac{{\mathcal U}}{{\mathcal U}_{\rm inf}},
  \quad C_{\rm T}=\frac{{\mathcal C}_{\rm T}}{R_{\star}},
\end{equation}
where $R_{\star}$ is a length-scale for ${\mathcal M}$.  We also
consider the rescaled domain $\Omega = R^{-1}_{\star} \M$ and the
rescaled patches
$\partial\Omega^{\eps}_i=R_{\star}^{-1}\partial{\mathcal M}_{i}$ of
small ${\mathcal O}(\eps)$ diameter, which are assumed to satisfy
$\partial\Omega^{\eps}_i\to\x_i\in \partial\Omega$ as $\eps\to 0$.
Moreover, we assume that the patches are well-separated in the sense
that $\left|\x_i-\x_j\right|={\mathcal O}(1)$ for all $i\neq j$ as
$\eps\to0$.
In the region exterior to $\Omega$,
(\ref{berg_intro:ssp0}) becomes \bsub \label{berg_bp:ssp}
\begin{align}
  \Delta_{\x} u & = 0 \,, \quad \x \in \R^3\backslash\Omega \,,
                  \label{berg_bp:ssp_1}\\
  \eps\partial_{n} u + b_i u & = 0\,, \quad \x \in \partial\Omega^{\eps}_i
                                    \,, \quad i=1,\ldots,N \,,\label{berg_bp:ssp_2b} \\
  \partial_{n} u & = 0 \,, \quad \x \in \partial \Omega_r
      =\partial\Omega\backslash\cup_{i=1}^{N}\partial\Omega^{\eps}_i\,,
                   \label{berg_bp:ssp_2}\\
  u & \sim  1 - \frac{C_{\rm T}}{|\x|} +{\mathcal O}(|\x|^{-2})
  \,,  \quad \mbox{as} \quad |\x|\to \infty \,, \label{berg_bp:ssp_3}
\end{align}
\esub where $\Delta_{\x}$ is the Laplacian in $\x$, and $\partial_n$
is the outward normal derivative, directed into $\Omega$.  The
relevant biophysical question is to understand the dependence of $J_i
\equiv J_i(\x_1,\ldots,\x_N)$ on the geometry $\Omega$ together with
the trapping configuration as described by the locations
$\{\x_1,\ldots,\x_N \}\in\partial\Omega$ of the centers of the
patches, their relative sizes, shapes, and reactivities
$\{b_1,\ldots,b_N\}$ (see Fig.~\ref{fig:intro_schem}).  The
distribution of fluxes across the reactive sites plays an important
role in cellular mechanisms of source detection
\cite{LLM2020,Lindsay2023a,BJNL2023,DOBRAMYSL201822,BL2025}.

In \S \ref{sec:berg} we will analyze {\clb the BP problem}
(\ref{berg_bp:ssp}) in the small patch limit $\eps\to 0$ by using the
method of matched asymptotic expansions.  Our analysis 
%for (\ref{berg:main_res}) 
provides a three-term asymptotic expansion for the capacitance $C_{\rm
T}$, which is valid for a general smooth and closed domain $\Omega$.
The result extends previous works \cite{LBS2018,LB17,LLM2020} where
perfectly absorbing circular patches ($b_i=\infty$) on the boundary of
a sphere were considered.  It also extends upon the results of
\cite{GrebenkovWard2026} where a three-term expansion for $C_{\rm T}$
was derived for a spherical domain with partially reactive boundary
patches of arbitrary shape.  Our result $C_{\rm T}$ for an arbitrary
bounded 3-D domain is given by (\ref{berg:main_res_1}) and depends on
the exterior surface Neumann Green's function, which will be computed
numerically, together with some local geometric properties of the
domain boundary $\partial\Omega$ near the patch locations.  
An important case of biological interest is $N$ identical {\clb
perfectly} absorbing circular patches of radius $\eps$.  In this
scenario, our main result for $\eps\to 0$, as given in Proposition
\ref{berg:main_res}, simplifies to
\bsub\label{eq:intro_CT}
\begin{equation}\label{eq:intro_CT_a}
  \frac{1}{C_T} \sim \frac{\pi}{\eps N}\left[ 1 -
    \frac{\eps\overline{\HC}}{N\pi} \log(4 \eps) + \frac{4\eps}{N}
    \left( p_e(\x_1,\ldots,\x_N) + \frac{3\overline{\HC}}{8\pi}\right) +
   {\mathcal O}(\eps^2\log\eps) \right] \,.
\end{equation}
Here 
\begin{equation} \label{eq:HCsum}
\overline{\HC} \equiv \sum_{i=1}^N\HC_i, 
\end{equation}
where $\HC_i$ is the mean curvature of the surface at $\x=\x_i$,
defined in terms of the two principal curvatures $\kappa_1(\x)$ and
$\kappa_2(\x)$ at $\x=\x_i$ by $\HC_i =
\tfrac{1}{2}\left[\kappa_1(\x_i)+
\kappa_2(\x_i)\right]$, with sign convention that $\HC_i=-1$ if
$\Omega$ is the unit sphere.  In (\ref{eq:intro_CT_a}), the
inter-patch interaction energy $p_e(\x_1,\ldots,\x_N)$ is defined in
terms of the exterior surface Neumann Green's function $G_e$ and its
regular part $R_e$ (see (\ref{berg:green_int}) and
(\ref{berg:green_int_sing}) below) by
\begin{equation}\label{eq:intro_CT_b}
  p_e(\x_1,\ldots,\x_N) =\sum_{i=1}^N\Big[R_{e}(\x_i) +
  \sum_{\substack{j=1\\j\neq i}}^N G_{e}(\x_j;\x_i)\Big]\,.
\end{equation}
\esub

\noindent \underline{The Narrow Escape Problem (NEP):} By using a
similar asymptotic framework, we will also analyze the {\em mean
first-reaction time} (MFRT) for diffusing ligands inside $\M$ to react
on a collection of small well-separated partially reactive patches
situated on an otherwise reflecting boundary. Adopting the same
notation as (\ref{berg_bp:ssp}), the MFRT ${\mathcal U}(\X)$ for a
Brownian particle starting at $\X\in {\mathcal M}$ satisfies the mixed
Neumann-Robin problem
\begin{subequations}  \label{mfpt_intro:ssp}
\begin{align}   
-D \Delta_{\X} {\mathcal U} & = 1 \,, \quad \X\in {\mathcal M}\,, \\ 
  D \partial_n {\mathcal U} + \B_i {\mathcal U} & = 0\,,
                                                 \quad \X \in
     \partial {\mathcal M}_i \,, \quad i=1,\ldots,N \,,\\
  \partial_{n} {\mathcal U} & = 0 \,, \quad \X \in \partial {\mathcal M}_r=
                              \partial {\mathcal M}\backslash\cup_{i=1}^{N}
                              \partial{\mathcal M}_i \,,
\end{align}
\end{subequations}
where $\partial_n$ now denotes the normal derivative directed outward
to ${\mathcal M}$.  A central quantity to describe is the {\clb global
(or volume-averaged)} MFRT,
\begin{equation} \label{mfpt_intro:global} \overline{U} \equiv
  \frac{1}{|{\mathcal M}|} \int_{{\mathcal M}} U(\X) \, d\X \,,
\end{equation}
which corresponds to the average MFRT with respect to a uniform
distribution of initial points $\X\in{\mathcal M}$.  Here $|{\mathcal
M}|$ denotes the volume of ${\mathcal M}$.  We non-dimensionalize the
MFRT problem (\ref{mfpt_intro:ssp}) by introducing the dimensionless
variables defined by
\begin{equation}\label{intro:scalings}
  \x = \frac{\X}{R_{\star}} \,, \quad L =\max\limits_{i} \{L_i\} \,, \quad
  \eps = \frac{L}{R_{\star}}
  \,, \quad a_i=\frac{L_i}{L} \,, \quad b_i = \frac{L \B_i}{D} \,,
  \quad     u(\x)=\frac{D}{R_{\star}^2}U(\x R_{\star}) \,.
\end{equation}
In terms of (\ref{intro:scalings}), the dimensionless MFRT $u(\x)$
with partially reactive patches and dimensionless reactivities $b_i>0$
satisfies
\bsub \label{mfpt:ssp}
\begin{align}
  \Delta_{\x} u & = - 1 \,, \quad \x \in \Omega \,,
                  \label{mfpt:ssp_1}\\
  \eps\partial_{n} u + b_i u & = 0\,, \quad \x \in \partial\Omega^{\eps}_i
                                    \,, \quad i=1,\ldots,N \,, \\
  \partial_{n} u & = 0 \,, \quad \x \in \partial \Omega_r
       =\partial\Omega\backslash\cup_{i=1}^{N} \partial\Omega^{\eps}_{i}\,,
                   \label{mfpt:ssp_2}
\end{align}
\esub 
with the same notation and assumptions on the patches as for
(\ref{berg_bp:ssp}). For a uniform distribution of initial points
$\x\in\Omega$, the dimensionless {\clb global} MFRT is
\begin{equation} 
 \overline{u} \equiv \frac{1}{\area} \int_{\Omega} u(\x) \,
 d\x \,. \label{mfpt:ubar}
\end{equation}
By using the method of matched asymptotic expansions, we will derive
in \S \ref{sec:mfpt} a three-term asymptotic expansion for
$\overline{u}$ in the limit $\eps\to 0$, which is valid for a general
smooth and closed domain $\Omega$.  Our analysis extends that of
\cite{cheviakov2010asymptotic} and \cite{GrebenkovWard2026big} for
perfectly absorbing and for partially reactive patches, respectively,
on the boundary of a spherical domain.  It also extends the work of
\cite{NURSULTANOV2021202} for a single perfectly absorbing boundary
patch by describing the interaction of $N$ non-overlapping partially
reactive patches.  Finally, it extends to higher order the two-term
asymptotic result of \cite{GomezCheviakov2015} for perfectly absorbing
circular patches on the boundary of non-spherical domains.

For the important case of $N$ identical perfectly absorbing circular
patches with common radius $\eps$, our main result for
(\ref{mfpt:ubar}) given in Proposition
\ref{mfpt_b:main_res} simplifies to
\begin{equation}\label{eqn:MFPT_reduced}
  \overline{u} \sim \frac{|\Omega|}{4N\eps}\left[ 1 -
    \frac{\eps \overline{\HC}}{N\pi} \log 4\eps +
    \frac{4\eps}{N} \left( p_{s}(\x_1,\ldots,\x_N) +
    \frac{3\overline{\HC}}{8\pi} \right)+ {\mathcal O}(\eps^2\log\eps)\right]\,.
\end{equation}
In (\ref{eqn:MFPT_reduced}), the quantities $\overline{\HC}$ and
$p_{s}(\x_1,\ldots,\x_N)$ take the form outlined in
(\ref{eq:intro_CT_b}), but with values drawn from the interior surface
Neumann Green's function $G_{s}$ and its regular part $R_{s}$
satisfying (\ref{mfpt:green_int}) and (\ref{mfpt:green_int_sing}).  In
(\ref{eqn:MFPT_reduced}), the sign convention is now $\HC_i=1$ if
$\Omega$ is the unit sphere.

Our asymptotic analysis of (\ref{berg_bp:ssp}) and (\ref{mfpt:ssp})
relies heavily on a careful resolution of the singularity behavior of
the surface Neumann Green's function with Dirac source located on the
curved boundary. While some previous works on resolving this behavior
were given in
\cite{Silbergleit2003,Popov92,Holcman2006a,NURSULTANOV2021202,lindsay20263D_GFun},
our local analysis in Appendix \ref{app_g:int_green} of this
singularity behavior is characterized in terms of the local
tangential-normal coordinate system near each patch (see Appendix
\ref{app_b:bernoff}).  This local coordinate system recapitulates the
weak path-dependent singularity structure uncovered using a microlocal
analysis approach in \cite{NURSULTANOV2021202}, but in terms of
coordinates that are especially convenient for our asymptotic
analysis.  In comparison to the analysis in \cite{Popov92} and
\cite{Holcman2006a}, which yielded the first two terms of the
singularity behavior only on the surface, our analysis resolves the
first three terms of this singularity behavior, and is valid in a
small neighborhood of the singular point, i.e. both on and off the
surface.

The outline of the paper is the following.  In \S \ref{sec:berg} we
analyze the BP problem (\ref{berg_bp:ssp}) in the limit $\eps\to0$.
Our result provides explicit formulas for the capacitance $C_T$
(\ref{berg_bp:ssp_3}) in
terms of the reactivities and spatial configuration of boundary
patches.  In \S \ref{sec:mfpt}, we apply this asymptotic analysis to
the NEP defined in (\ref{mfpt:ssp}).  Again our result takes the form
of an explicit asymptotic expansion for the mean exit time of a
Brownian particle through multiple boundary windows. Numerical
validation of our results are given in \S \ref{sec:results} and we
show that the derived expansion has the errors predicted by the
asymptotic theory.  We also demonstrate that our asymptotic framework
is able to accurately describe the effects that non-constant surface
curvature has on the capture statistics of Brownian particles to small
reactive sites. In \S \ref{sec:discussion} we conclude by discussing
some future research directions which emanate from this work.

\section{Asymptotic analysis of the Berg-Purcell problem}\label{sec:berg}

We now use the method of matched asymptotic expansions to derive a
three-term asymptotic expansion for (\ref{berg_bp:ssp}) in the
small-patch limit $\eps\to 0$. The analysis is valid for arbitrary
reactivities and arbitrary patch shapes on the domain boundary.

For our analysis, it is convenient to introduce $U$ by
\begin{equation}\label{berg:change}
  u=-C_{\rm T} U \,,
\end{equation}
so that from (\ref{berg_bp:ssp}) we find that $U$ satisfies
\bsub \label{berg:ssp}
\begin{align}
  \Delta_{\x} U & = 0 \,, \quad \x \in \R^3\backslash\Omega \,,
                  \label{berg:ssp_1}\\
  \eps\partial_{n} U + b_i U & = 0\,, \quad \x \in \partial\Omega^{\eps}_i
                                    \,, \quad i=1,\ldots,N \,, \\
  \partial_{n} U & = 0 \,, \quad \x \in \partial \Omega_r
   = \partial\Omega\backslash\cup_{i=1}^{N} \partial\Omega^{\eps}_{i}\,,
     \label{berg:ssp_2}\\
  U & \sim  -\frac{1}{C_{\rm T}} + \frac{1}{|\x|} + \frac{\DT_T {\bf \cdot} \x}
      {C_{\rm T}|\x|^3} + \cdots\,, \quad \mbox{as}\quad |\x|\to \infty \,.
      \label{berg:ssp_3}
\end{align}
In (\ref{berg:ssp_3}), $\DT_T$ is the dipole vector.  Equivalently, we
can write the far-field condition (\ref{berg:ssp_3}) as a flux
condition over the boundary $\partial\Omega_\sigma$ of a large sphere
of radius $\sigma$ centered at $\x=0$, in the form
\begin{equation}\label{berg:flux}
  \lim_{\sigma\to\infty} \int_{\partial\Omega_\sigma} \partial_{r} U \vert_{r=\sigma}
  \, dS = - 4\pi\,.
\end{equation}
\esub

In the outer region away from the boundary patches we expand the outer
solution for (\ref{berg:ssp}) as
\begin{equation}
  U \sim \eps^{-1} U_0 + U_1 + \eps \log(\eps) \, U_2 + \eps U_3 +
   \ldots \,,
  \label{berg:outex}
\end{equation}
where $U_0$ is a constant to be determined. The correction terms $U_k$
for $k\geq 1$ satisfy
\begin{equation}  \label{berg:Uk}
\begin{split}
&  \Delta_{\x} U_k = 0 \,, \quad \x \in \R^{3}\backslash \Omega \,; 
  \qquad \partial_n U_k = 0 \,, \quad \x\in \partial\Omega\backslash
  \lbrace{\x_1,\ldots,\x_N\rbrace} \,, \\
  &  \lim_{\sigma\to\infty} \int_{\partial\Omega_\sigma} \partial_n U_k \vert_{r=\sigma}
  \, dS = - 4\pi \delta_{k1}\,,
\end{split}
\end{equation}
where $\delta_{k1}$ is the Kronecker symbol.  Singularity behaviors for
$U_{k}$ as $\x\to \x_i$, for each $i=1,\ldots,N$, will be derived by
asymptotic matching to the local or inner solutions near each patch.

Near each boundary patch centered at $\x=\x_i\in \partial\Omega$, in
(\ref{app_b:coord}) of Appendix \ref{app_b:bernoff} we introduce a
local tangential-normal coordinate system $(t_1,t_2,d)^{T}$, where
$d=\mbox{dist}(\x,\partial\Omega)\geq 0$, and where $t_1$ and $t_2$
are aligned with the two principal directions through $\x_i\in
\partial\Omega$ associated with the two principal curvatures
$\kappa_{1i}$ and $\kappa_{2i}$, respectively.  We define the
$\eps-$localized coordinates by $\txione\equiv {t_1/\eps}$,
$\txitwo\equiv {t_2/\eps}$, and $\teta\equiv {d/\eps}$, and label the
inner solution near each $\x_i$ as $V_i(\y)$, where $\y\equiv
(\txione,\txitwo,\teta)^T$.

We then expand the inner solution near each patch as
\begin{equation}
  V  \sim \eps^{-1} V_{0i} + \log(\eps)\, V_{1i} +
  V_{2i}  +   \ldots \,. \label{berg:innex}
\end{equation}
Upon substituting (\ref{berg:innex}) into the local transformation
(\ref{app_b:local_laplace}) of the Laplacian derived in Appendix
\ref{app_b:bernoff}, we obtain that $V_{ki}$ for $k=0,1$ satisfies
\bsub \label{berg:Vk}
\begin{align}
  \tdelta V_{ki} &= 0\,, \quad
   \y \in \R_{+}^{3} \,, \label{berg:Vk_1}\\
   -\partial_{\teta} V_{ki} + b_i V_{ki} &=0 \,, \quad \teta=0 \,,\,
    (\txione,\txitwo)\in \PT_i\,,  \label{berg:Vk_2}\\
  \partial_{\teta} V_{ki} &=0 \,, \quad \teta=0 \,,\, (\txione,\txitwo)
                            \notin \PT_i  \,, \label{berg:Vk_3}
\end{align}
\esub 
where $\tdelta\equiv \partial_{\txione\txione}+\partial_{\txitwo\txitwo}+
\partial_{\teta\teta}$, {\clb with the behavior at infinity being fixed
below via asymptotic matching conditions.}  Here $\PT_i \asymp
\eps^{-1}\partial\Omega^{\eps}_i$ is the {\clb flattened}
compactly-supported partially reactive patch on the horizontal plane
$\teta=0$, obtained from the $\eps-$localized tangential-normal
coordinate system (e.g., if $\partial\Omega_i^\eps$ is a small
spherical cap of radius $\eps a_i$, then $\Gamma_i$ is the disk of
radius $a_i$; in other words, $\Gamma_i$ represents a rescaled
flattened shape of the patch in the limit $\eps\to 0$). In addition,
we find that $V_{2i}$ satisfies
\bsub \label{berg:V2}
\begin{align}
&  \tdelta V_{2i} = 2 \HC_i\left(\teta V_{0i,\teta\teta} +  V_{0i,\teta}
        \right) - 2\kapdif_i\teta\left(V_{0i,\txione\txione}-V_{0i,\txitwo\txitwo}
                       \right) \,, \quad
   \y \in \R_{+}^{3} \,, \label{berg:V2_1}\\
 &  -\partial_{\teta} V_{2i} + b_i V_{2i} =0 \,, \quad \teta=0 \,,\,
    (\txione,\txitwo)\in \PT_i\,,  \label{berg:V2_2}\\
&  \partial_{\teta} V_{2i} =0 \,, \quad \teta=0 \,,\, (\txione,\txitwo)
                            \notin \PT_i  \,. \label{berg:V2_3}
\end{align}
\esub 
Here $\HC_i$ and $\kapdif_i$ are the mean curvature and curvature
difference of $\partial\Omega$ at $\x_i$, given
\begin{equation}\label{berg:curve}
  \HC_i \equiv \frac{1}{2}(\kappa_{1i}+\kappa_{2i}) \,, \qquad
  \kapdif_i \equiv \frac{1}{2}(\kappa_{1i}-\kappa_{2i}) \,,
\end{equation}
with $\HC_i<0$ if $\partial\Omega$ is convex at $\x_i$ as seen from a
point $\x\in \R^3\setminus \Omega$.  In (\ref{berg:Vk}) and
(\ref{berg:V2}), we label the upper half-space by
\begin{equation}\label{berg:r3+}
  \R_{+}^{3}\equiv\lbrace{\y=(\txione,\txitwo,\teta)^T \, \vert \, \, \teta>0
    \,, \, -\infty<\txione,\txitwo<\infty \rbrace} \,.
\end{equation}

Upon imposing the leading-order matching condition $V_{0i}\sim U_0$ as
$|\y|\to\infty$, the solution to (\ref{berg:Vk}) for $k=0$ is
\begin{equation}\label{berg:v0sol}
    V_{0i} = U_0 \left( 1 - w_{i} \right) \,,
\end{equation}
where $w_{i}(\y;b_i)$ is the solution to
\bsub \label{berg:wc}
\begin{align}
    \tdelta w_{i} &=0 \,, \quad \y \in \R_{+}^{3} \,, \label{berg:wc_1}\\
    -\partial_{\teta} w_{i} + b_i w_{i} &=b_i \,, \quad \teta=0 \,,\,
    (\txione,\txitwo)\in \PT_i\,,  \label{berg:wc_2}\\
  \partial_{\teta} w_{i} &=0 \,, \quad \teta=0 \,,\, (\txione,
                           \txitwo)\notin \PT_i
    \,, \label{berg:wc_3}\\
  w_{i}&\sim \frac{C_{i}(b_i)}{|\y|} +
                {  \frac{\DT_i(b_i) {\bf \cdot} \y}{|\y|^3}}
         + \cdots\,,  \quad \mbox{as}\quad
    |\y|\to \infty \,, \label{berg:wc_4}
\end{align}
\esub 
where $|\y|=(\txione^2+\txitwo^2+\teta^2)^{1/2}$. In
(\ref{berg:wc_4}), $C_i(b_i)$ is referred to as the {\em reactive
capacitance} of $\PT_i$
(cf.~\cite{GrebenkovWard2026big},\cite{GrebenkovWard2026}).  From the
divergence theorem, it satisfies the identity
\begin{equation}\label{berg:wc_charge}
  C_i(b_i) = \frac{1}{\pi} \int_{\PT_i} q_{i}(\txione,\txitwo; b_i) \,
  d\txione d\txitwo  \,, \quad \mbox{where} \quad
  q_i(\txione,\txitwo;b_i)\equiv -\frac{1}{2}
  \partial_{\teta} w_{i}\vert_{\teta=0}\,.
\end{equation}
We refer to $q_i$ as the {\em charge density} in analogy with
electrostatics.  We remark that $C_i(b_i)$ is invariant under rotation
of a given patch shape $\PT_i$.  As a result, in computing $C_i(b_i)$
it is not necessary to align the coordinate axes of $\PT_i$ along the
two principal directions.  The dipole vector $\DT_i=\DT_i(b_i)$ in
(\ref{berg:wc_4}) must have the form $\DT_i=(p_{1i},p_{2i},0)^T$ to
ensure that the far-field behavior (\ref{berg:wc_4}) satisfies
(\ref{berg:wc_3}).

In \cite{GrebenkovWard2026big} (see also \cite{GrebenkovWard2026}),
the reactive capacitance for a circular patch was computed numerically
in terms of a Steklov eigenfunction expansion.  The form of this
expansion was the motivation for constructing a heuristic
approximation for $C_i$, which is rather accurate over the full range
$0<b_i<\infty$ (see \cite{GrebenkovWard2026big} and \S
\ref{sec:numerics_consts}).  Moreover, this analysis of
\cite{GrebenkovWard2026big} and \cite{GrebenkovWard2026} was extended
in \cite{Grebenkov-Maurette} to numerically compute $C_i(b_i)$ for
patches of arbitrary shape and to derive monotonicity principles for
$C_i$. For circular patches, the following lemma was established in
\cite{GrebenkovWard2026} (see also
\cite{GrebenkovWard2026big}):

\begin{lemma}\label{lemma:Cj_kappa} (Lemma 3.1 of \cite{GrebenkovWard2026big}) 
When $\PT_i$ is the disk $0\leq s\leq a_i$, where $s^2\equiv
\txione^2+\txitwo^{2}$, we have $w_{i}=w_{i}(s,\teta;b_i)$ and
\begin{equation}\label{berg:Cj}
  C_i(b_i) = 2 \int_{0}^{a_i} q_i(s;b_i) s \, ds \,, \qquad
  q_i(s;b_i)=-\frac{1}{2} w_{i, \teta}\vert_{\teta=0} \,,
\end{equation}
where $w_i$ is the solution to (\ref{berg:wc}).  The asymptotics of
$C_i$ are
\bsub\label{berg:Cj_asy}
\begin{align}
      C_i(b_i) &\sim C_i(\infty) + {\mathcal O}\left(
  \frac{\log b_i}{b_i}\right) \,,\quad \mbox{as}\quad
                    b_i\to\infty\,, \quad \mbox{with} \quad
  C_i(\infty)=\frac{2a_i}{\pi} \,, \label{berg:Cj_large}  \\  
  C_i(b_i) &  \sim a_i \biggl[c_{1} (b_i a_i) - c_{2} (b_i a_i)^2 
     + c_{3} (b_i a_i)^3 + {\mathcal O}((b_i a_i)^4)\biggr] \,,
     \quad \mbox{as}  \quad b_i\to 0 \,,  \label{berg:Cj_small}
\end{align}
\esub 
where $c_{1}=0.5$, $c_{2} ={4/(3\pi)} \approx 0.4244$ and
$c_{3} \approx 0.3651$. Moreover, for $b_i\to \infty$,
$q_i\sim q_{i}(s;\infty)=\pi^{-1}\left(a_i^2-s^2\right)^{-1/2}$ on
$0\leq s<a_i$.
\end{lemma}
  
We now proceed with the asymptotic analysis.  The matching condition
is that the local behavior of the outer expansion (\ref{berg:outex})
as $\x\to\x_i$ must agree with the far-field behavior of the inner
expansion (\ref{berg:innex}), so that for each $i=1,\ldots,N$ we must have
\begin{equation} \label{berg:mat_1}
  \begin{split}
  \frac{U_0}{\eps} + U_1 + \eps \log(\eps)  \, U_2 + &\eps U_3 + \ldots \\
   \sim \frac{U_0}{\eps} & \left( 1 - \frac{C_i}{|\y|} - \frac{\DT_i {\bf \cdot}
       \y}{|\y|^3} \right)  + 
  \log(\eps) V_{1i} + V_{2i} + \ldots \,.
  \end{split}
\end{equation}
By using  $|\y|\sim\eps^{-1}|\x-\x_i|$ from (\ref{app_g:loc}) of
Appendix \ref{app_b:bernoff}, this matching condition provides the
singular behavior for $U_1$ as $\x\to\x_i$. In this way, we conclude
that $U_1$ satisfies
\bsub \label{berg:U1prob}
\begin{gather}
  \Delta_{\x} U_{1}= 0 \,, \quad \x\in \R^3\backslash \Omega \,; \qquad
  \partial_n U_1=0 \,, \quad \x\in \partial\Omega\backslash
  \lbrace{\x_1,\ldots,\x_N\rbrace} \,, \\
  U_1\sim -\frac{U_0 C_{i}}{|\x-\x_i|}\,,  \quad \mbox{as} \quad
  \x\to\x_i \in \partial\Omega \,, \quad i=1,\ldots,N \,,\\
  \lim_{\sigma\to\infty} \int_{\partial\Omega_\sigma} \partial_{r}U_{1}\vert_{r=\sigma}
  \, dS=-4\pi \,. 
\end{gather}
\esub
From the divergence theorem, the solvability condition for
(\ref{berg:U1prob}) determines $U_0$ as
\begin{equation} 
  U_0 = -\frac{2}{\overline{C}} \,, \qquad \mbox{where} \qquad
\overline{C}\equiv\sum_{i=1}^{N} C_i(b_i) \,. \label{berg:U0sol}
\end{equation}

To determine $U_1$, we introduce the surface Neumann Green's function
$G_e(\x;\x_i)$ for the exterior problem, which is the unique solution to
\bsub\label{berg:green_ext_full}
\begin{equation}\label{berg:green_int}
  \begin{split}
 \Delta_{\x} G_e &= 0 \,, \quad  \x\in \R^{3}\backslash\Omega \,; \qquad
 \partial_{n} G_e =  \delta(\x-\x_i)\,, \quad \x\in \partial\Omega\,, \\
 G_e &\sim \frac{1}{4\pi|\x|} \,, \quad \mbox{as} \quad |\x|\to \infty\,,
  \end{split}
\end{equation}
where $\x_i\in\partial\Omega$ and $\partial_{n}$ is the outward normal
derivative to $\R^{3}\backslash\Omega$, which points into $\Omega$.
In Appendix \ref{app_g:int_green} we show that the local singularity
behavior of $G_e$ is
\begin{equation}\label{berg:green_int_sing}
  G_{e}(\x;\x_i)= \frac{1}{2\pi|\x-\x_i|} -\frac{\HC_i}{4\pi}
  \log\left(|\x-\x_i|
    +d \right) + e(\x;\x_i) + R_{e}(\x_i)  + o(1)
  \,, \quad \mbox{as} \quad \x\to \x_i
  \,.
\end{equation}
\esub
Here $\HC_i$ is the mean curvature of $\partial\Omega$ at $\x=\x_i$
(with $\HC_i=-1$ if $\Omega$ is the unit sphere),
$d=\mbox{dist}(\x,\partial\Omega)>0$, while the term $e(\x;\x_i)$
given in (\ref{app_g:e_final}) is proportional to the curvature
difference and depends on the specific path of approach to
$\x=\x_i\in\partial\Omega$.  In contrast, $R_{e}(\x_i)$ is a {\em
globally determined quantity}, referred to as the {\em regular part}
of $G_e$, which depends on the domain shape and on $\x_i$. When
$\Omega$ is the unit sphere, the explicit analytical solution to
(\ref{berg:green_int}) and the corresponding regular part $R_e$ are
given in (\ref{berg:gs_exact}).  For the case of prolate spheroids, a
series solution for {\clb $G_e$} and $R_e$ has been derived in
\cite{lindsay20263D_GFun}, while for more general domains a numerical
method is also available \cite{lindsay20263D_GFun}.

The solution to (\ref{berg:U1prob}) is represented in terms of $G_e$
as
\begin{equation}\label{berg:u1_sol}
  U_1 = \overline{U}_1 - 2\pi U_0 \sum_{j=1}^{N} C_j G_{e}(\x;\x_j) \,, \quad
\end{equation}
where $\overline{U}_1$ is an unknown constant, which must be expanded as
(cf.~\cite{cheviakov2010asymptotic}, \cite{LWB2017},
\cite{GrebenkovWard2026big}) 
\begin{equation}\label{berg:swit}
  \overline{U}_1 = \overline{U}_{10} \log(\eps) + \overline{U}_{11}\,,
\end{equation}
where $\overline{U}_{10}$ and $\overline{U}_{11}$ are constants,
independent of $\eps$, to be determined.  The term
$\overline{U}_{10}\log(\eps)$ is a ``switchback term''
(cf.~\cite{LA}) and simply corresponds to inserting a
constant term between ${U_0/\eps}$ and $U_1$ in the outer expansion
(\ref{berg:outex}).

In order to determine the local behavior of $U_1$ as $\x\to\x_i$, in
Appendix \ref{app_g:int_green} we establish the refined asymptotic
behavior of the surface Neumann Green's function $G_e$ as
$\x\to \x_i$, written in terms of the local tangential-normal
coordinate system.  In terms of the $\eps$-localized boundary-fitted
coordinates it is given in (\ref{app_g:glocal_all}).  In this way, by
using (\ref{app_g:glocal_all}) and (\ref{berg:swit}) in
(\ref{berg:u1_sol}), we obtain as $\x\to\x_i$ that
\begin{equation}\label{berg:u1_loc}
  \begin{split}
    U_{1} &\sim -\frac{U_0C_i}{\eps |\y|} + \left( \frac{\HC_i U_0 C_i}{2}
      + \overline{U}_{10}\right) \log\eps +
    \frac{\HC_i U_0 C_{i}}{2} \left( \log\left(\teta+|\y|\right) -
    \frac{\teta (\txione^2+\txitwo^2)}{|\y|^3} \right) \\
  & \qquad - \frac{\kapdif_i U_0 C_i}{4} \left(
    \frac{2\teta}{|\y|^3} + \frac{1}{\left(\teta+|\y|\right)^2}
  \right) (\txione^2-\txitwo^2) + \overline{U}_{11} + U_0\beta_{ei}\,.
\end{split}
\end{equation}
Here the constant $\beta_{ei}$ is the $i$-th component of a
matrix-vector product defined by
\begin{equation}\label{berg:Bi}
  \beta_{ei} \equiv  -2\pi \left({\mathcal G}_e \vc\right)_{i} \,,
\end{equation}
where $\vc\equiv(C_1,\ldots,C_N)^T$ and ${\mathcal G}_{e}$ is the
Green's matrix defined by
\begin{equation}\label{berg:green_mat}
    {\mathcal G}_e \equiv \left ( 
\begin{array}{cccc}
 R_{e1} & G_{e12} & \cdots & G_{e1N} \\
 G_{e21} & R_{e2} & \cdots   &G_{e2N} \\
 \vdots & \vdots  &\ddots  &\vdots\\ 
 G_{eN1} &\cdots & G_{eN,N-1} & R_{eN}
\end{array}
\right ) \,, \quad R_{ei} \equiv R_{e}(\x_i) \,, \quad G_{eij} \equiv
  G_{e}(\x_i;\x_j) \,.
\end{equation}

Upon substituting (\ref{berg:u1_loc}) into the matching condition
(\ref{berg:mat_1}), we observe that the ${\mathcal O}(\log\eps)$ term
in (\ref{berg:u1_loc}) specifies the far-field limiting behavior
$V_{1i}\sim \overline{U}_{10} + {\HC_i U_0 C_i/2}$ as $|\y| \to \infty$,
where $V_{1i}$ satisfies the inner problem (\ref{berg:Vk}) with
$k=1$. In terms of $w_{i}$ of (\ref{berg:wc}), we conclude that
\begin{equation}
  V_{1i} = \left( \frac{\HC_i U_0 C_i}{2} + \overline{U}_{10} \right)
  \left(1 - w_{i}\right) \,, \label{berg:V1sol}
\end{equation}
which has the far-field behavior
\begin{equation}
  V_{1i} \sim \left( \frac{\HC_i U_0 C_i}{2} + \overline{U}_{10} \right)
  \left(1 - 
      \frac{C_i}{|\y|} +\cdots  \right)\,,
    \quad \mbox{as} \quad |\y| \to \infty\,.
        \label{berg:V1ff}
\end{equation}

Upon substituting (\ref{berg:V1ff}) into the matching condition
(\ref{berg:mat_1}), we obtain the singularity behavior for the
correction $U_2$ in the outer expansion (\ref{berg:outex}). In this
way, we conclude that $U_2$ satisfies
\bsub \label{berg:U2prob}
\begin{gather}
  \Delta_{\x} U_{2} = 0 \,, \quad \x\in \R^3\backslash \Omega \,; \qquad
  \partial_n U_2=0 \,, \quad \x\in \partial\Omega\backslash
  \lbrace{\x_1,\ldots,\x_N\rbrace} \,, \\
  U_2\sim -\left( \frac{\HC_i U_0 C_i}{2} +\overline{U}_{10} \right) \frac{C_i}
  {|\x-\x_i|}\,, \quad \mbox{as} \quad
  \x\to\x_i \in \partial\Omega \,, \quad i=1,\ldots,N \,, \\
  \lim_{\sigma\to\infty} \int_{\partial\Omega_\sigma} \partial_{r}U_{2}\vert_{r=\sigma}
  \, dS=0  \,. 
\end{gather}
\esub

By using the divergence theorem, we readily find that (\ref{berg:U2prob}) is
solvable only when $\sum_{j=1}^{N}C_j\left[\overline{U}_{10}+{\HC_j U_0 C_j/2}
\right]=0$, which determines $\overline{U}_{10}$ as
\begin{equation}
  \frac{\overline{U}_{10}}{U_0} = -\frac{1}{2\overline{C}} \left(
    \sum_{j=1}^{N} \HC_j C_j^2 \right)\,, \quad \mbox{where} \quad
   \overline{C}=\sum_{j=1}^{N} C_j \,. \label{berg:u10}
\end{equation}
With $\overline{U}_{10}$ determined in this way, the solution to
(\ref{berg:U2prob}) is given in terms of the surface Neumann
Green's function and an additional unknown constant
$\overline{U}_{2}$ as
\begin{equation}
  U_2 = \overline{U}_{2} - 2\pi \sum_{j=1}^{N} C_j \left( \frac{\HC_j U_0 C_j}{2}
    + \overline{U}_{10} \right)  G_{e}(\x;\x_j)  \,.
  \label{berg:U2}
\end{equation}

Next, we match the ${\mathcal O}(1)$ terms in (\ref{berg:u1_loc})
when substituted into the matching condition (\ref{berg:mat_1}). In this
way, we find that $V_{2i}$ satisfies (\ref{berg:V2}) with the
far-field behavior
\begin{equation}\label{berg:V2inf}
  \begin{split}
    V_{2i} \sim V_{2i\infty} &\equiv \beta_{ei} U_0 + \overline{U}_{11} +
    \frac{\HC_i U_0 C_i}{2} \left(
    \log(\teta + |\y|) - \frac{\teta (\txione^2 + \txitwo^2)}{|\y|^3}
  \right)\,, \\
  & \qquad \qquad - \frac{\kapdif_i U_0 C_i}{4} \left[
    \frac{2\teta}{|\y|^3} + \frac{1}{\left(\teta+|\y|\right)^2}
  \right] (\txione^2-\txitwo^2) \,, \quad \mbox{as} \quad |\y| \to \infty \,. 
\end{split}
\end{equation}
To analyze (\ref{berg:V2inf}), it is convenient to decompose $V_{2i}$ as
\begin{equation}\label{berg:V2_decomp}
  V_{2i} = \left( \beta_{ei} U_0 + \overline{U}_{11}\right) \left(1-w_i
  \right) + \HC_i U_0 \Phi_{i+} + \kapdif_{i} U_0 \Phi_{i-} \,,
\end{equation}
while substituting $V_{0i}=U_0(1-w_i)$ into the right-hand side of
(\ref{berg:V2}). In this way, we obtain that $\Phi_{i+}$ satisfies
\bsub \label{berg:Phi+}
\begin{align}
  \tdelta \Phi_{i+} &= -2 \left( \teta w_{i,\teta\teta} + w_{i,\teta} 
                       \right) \,, \quad
   \y \in \R_{+}^{3} \,, \label{berg:Phi+_1} \\
   -\partial_{\teta} \Phi_{i+} + b_i \Phi_{i+} &=0 \,, \quad \teta=0 \,,\,
    (\txione,\txitwo)\in \PT_i\,,  \label{berg:Phi+_2}\\
  \partial_{\teta} \Phi_{i+} &=0 \,, \quad \teta=0 \,,\, (\txione,\txitwo)
                               \notin \PT_i
                               \,, \label{berg:Phi+_3}\\
  \Phi_{i+} &\sim  \frac{C_i}{2} \left(
              \log(\teta + |\y|) - \frac{\teta (\txione^2 + \txitwo^2)}{|\y|^3}
              \right)\,, \quad
   \mbox{as} \quad |\y| \to \infty \,, \label{berg:Phi+_4} 
\end{align}
\esub
while $\Phi_{i-}$ is the solution to
\bsub \label{berg:Phi-}
\begin{align}
  \tdelta \Phi_{i-} &= {\mathcal N}_{-}(\y) \equiv
                      2 \teta \left( w_{i,\txione\txione} -w_{i,\txitwo\txitwo}
                       \right) \,, \quad
   \y \in \R_{+}^{3} \,, \label{berg:Phi-_1}\\
   -\partial_{\teta} \Phi_{i-} + b_i \Phi_{i-} &=0 \,, \quad \teta=0 \,,\,
    (\txione,\txitwo)\in \PT_i\,,  \label{berg:Phi-_2}\\
  \partial_{\teta} \Phi_{i-} &=0 \,, \quad \teta=0 \,,\, (\txione,\txitwo)
                               \notin \PT_i
                               \,, \label{berg:Phi-_3}\\
  \Phi_{i-} &\sim  -\frac{C_i}{4} \left[
              \frac{2\teta}{|\y|^3} + \frac{1}{\left(\teta+|\y|\right)^2}
  \right] (\txione^2-\txitwo^2) \,, \quad
   \mbox{as} \quad |\y| \to \infty \,. \label{berg:Phi-_4} 
\end{align}
\esub

In Appendix C of \cite{GrebenkovWard2026big} (see also
\cite{GrebenkovWard2026}), with derivation summarized
below in Appendix \ref{app:mono-plus}, it was shown that the solution
to (\ref{berg:Phi+}) has the refined far-field behavior
\begin{equation}\label{berg:Phi+_ff}
  \Phi_{i+} \sim \frac{C_i}{2} \left(
    \log(\teta + |\y|) - \frac{\teta (\txione^2 + \txitwo^2)}{|\y|^3}
  \right) + \frac{E_{i+}}{|\y|}\,, \quad \mbox{as} \quad |\y| \to \infty \,.
\end{equation}
The properties of the monopole coefficient $E_{i+}=E_{i+}(b_i)$, as
derived in \cite{GrebenkovWard2026big} and also summarized in Appendix
\ref{app:mono-plus} (see also \cite{GrebenkovWard2026}), are given in
the next lemma.

\begin{lemma}\label{lemma:monop} For an arbitrary patch shape $\PT_i$,
  we have
\begin{equation}  \label{berg:Ei_general0}
  E_{i+}(b_i) = - \frac{1}{2\pi^2} \int\limits_{\PT_i}\int\limits_{\PT_i}
  q_i(\y;b_i)\, q_i(\y^{\prime};b_i)
\, \log|\y-\y^{\prime}|\, d\y \,d\y^{\prime} \,,
\end{equation}
where we have labeled $\y=(\txione,\txitwo)^T$ and $\y^{\prime}=
(\txione^{\prime},\txitwo^{\prime})^T$. As $b_i\to 0$, we have 
$q_{i}\sim {b_i/2}$ so that to leading order
\begin{equation}  \label{berg:Ei_asympt0}
  E_{i+}(b_i) \sim -\frac{b_i^2}{8\pi^2} \left(\,
  \int\limits_{\PT_i}\int\limits_{\PT_i} \log|\y-\y^{\prime}|\, d\y \,
  d\y^{\prime} \right) \,.
\end{equation}
The coefficient $E_{i+}$ is invariant under rotations and so does not
depend on the orientation of the patch with respect to the principal
directions. When $\PT_i$ is the disk
$s^{2}\equiv \txione^2+\txitwo^{2}\leq a_i^2$, for which
$q_i=q_i(s;b_i)$ in (\ref{berg:wc_charge}) is radially symmetric, we
have
\begin{equation}\label{berg:Ej_all}  
  E_{i+} =  E_{i+}(b_i) = -\frac{\log{a_i}}{2} [C_i(b_i)]^2 + 2
  \int_{0}^{a_i} \frac{1}{s}
\left(\int_{0}^{s}  t q_i(t;b_i) \, dt\right)^2  \,  ds\,,
\end{equation}
where $C_i(b_i)$ is given by Lemma \ref{lemma:Cj_kappa}. For the disk,
the asymptotic behavior of $E_{i+}$ is 
\bsub \label{berg:Ej_asy}
  \begin{align}
  E_{i+} &\sim E_{i+}(\infty)\equiv -\frac{2 a_i^2}{\pi^2}\left(
    \log{a_i} + \log{4} - \frac{3}{2} \right)\,,   \quad \mbox{as}\quad
  b_i\to \infty\,, \label{berg:Ej_asy_large} \\
  E_{i+} & \sim \frac{b_i^2 a_i^4}{8} \left( \frac{1}{4}-\log{a_i}\right)\,,
  \quad \mbox{as} \quad b_i \to 0 \,.  \label{berg:Ej_asy_small}
\end{align}
\esub
\end{lemma}

The new problem (\ref{berg:Phi-}), which incorporates curvature anisotropy
due to differences in the two principal curvatures for
non-circular patches, is analyzed in Appendix
\ref{app:mono-minus}. There it is shown that, in terms of an
additional monopole coefficient $E_{i-}=E_{i-}(b_i)$, $\Phi_{i-}$ has
the refined far-field asymptotic behavior
\begin{equation}\label{berg:Phi-_ff}
  \Phi_{i-} \sim  -\frac{C_i}{4} \left[
              \frac{2\teta}{|\y|^3} + \frac{1}{\left(\teta+|\y|\right)^2}
  \right] (\txione^2-\txitwo^2) + \frac{E_{i-}}{|\y|} \,.
\end{equation}
As derived in Appendix \ref{app:mono-minus}, this new monopole coefficient
is characterized as follows:

\begin{lemma}\label{lemma:monom} For an arbitrary patch shape $\PT_i$,
  we have
\bsub \label{berg:Ei_-}
\begin{equation}  \label{berg:Ei-main}
  E_{i-}(b_i) =  \frac{1}{4\pi^2} \int_{\PT_i} q_{i}(\y;b_i)
  {\mathcal J}_{i}(\y;b_i) \, d\y\,, \quad \mbox{with} \quad
  q_i(\y;b_i)=-\frac{1}{2} w_{i, \teta}\vert_{\teta=0} \,,
\end{equation}
where ${\mathcal J}_{i}$ is defined by
\begin{equation}\label{berg:Ei-main_J}
  {\mathcal J}_{i}(\y;b_i) \equiv \int_{\PT_i}
  \frac{\left(\txione^{\prime}-\txione\right)^2 -
      \left(\txitwo^{\prime}-\txitwo\right)^2}{|\y^{\prime}-\y|} \,
    q_{i}(\y^{\prime};b_i) \, d\y^{\prime} \,.
  \end{equation}
\esub
  As $b_i\to 0$, we have $q_{i}\sim {b_i/2}$ so that to leading order
\begin{equation}  \label{berg:Ei_minus_asympt0}
  E_{i-}(b_i) \sim \frac{b_i^2}{16\pi^2} 
    \int\limits_{\PT_i}\int\limits_{\PT_i}
    \frac{  \left(\txione^{\prime}-\txione\right)^2 -
      \left(\txitwo^{\prime}-\txitwo\right)^2}{|\y^{\prime}-\y|} 
  d\y \, d\y^{\prime}  \,.
\end{equation}
In (\ref{berg:Ei_-}) and (\ref{berg:Ei_minus_asympt0}), we have
labeled $\y=(\txione,\txitwo)^T$ and
$\y^{\prime}=(\txione^{\prime},\txitwo^{\prime})^T$.
The coefficient $E_{i-}$ is not invariant under rotations of $\PT_i$, and
so it depends on the orientation of the patch with respect to the
two principal directions. When $\PT_i$ is a
disk, then $E_{i-}\equiv 0$ for all $b_i>0$.
\end{lemma}

To show that $E_{i-}\equiv 0$ when $\PT_i$ is a disk of radius $a_i$,
we first observe for a disk that $q_i$ is radially symmetric so that
$q_i=q_i\left(|\y^{\prime}|;b_i\right)$. It then readily follows that
${\mathcal J}_{i}$ in (\ref{berg:Ei-main_J}) has the symmetry property
(written in component form)
${\mathcal J}_{i}(\txione,\txitwo;b_i)= -{\mathcal
  J}_{i}(\txitwo,\txione;b_i)$. As a result, since the integrand in
(\ref{berg:Ei-main}) is an odd function with respect to reflection
about the line $\xi_1=\xi_2$ that partitions the disk $\PT_i$ into
equal halves, it follows from symmetry that $E_{i-}\equiv 0$ for all
$b_i>0$.

With the properties of $E_{i\pm}$ established above, we now proceed
with the asymptotic analysis. To determine the refined far-field
behavior of $V_{2i}$ we substitute the far-field behaviors
(\ref{berg:Phi+_ff}), (\ref{berg:Phi-_ff}), and $w_{i}\sim {C_i/|\y|}$
as $|\y|\to \infty$ into our decomposition (\ref{berg:V2_decomp}). We
conclude that
\begin{equation}\label{berg:v2i_refined}
  V_{2i} \sim V_{2i\infty} + \frac{E_{i} U_0}{|\y|} -
  \frac{C_i \left(U_0 \beta_{ei} + \overline{U}_{11}\right)}{|\y|} \,, \quad
  \quad \mbox{as} \quad |\y|\to \infty \,,
\end{equation}
where $V_{2i\infty}$ was defined in (\ref{berg:V2inf}) and $E_{i}$ is defined
in terms of $E_{i\pm}$ by
\begin{equation}\label{berg:Ei_all}
  E_{i} \equiv \HC_i E_{i+} + \kapdif_i E_{i-} \,,
\end{equation}
where $\HC_i$ and $\kapdif_i$ were defined in (\ref{berg:curve}).

Upon substituting (\ref{berg:v2i_refined}) into the matching condition
(\ref{berg:mat_1}), we observe that the two monopole terms in
(\ref{berg:v2i_refined}) determine one component of the required
singularity behavior for the outer correction $U_3$ in
(\ref{berg:outex}). The remaining part of the singularity condition
arises from the dipole term $\DT_i$ in (\ref{berg:wc_4}), which is
written in outer variables using
$\y\sim {{\orthomat}_i (\x-\x_i)/\eps}$ from (\ref{app_b:vab_all}),
where ${\orthomat}_i$ is a suitable orthogonal matrix. In this way,
we conclude from (\ref{berg:Uk}), (\ref{berg:mat_1}) and
(\ref{berg:v2i_refined}) that $U_3$ must satisfy \bsub
\label{berg:U3prob}
\begin{align}
  \Delta_{\x} U_{3} &= 0 \,, \quad \x\in \R^3\backslash \Omega \,; \qquad
             \partial_n U_3=0 \,, \quad \x\in \partial\Omega\backslash
                      \lbrace{\x_1,\ldots,\x_N\rbrace} \,, \\
  U_3 &\sim \frac{\left[U_0 E_i -\left(\beta_{ei} U_0 +\overline{U}_{11}\right)C_i
        \right]}{|\x-\x_i|} -U_0 \frac{\DT_i {\bf \cdot}
        {\orthomat}_i (\x-\x_i)}{|\x-\x_i|^3}\,, \nonumber \\
&\qquad \qquad \mbox{as} \quad \x\to\x_i \in \partial\Omega \,, \quad
                                           i=1,\ldots,N \,,\\
 &  \lim_{\sigma\to\infty} \int_{\partial\Omega_\sigma} \partial_{r}U_{3}\vert_{r=\sigma}
        \, dS=0  \,.                                            
\end{align}
\esub 

From the divergence theorem, we conclude that (\ref{berg:U3prob}) has a
solution if and only if $\overline{U}_{11}$ satisfies
\begin{equation}
  \frac{\overline{U}_{11}}{U_0} = \frac{1}{\overline{C}}
  \left(\sum_{j=1}^{N} E_j -
    \sum_{j=1}^{N} \beta_{ej} C_j\right) \,. \label{berg:u11_1}
\end{equation}
We remark that the contribution from the dipole term vanishes
identically by symmetry since $\DT_i$ has the form
$\DT_i=(p_{1i},p_{2i},0)^T$.  Finally, by using (\ref{berg:Bi}) for
$\beta_{ei}$, and recalling (\ref{berg:Ei_all}) for $E_i$, we determine
$\overline{U}_{11}$ as
\begin{equation}\label{berg:u11}
  \frac{\overline{U}_{11}}{U_0} = \frac{2\pi}{\overline{C}} \vc^T
  {\mathcal G}_e \vc + \frac{\overline{E}}{\overline{C}} \,, \quad
  \mbox{where} \quad  \overline{E}\equiv \sum_{j=1}^{N} E_j=\sum_{j=1}^{N}
  \left( \HC_j E_{j+} + \kapdif_j E_{j-}\right) \,.
\end{equation}

Finally, to identify the capacitance $C_{\rm T}$ in
(\ref{berg:ssp_3}), we must take the limit as $|\x|\to\infty$ of our
outer asymptotic expansion (\ref{berg:outex}) and compare it with the
required limiting behavior in (\ref{berg:ssp_3}). By comparing the
${\mathcal O}(1)$ terms in the resulting expression we obtain the
following main result for $C_{\rm T}$:

\begin{prop}\label{berg:main_res} 
  As $\eps \to 0$, the dimensionless capacitance $C_{\rm T}$ for
  (\ref{berg:ssp}) in the presence of $N$ well-separated partially
  reactive Robin patches, centered at $\x_i$ and
  with local reactivities $b_i$ for $i=1,\ldots,N$, has the three-term
  asymptotic expansion
\bsub \label{berg:main_res_1}
\begin{equation}
  \frac{1}{C_{\rm T}} \sim \frac{|U_0|}{\eps} \left[ 1 +
    \eps \frac{\overline{U}_{10}}{U_0} \log\left(\eps\right)
    + \eps \frac{\overline{U}_{11}}{U_0} + {\mathcal O}\left(\eps^2
    \log\eps\right) \right]\,,
\end{equation}
where
\begin{equation}\label{berg:main_res_c}
  \begin{split}
 & |U_0| = \frac{2}{\overline{C}} \,, \qquad
 \frac{\overline{U}_{10}}{U_0} = -\frac{1}{2\overline{C}} \sum_{j=1}^{N}
 \HC_j C_j^2    \,, \\
 &  \frac{\overline{U}_{11}}{U_0} =
 \frac{2\pi}{\overline{C}} \vc^T {\mathcal G}_e
 \vc + \frac{\overline{E}}{\overline{C}}\,, \qquad
 \overline{E}=\sum_{j=1}^{N}
 \left( \HC_j E_{j+} + \kapdif_j E_{j-}\right) \,,
 \end{split}
\end{equation}
\esub 
with $\HC_j={(\kappa_{1j}+\kappa_{2j})/2}$,
$\kapdif_j={(\kappa_{1j}-\kappa_{2j})/2}$, and
$\overline{C}=\sum_{j=1}^{N} C_j$.  Here $E_{j+}$ and $E_{j-}$ are
characterized in Lemmas \ref{lemma:monop} and \ref{lemma:monom},
respectively. The Green's matrix ${\mathcal G}_e$ in
(\ref{berg:main_res_c}) is given by (\ref{berg:green_mat}) in terms of
the surface Neumann Green's function of (\ref{berg:green_int}).  In
terms of the dimensional capacitance ${\mathcal C}_{\rm T}$, defined
in (\ref{berg_intro:ssp0_c}), we use (\ref{intro_bp:scalings}) to
obtain for a domain of diameter $2R_{\star}$ and with a collection of
partially reactive boundary patches of maximum diameter $2L$, that
\begin{equation}\label{berg:dimen}
  {\mathcal C}_{\rm T} = R_{\star} C_{\rm T} \,,
\end{equation}
which determines the dimensional flux to all the surface receptors as
(see (\ref{eq:j_origin}))
\begin{equation}\label{berg:flux_all}
  J\equiv   4\pi D \, {\mathcal U}_{\rm \inf}{\mathcal C}_{\rm T} \,.
\end{equation}
In evaluating $C_{\rm T}$ from (\ref{berg:main_res_1}) for the
dimensional problem, we use $C_i=C_i\left({L\B_i/D}\right)$ and
$E_{i\pm}=E_{i\pm}\left({L\B_i/D}\right)$, while calculating the
Green's matrix ${\mathcal G}_e$ at $\x_i={\X_i/R_{\star}}$.
\end{prop}

When $\Omega$ is the unit sphere, we have $\HC_i=-1$ and
$\kapdif_{i}=0$ for all $i=1,\ldots,N$, and so our main result
(\ref{berg:main_res_1}) reduces to that derived in Proposition 1 of
\cite{GrebenkovWard2026}.

For $N$ identical locally circular, and perfectly absorbing, patches
of radius $\eps$ on a smooth boundary, our main result
(\ref{berg:main_res_1}) can be simplified to that given in
(\ref{eq:intro_CT}) by setting $C_i={2/\pi}$, $E_{i-}=0$, and
$E_{i+}=-2{\left(\log{4}-{3/2}\right)/\pi^2}$ for $i=1,\ldots,N$.

We emphasize that the term $\kapdif_{i}E_{i-}$ appearing in
(\ref{berg:main_res_c}) will in general be non-vanishing {\em only}
when the patch has a non-circular shape and is centered at a point on
the boundary where the two principal curvatures are different.

\section{The Mean First-Reaction Time}\label{sec:mfpt}

In this section we use a very similar asymptotic approach to analyze
(\ref{mfpt:ssp}) and to derive a three-term expansion for the global
MFRT $\overline{u}$ in (\ref{mfpt:ubar}).

In the outer region, we expand the outer solution as in
(\ref{berg:outex}) where $U_0$ is a constant to be determined, and
where the correction terms $U_k$ for $k\geq 1$ satisfy
\begin{equation}
  \Delta_{\x} U_k = - \delta_{k1} \,, \quad \x \in \Omega \,; 
  \qquad \partial_n U_k = 0 \,, \quad \x\in \partial\Omega\backslash
\lbrace{\x_1,\ldots,\x_N\rbrace} \,,  \label{mfpt:Uk}
\end{equation}
where $\delta_{k1}$ is the Kronecker symbol. Singularity
behaviors for $U_{k}$ as $\x\to \x_i$, for $i=1,\ldots,N$, will be derived
by asymptotic matching to the inner (local) solutions near each patch.
The outer corrections $U_k$ for $k\geq 1$ are
now represented in terms of the surface Neumann Green's function
$G_{s}(\x;\x_i)$ for the interior problem, which is the unique solution to
\begin{equation}
\Delta_{\x} G_{s} = \frac{1}{|\Omega|} \,, \quad  \x\in \Omega \,; \qquad
\partial_n G_{s} =  \delta(\x-\x_i)\,, \quad \x \in \partial \Omega \,; \qquad
 \int_{\Omega} G_{s} \, d\x  = 0 \,, \label{mfpt:green_int}
\end{equation}
with $\x_i\in\partial\Omega$. As shown in Appendix \ref{app_g:int_green} the
local singularity behavior of $G_s$ is
\begin{equation}\label{mfpt:green_int_sing}
  G_{s}(\x;\x_i)= \frac{1}{2\pi|\x-\x_i|} -\frac{\HC_i}{4\pi}
  \log\left(|\x-\x_i|
    + d \right) + e(\x;\x_i) + R_{s}(\x_i)+ o(1)
  \,, \quad \mbox{as} \quad \x\to \x_i  \,,
\end{equation}
where $d=\mbox{dist}(\x,\partial\Omega)$. For this interior problem we note the sign convention that if the
boundary at $\x=\x_i\in\partial\Omega$ is locally convex as
seen from $\x\in \Omega$, then $\HC_i>0$. When $\Omega$ is the unit
sphere, $\HC_i=1$, and the exact solution to
(\ref{mfpt:green_int}) is given in (\ref{mfpt:gs_exact}).

Since the asymptotic analysis of the inner solution near each patch is
identical to that for the exterior problem studied in \S \ref{sec:berg},
we can readily determine the appropriate singularity behavior for
$U_k$ as $\x\to \x_i$ for (\ref{mfpt:Uk}) by simply appealing to
the results in \S \ref{sec:berg}.

In particular, in analogy with (\ref{berg:U1prob}), we obtain that
$U_{1}$ satisfies \bsub \label{mfpt:U1prob}
\begin{gather}
  \Delta_{\x} U_{1}= -1 \,, \quad \x\in \Omega \,; \qquad
  \partial_n U_1=0 \,, \quad \x\in \partial\Omega\backslash
  \lbrace{\x_1,\ldots,\x_N\rbrace} \,, \\
  U_1\sim -\frac{U_0 C_{i}}{|\x-\x_i|}\,,  \quad \mbox{as} \quad
  \x\to\x_i \in \partial\Omega \,, \quad i=1,\ldots,N \,.
\end{gather}
\esub From the divergence theorem, the solvability condition for
(\ref{mfpt:U1prob}) determines $U_0$ as
\begin{equation} 
  U_0 = \frac{|\Omega|}{2\pi\overline{C}}\,,  \quad \mbox{where} \quad
  \overline{C}\equiv   \sum_{j=1}^{N} C_{j} \,. \label{mfpt:U0sol}
\end{equation}
In terms of constants $\overline{U}_{10}$ and $\overline{U}_{11}$ to
be determined, the solution to (\ref{mfpt:U1prob}) is
\begin{equation}\label{mfpt:u1_sol}
  U_1 = \overline{U}_1 - 2\pi U_0 \sum_{j=1}^{N} C_j G_{s}(\x;\x_j) \,, \quad
  \mbox{where} \quad \overline{U}_1=\overline{U}_{10} \log(\eps) +
  \overline{U}_{11}\,.
\end{equation}

In analogy with (\ref{berg:U2prob}), we obtain that $U_{2}$ satisfies
\bsub \label{mfpt:U2prob}
\begin{gather}
  \Delta_{\x} U_{2} = 0 \,, \quad \x\in \Omega \,; \qquad
  \partial_n U_2=0 \,, \quad \x\in \partial\Omega\backslash
  \lbrace{\x_1,\ldots,\x_N\rbrace} \,, \\
  U_2\sim -\left( \frac{\HC_i U_0 C_i}{2} +\overline{U}_{10} \right) \frac{C_i}
  {|\x-\x_i|}\,, \quad \mbox{as} \quad
  \x\to\x_i \in \partial\Omega \,, \quad i=1,\ldots,N \,.
\end{gather}
\esub By using the divergence theorem, the solvability condition for
(\ref{mfpt:U2prob}) determines $\overline{U}_{10}$ as
\begin{equation}
  \frac{\overline{U}_{10}}{U_0} = -\frac{1}{2\overline{C}}
  \left( \sum_{j=1}^{N} \HC_j C_j^2 \right)\,,  \quad
 \mbox{where}\quad  \overline{C} \equiv \sum_{j=1}^{N} C_j \,. \label{mfpt:u10}
\end{equation}
The solution to (\ref{mfpt:U2prob}) is then represented in terms of
an unknown constant $\overline{U}_{2}$ as
\begin{equation}
  U_2 = \overline{U}_2 -2\pi \sum_{j=1}^{N} C_j \left( \frac{\HC_j U_0 C_j}{2} +
    \overline{U}_{10} \right)  G_{s}(\x;\x_j)  \,.
  \label{mfpt:U2sol}
\end{equation}

Finally, in analogy with (\ref{berg:U3prob}), we obtain that $U_3$
satisfies
\bsub \label{mfpt:U3prob}
\begin{align}
  \Delta_{\x} U_{3} &= 0 \,, \quad \x\in \Omega \,; \qquad
             \partial_n U_3=0 \,, \quad \x\in \partial\Omega\backslash
                      \lbrace{\x_1,\ldots,\x_N\rbrace} \,, \\
  U_3 &\sim \frac{\left[U_0 E_i -\left(\beta_{si} U_0 +\overline{U}_{11}\right)C_i
        \right]}{|\x-\x_i|} -U_0 \frac{\DT_i {\bf \cdot} {\orthomat}_i
        (\x-\x_i)}{|\x-\x_i|^3}\,, \nonumber \\
     & \qquad \qquad \mbox{as} \quad \x\to\x_i \in \partial\Omega \,, \quad
                      i=1,\ldots,N \,, \label{mfpt:U3prob_2} 
\end{align}
\esub where $E_i$ was defined in (\ref{berg:Ei_all}) and
${\orthomat}_i$ is the rotation matrix between Cartesian and local
coordinates (see (\ref{app_b:coord_1})). In
(\ref{mfpt:U3prob_2}), $\beta_{si}$ is defined by the matrix-vector
product
\begin{equation}\label{mfpt:Bi}
  \beta_{si} =  -2\pi \left({\mathcal G}_s \vc\right)_{i} \,,
\end{equation}
where $\vc=(C_1,\ldots,C_N)^T$ and 
\begin{equation}\label{mfpt:green_mat}
    {\mathcal G}_s \equiv \left ( 
\begin{array}{cccc}
 R_{s1} & G_{s12} & \cdots & G_{s1N} \\
 G_{s21} & R_{s2} & \cdots   &G_{s2N} \\
 \vdots & \vdots  &\ddots  &\vdots\\ 
 G_{sN1} &\cdots & G_{sN,N-1} & R_{sN}
\end{array}
\right ) \,, \quad R_{si} \equiv R_{s}(\x_i) \,, \quad G_{sij} \equiv
  G_{s}(\x_i;\x_j) \,.
\end{equation}

From the divergence theorem, (\ref{mfpt:U3prob}) has a solution only
when $\overline{U}_{11}$ satisfies
\begin{equation}
  \frac{\overline{U}_{11}}{U_0} = \frac{1}{\overline{C}}
  \left(\sum_{j=1}^{N} E_j -
    \sum_{j=1}^{N} \beta_{sj} C_j\right) \,. \label{mfpt:u11_1}
\end{equation}
By using (\ref{mfpt:Bi}) for $\beta_{si}$, and recalling
(\ref{berg:Ei_all}) for $E_i$, we conclude that
\begin{equation}\label{mfpt:u11}
  \frac{\overline{U}_{11}}{U_0} = \frac{2\pi}{\overline{C}} \vc^T
  {\mathcal G}_s \vc + \frac{\overline{E}}{\overline{C}} \,, \quad
  \mbox{where} \quad  \overline{E}\equiv \sum_{j=1}^{N} E_j=\sum_{j=1}^{N}
  \left( \HC_j E_{j+} + \kapdif_j E_{j-}\right) \,.
\end{equation}

We summarize our main result for the dimensionless MFRT $u(\x)$ and
the {\clb global} MFRT $\overline{u}$ in the small-patch limit in the
following formal proposition.  We also provide the corresponding
dimensional result based on the scalings (\ref{intro:scalings}).

\begin{prop}\label{mfpt_b:main_res} 
  As $\eps \to 0$, the asymptotic solution to (\ref{mfpt:ssp}) is
  given in the outer region $|\x-\x_i|\gg {\mathcal O}(\eps)$
  for $i=1,\ldots,N$ by
\begin{equation}\label{mfpt:main_res_1}
  \begin{split}
    u(\x) &\sim \frac{U_0}{\eps} \left[ 1 +
      \eps\log(\eps) \frac{\overline{U}_{10}}{U_0} +
      \eps\left(\frac{\overline{U}_{11}}{U_0} -2\pi \sum_{j=1}^{N}
        C_j G_{s}(\x;\x_j)\right)
    \right.\\
    & \qquad \left. + \eps^2\log(\eps)\left(\frac{\overline{U}_2}{U_{0}} -2\pi
        \sum_{j=1}^{N} C_j \left(\frac{\HC_j C_j}{2}+
          \frac{\overline{U}_{10}}{U_0}\right)
        G_{s}(\x;\x_j) \right) +
      {\mathcal O}(\eps^2) \right] \,,
  \end{split}
\end{equation}
up to an unknown constant $\overline{U}_2$.  The global  MFRT
$\overline{u}$, defined by (\ref{mfpt:ubar}), satisfies
\begin{equation}\label{mfpt:main_res_2}
  \overline{u}\sim \frac{U_0}{\eps} \left[ 1 +
    \eps\log(\eps)
    \frac{\overline{U}_{10}}{U_0} + \eps \frac{\overline{U}_{11}}{U_0} +
    {\mathcal O}\left(\eps^2\log \eps\right)\right] \,.
\end{equation}
In (\ref{mfpt:main_res_1}) and (\ref{mfpt:main_res_2}), $U_0$,
$\overline{U}_{10}$ and $\overline{U}_{11}$ are determined in terms of
$\vc=(C_1,\ldots,C_N)^T$, $\overline{C}= \sum_{j=1}^{N} C_j$,
$\overline{E}= \sum_{j=1}^{N} E_j$, and the Green's matrix ${\mathcal
G}_s$, defined in (\ref{mfpt:green_mat}), by
\begin{equation}\label{mfpt:main_U}
 U_0=\frac{|\Omega|}{2\pi \overline{C}} \,,\qquad
 \frac{\overline{U}_{10}}{U_0}= -\frac{1}{2\overline{C}}
   \sum_{j=1}^{N} \HC_j C_j^2 \,, \qquad
  \frac{\overline{U}_{11}}{U_0} = \frac{2\pi}{\overline{C}} \vc^T {\mathcal G}_s
  \vc + \frac{\overline{E}}{\overline{C}} \,.
\end{equation}
Here $\HC_i$ is the mean curvature of $\partial\Omega$ at $\x=\x_i$,
and $E_j$ is defined in (\ref{mfpt:u11}).  In terms of the dimensional
variables, we use (\ref{intro:scalings}) to conclude for a domain with
length-scale $R_{\star}$, which has a well-separated collection of partially
reactive Robin patches of length-scale $L$, that
\begin{equation}\label{mfpt_b:dimen}
  \overline{U} \sim \frac{R_{\star}^2}{D} \overline{u}\,.
\end{equation}
Here in calculating $\overline{u}$ in (\ref{mfpt:main_res_2}) we set
$C_i=C_i\left({L\B_i/D}\right)$ and
$E_{i\pm}=E_{i\pm}\left({L\B_i/D}\right)$ in (\ref{mfpt:main_U}),
while evaluating the Green's matrix ${\mathcal G}_s$ at
$\x_i={\X_i/R_{\star}}$.
\end{prop}

{\clb Although determining the constant $\overline{U}_2$ in
(\ref{mfpt:main_res_1}) from a higher-order analysis is intractable
analytically, our analysis does provide the spatial dependence
of the MFRT, up to this unknown constant.}

When $\Omega$ is the unit sphere, for which $\HC_i=1$ and
$\kapdif_i=0$ for $i=1,\ldots,N$, (\ref{mfpt:main_res_2}) reduces to
the result derived in Proposition 1 of \cite{GrebenkovWard2026big}.
This previous result of \cite{GrebenkovWard2026big} extended that
derived in
\cite{cheviakov2010asymptotic} for the case of perfectly reactive
circular patches on the boundary of the sphere.

For $N$ identical locally circular patches of a common
radius $\eps$ on a smooth boundary that are all perfectly absorbing,
(\ref{mfpt:main_res_1}) reduces to the result given in
(\ref{eqn:MFPT_reduced}) upon setting $C_i={2/\pi}$, $E_{i-}=0$ and
$E_{i+}=-2{\left(\log{4}-{3/2}\right)/\pi^2}$ for $i=1,\ldots,N$.
In particular, for one such circular patch centered at $\x_1$ on
the boundary, (\ref{eqn:MFPT_reduced}) can be written as
\begin{equation}\label{mfpt_b:one_circular}
  \overline{u} \sim \frac{|\Omega|}{4\eps}\left[ 1 -
    \frac{\eps \HC_1}{\pi} \log(4\eps) + \eps \left(4 R_{s}(\x_1) +
       \frac{3\HC_1}{2\pi}\right) + {\mathcal O}(\eps^2\log\eps)\right]\,.
\end{equation}
This result for the {\clb global} MFRT is equivalent to that derived
in Theorem 1.2 of \cite{NURSULTANOV2021202}.  For one circular
perfectly absorbing patch of radius $\eps$ on the boundary of the unit
sphere we substitute the known values $\HC_1=1$,
$R_{s}(\x_1)=\frac{1}{4\pi}\log2-\frac{9}{20\pi}$ (see
(\ref{mfpt:gs_loc1}) of Appendix \ref{app:explicit}) for any $\x_1
\in\partial\Omega$.  With $|\Omega|={4\pi/3}$ this reduces
(\ref{mfpt_b:one_circular}) to
\begin{equation}
  \overline{u} = \frac{|\Omega|}{4\eps}
  \left( 1  - \frac{\eps}{\pi} \log(2\eps)  -\frac{3\eps}{10\pi}\right)\,,
\end{equation}
which agrees with the single patch result previously derived in
Eq. (2.45) of \cite{cheviakov2010asymptotic}. {\clb We remark that the
leading-order term of this expansion was discovered by Lord Rayleigh
\cite{Rayleigh}, whereas the logarithmic correction was first given in
\cite{Holcman2006a}, but with an incorrect numerical prefactor.}

Another limiting case of our main result corresponds to a perfectly
absorbing elliptical-shaped patch $\PT_1$ of semi-axes $\eps a_{1}$
and $\eps a_{2}$ that are aligned with the directions of principal
curvatures.  By adapting the approach in \cite{Strieder09}, for this
special case in Lemma \ref{lemma:ellipse} of Appendix \ref{app:ellipe}
we give explicit analytical formulas for the reactive capacitance
$C_1=C_1(\infty)$, the charge density $q_1(\y;\infty)$, and the
monopole coefficients $E_{1+}=E_{1+}(\infty)$ and
$E_{1-}=E_{1-}(\infty)$.  For this special case,
(\ref{mfpt:main_res_2}) can be written as
\bsub \label{mfpt_b:one_ellipse}
\begin{equation}\label{mfpt_b:one_ellipse_a}  
  \overline{u} \sim \frac{|\Omega|}{4\pi} \biggl[\frac{2}{C_1 \eps} -
  \HC_1 \log\eps + 4\pi R_{s}(\x_1) +
  2 \left( \HC_{1} \frac{E_{1+}}{C_{1}^2} +
    \kapdif_{1} \frac{E_{1-}}{C_{1}^2} \right) +
  {\mathcal O}(\eps\log\eps)\biggr]\,,
%  \overline{u} \sim \frac{|\Omega|}{2\pi C_1 \eps} -
%  \frac{|\Omega| \HC_1}{4 \pi} \log\eps + |\Omega| R_{s}(\x_1) +
%  \frac{|\Omega|}{2\pi} \left( \HC_{1} \frac{E_{1+}}{C_{1}^2} +
%    \kapdif_{1} \frac{E_{1-}}{C_{1}^2} \right) +
%  {\mathcal O}(\eps\log\eps)\,,
\end{equation}
where $C_1$, ${E_{1+}/C_1^2}$ and ${E_{1-}/C_1^2}$ are given in
(\ref{app:c1}), (\ref{app:e+}) and (\ref{app:e-}), respectively.  The
resulting expression for the {\clb global} MFRT for one perfectly
absorbing elliptical-shaped patch is equivalent to that derived in
Eq. (1.3) of Theorem 1.3 of \cite{NURSULTANOV2021202}.  Moreover, in
Lemma \ref{lemma:ellipse} we derive the new result that the ratios
${E_{1\pm}/C_1^2}$ in (\ref{mfpt_b:one_ellipse_a}) can be explicitly
written in terms of the semi-axes of the ellipse as
\begin{equation}\label{mfpt_b:one_ellipse_b}
  \frac{E_{1+}}{C_{1}^2} = -\frac{1}{2}\log\left(\frac{a_1+a_2}{2}\right)
  -\log{2} + \frac{3}{4}  \,, \qquad \frac{E_{1-}}{C_{1}^2} = \frac{a_1-a_2}
  {4(a_1+a_2)} \,.
\end{equation}
\esub

\subsection{Principal Eigenvalue of the Laplacian}
\label{sec:mfrt:eig}

By using the result in (\ref{mfpt:main_res_2}) and (\ref{mfpt:main_U})
for the MFRT $\overline{u}$, we now briefly outline the derivation of
the asymptotic result for the principal eigenvalue $\lambda_0$ of the
Laplacian for
\bsub \label{mfpt_eig:ssp}
\begin{align}
  \Delta_{\x} \phi +\lambda \phi & = 0 \,, \quad \x \in \Omega \,; \qquad
   \int_{\Omega} \phi^2 \, d\x =1 \,,  \label{mfpt_eig:ssp_1}\\
  \eps\partial_{n} \phi + b_i \phi & = 0\,, \quad \x \in
     \partial\Omega^{\eps}_i  \,, \quad i=1,\ldots,N \,, \\
  \partial_{n} \phi & = 0 \,, \quad \x \in \partial \Omega_r
        =\partial\Omega\backslash\cup_{i=1}^{N} \partial\Omega^{\eps}_{i}
                      \,. \label{mfpt_eig:ssp_2}
\end{align}
\esub In {\clb Sec.~4.5} of \cite{GrebenkovWard2026big} a three-term expansion
of $\lambda_0$ was derived for (\ref{mfpt_eig:ssp}) when $\Omega$ is a
sphere, which extended the previous result in
\cite{cheviakov2010asymptotic} that was limited to perfectly reactive
circular patches on the boundary of the sphere.

To derive the corresponding result for $\lambda_0$ for a general
domain, we simply proceed as in {\clb Sec.~4.5} of
\cite{GrebenkovWard2026big} to relate $\lambda_0$ to $\overline{u}$,
through an eigenfunction expansion of $u$ in terms of the eigenpairs
of (\ref{mfpt_eig:ssp}), which yields that
\begin{equation}\label{sec_eig:uave_1}
  \lambda_0 \sim \frac{1+{\mathcal O}(\eps^2)}
  {\overline{u}-{\mathcal O}(\eps^2)}\,.
\end{equation}
Upon substituting the expansion (\ref{mfpt:main_res_2}) for $\overline{u}$
in (\ref{sec_eig:uave_1}) we observe that we can neglect the
${\mathcal O}(\eps^2)$ terms in (\ref{sec_eig:uave_1}), to obtain
\begin{equation}\label{sec_eig:lam0_1}
\lambda_0 \sim \frac{\eps}{U_0} \left( 1 + \eps \log\eps
    \frac{\overline{U}_{10}}{U_0} + \eps \frac{\overline{U}_{11}}{U_0}
     + {\mathcal O}(\eps^2 \log\eps) \right)^{-1} \,.
\end{equation}
Then, by using $(1+y)^{-1}\sim 1-y + {\mathcal O}(y^2)$ for $|y|\ll 1$
together with (\ref{mfpt:main_U}) we obtain that the principal
eigenvalue $\lambda_0$ of (\ref{mfpt_eig:ssp}) has the three-term
asymptotics
\begin{equation}\label{sec_eig:end}  
  \lambda_0 |\Omega| \sim 2\pi \eps \overline{C} +
  \pi \eps^2\log\eps  \sum_{j=1}^{N} \HC_j C_j^2 \\
   - 2\pi \eps^2 \left(2\pi\vc^{T}{\mathcal G}_s\vc +
    \overline{E} \right) + {\mathcal O}(\eps^3\log^2\eps)\,,
%  \lambda_0 \sim \frac{2\pi \eps \overline{C}}{|\Omega|} +
%  \frac{\pi \eps^2\log\eps}{|\Omega|} \sum_{j=1}^{N} \HC_j C_j^2 \\
%   -\frac{2\pi \eps^2}{|\Omega|} \left(2\pi\vc^{T}{\mathcal G}_s\vc +
%    \overline{E} \right) + {\mathcal O}(\eps^3\log^2\eps)\,,
\end{equation}
where $\overline{C}=\sum_{j=1}^{N}C_j$, $\overline{E}=\sum_{j=1}^{N}
\left( \HC_j E_{j+} + \kapdif_j E_{j-}\right)$, while ${\mathcal G}_s$
is the surface Neumann Green's matrix defined in
(\ref{mfpt:green_mat}).  When $\Omega$ is the unit sphere, where
$\HC_j=1$ and $\kapdif_j=0$ for $j=1,\ldots,N$, (\ref{sec_eig:end}) is
equivalent to that derived in \cite{GrebenkovWard2026big} upon noting
the different definition of the common diagonal elements of the
Green's matrix that was compensated for by the modified logarithmic
gauge $\log\left({\eps/2}\right)$.

\vspace*{0.1cm}

\section{Numerical validation and optimization results}\label{sec:results}

In this section, we describe our methods for the boundary integral
solution of the Berg-Purcell \eqref{berg_bp:ssp} and narrow escape
problems \eqref{mfpt:ssp}.  Using these numerical methods, we validate
the main results in Propositions~\ref{berg:main_res} and
\ref{mfpt_b:main_res}, and demonstrate that these expansions exhibit
the predicted error as $\eps\to0$.  Through some simple examples, we
explore how geometry and patch configurations combine to modulate the
key quantities of capacitance and GMFRT.

\subsection{Boundary integral method}\label{sec:BIEMethod}

Previous numerical methods relied on the Neumann-to-Dirichlet map to
represent the solution of \eqref{berg_bp:ssp} and \eqref{mfpt:ssp}
followed by spectral expansion of the surface potential and flux on
each patch \cite{BL2018,KAYE2020}, {\clb as well as other spectral
methods \cite{Grebenkov19} and Monte Carlo simulations
\cite{PlunkettLawley2024,GrebenkovWard2026}.}  In our numerical treatment, we use a standard boundary integral
equation approach to solve both the Berg-Purcell problem
\eqref{berg_bp:ssp} and narrow escape problem \eqref{mfpt:ssp}.  In
order to make the presentation as self-contained as possible, we
briefly sketch the procedure.  We first describe the approach for
solving \eqref{berg_bp:ssp} before outlining the modifications
required for \eqref{mfpt:ssp}.  In the following we will always take
$\partial \Omega$ to be smooth and $\n$ to denote its outward pointing
normal (note that this is the opposite of the convention used in
\eqref{berg_bp:ssp} but is more standard in the boundary integral
literature).

We begin by introducing the ansatz,
\begin{align}\label{ansatz:berg}
  u(\x) = 1 + \int_{\partial \Omega} \frac{1}{4\pi |\x-\y|}
  \sigma(\y)\,ds(\y), \quad \x \in \mathbb{R}^3\setminus \Omega \,,
\end{align}
where $\sigma$ is a yet to be determined boundary density. Noting that
$\Delta_\x [4\pi |\x-\y|]^{-1} = -\delta(\x-\y)$, it follows
immediately that the $u(\x)$ defined in \eqref{ansatz:berg} trivially
satisfies \eqref{berg_bp:ssp_1} and \eqref{berg_bp:ssp_3}, provided
$\sigma \in L^1(\partial \Omega)$.  All that remains is to satisfy the
boundary conditions on $\partial \Omega$.  To that end, we note that
for $\sigma \in L^2(\partial \Omega)$ the normal trace of the integral
exists in an $L^2(\partial \Omega)$ sense~\cite{kersten}, and
\begin{align}\label{limit_L2}
  \int_{\partial \Omega} \left| \n(\x)\cdot \nabla_\x \mathcal{S}[\sigma]
  (\x\pm h \n(\x)) - S^{\prime}[\sigma](\x) \pm
  \frac{1}{2} \sigma(\x) \right|^2\,
  ds(\x) \to 0 \,,
\end{align}
as $h\to 0^+$, where we have defined
\begin{equation*}
  \mathcal{S}[\sigma](\x) \equiv \int_{\partial \Omega} \frac{1}{4\pi
    |\x-\y|}\sigma(\y)\,ds(\y)\,, \quad \x \notin \partial \Omega\,,
\end{equation*}
and 
\begin{equation*}
  S^{\prime}[\sigma](\x) \equiv \int_{\partial \Omega}
  \frac{(\x-\y)\cdot \n(\x)}{4\pi |\x-\y|^3}\sigma(\y)\, ds(\y)\,,
  \quad \x \in \partial \Omega \,.
\end{equation*}
We also let $S$ denote the trace of $\mathcal{S}$, i.e. $S[\sigma](\x)
= \lim_{h\to 0} \mathcal{S}[\sigma](\x + h \n(\x))$ for $\x \in
\partial \Omega$.  
\iffalse Note that if $\partial \Omega$ is twice continuously differentiable then $(\x-\y)\cdot \n(\x) = O(|\x-\y|^2)$ as $\x \to \y$ on $\partial \Omega,$ with an implicit constant that can be chosen uniformly in $\partial \Omega.$ Hence, for some fixed constant $C$ depending only on $\partial \Omega$ the kernels of both $S$ and $S'$ are both bounded by $C/|\x-\y|$ uniformly in $(\x,\y) \in \partial \Omega \times \partial \Omega.$ In particular, if $\partial \Omega$ is smooth then $S',S : H^{s}(\partial \Omega) \to H^{s+1}(\partial \Omega)$ for every $s \in \mathbb{R},$ and are compact from $L^2(\partial \Omega)$ to itself (see~\cite{costabelShapeDerivativesBoundary2012} for example, and the references therein).\fi

Inserting our ansatz into the boundary conditions
(\ref{berg_bp:ssp_2b}-\ref{berg_bp:ssp_2}) yields the following
integral equation
\begin{align}\label{eqn:berg_bie}
  \frac{1}{2}\sigma(\x) -S^{\prime}[\sigma](\x) +
  \sum_{i=1}^N \frac{b_i}{\eps} \chi_{\partial \Omega_i^\eps}(\x)S[\sigma](\x)
  = - \sum_{i=1}^N \frac{b_i}{\eps}\chi_{\partial \Omega_i^\eps}(\x)\,,
  \quad \x \in \partial \Omega \,,
\end{align}
where $\chi_{\partial \Omega_i^\eps}$ denotes the indicator function
of $\partial \Omega_i^\eps.$ Standard arguments~\cite{kress} give that
$S$ and $S^{\prime}$ are compact from $L^2(\partial \Omega)$ to itself which
in turn implies that the above equation is a Fredholm second-kind
integral equation.  Uniqueness of \eqref{eqn:berg_bie} follows by
standard arguments, see~\cite{kress}, for example.  For the given
right-hand side, standard bootstrapping arguments give that $\sigma$
is $C^\infty$ away from the boundaries of the patches $\partial
\Omega_i^{\eps}.$ Near the patch boundaries $\sigma$ is discontinuous,
and the leading singular behavior beyond the jump scales like $d \log
d,$ where $d$ is the distance to the patch boundary.

Conveniently, we note that the capacitance $C_{\rm T}$ can be computed
directly from $\sigma$ as
\begin{equation*}
C_{\rm T} = -\frac{1}{4\pi} \int_{\partial \Omega} \sigma(\x)\,{\rm d}\x \,.
\end{equation*}

We discretize~\eqref{eqn:berg_bie} using the package
\texttt{fmm3dbie}~\cite{fmm3dbie}, which uses the collocation scheme
described in~\cite{loc_cor_quad}.  Since $\sigma$ is discontinuous
across the patch boundaries, we require meshes with patch boundaries
coinciding with edges of mesh elements. We construct such meshes as
follows.  We first compute a first-order triangular mesh using
\texttt{distmesh}~\cite{doi:10.1137/S0036144503429121}, modified to
constrain a ring of vertices to lie on each patch boundary, so that
every patch boundary is approximated by a closed polygon of mesh
edges.  Each triangle is then upgraded to a curvilinear mesh element
by placing Vioreanu--Rokhlin
nodes~\cite{vioreanu2014spectra,xiaogimbutas} on the flat triangle and
projecting them onto $\partial\Omega$ by a Newton iteration along the
gradient of the level-set function defining $\Omega$.  For triangles
with an edge or vertex on a patch boundary, the chart is blended with
a parameterization of the bounding circle, so that curvilinear mesh
element boundaries lie on the patch boundaries to the full order of
the interpolation.  Finally, mesh elements may be graded dyadically
toward the patch boundaries to resolve singularities in the density.

Solving \eqref{mfpt:ssp} proceeds in an almost identical way,
replacing the ansatz \eqref{ansatz:berg} with
\begin{equation*}
u(\x) = -\frac{|\x|^2}{6} +\mathcal{S}[\sigma](\x)\,, \quad \x \in \Omega\,.
\end{equation*}
The volume integral \eqref{mfpt:ubar} can be reduced to a boundary
integral using integration by parts.  For further details, see
\cite{chakraborty2025fastintegralmethodsneumann,lindsay20263D_GFun}.

Numerical experiments suggest our method is able to resolve the
solutions of \eqref{berg_bp:ssp} and \eqref{mfpt:ssp} with numerous
well separated small patches.  With modest grading, we obtain relative
errors in the density of at most $10^{-4}$ which is sufficient for
verification of the asymptotic results.

\subsection{Calculation of capacitance and monopole constants}\label{sec:numerics_consts}

For validation of the asymptotic expansions in the case of circular
patches, we need to calculate the capacitance $C$ and monopole
coefficient $E_{+}$. For a disk-shaped patch $\PT$ of radius $a$ with
reactivity $b$, the capacitance is given by
\cite[Eqn.~(2.6)]{GrebenkovWard2026},
\begin{equation}\label{mfpt:C_steklov}
  C(b) =  a {\mathcal C}^{}(ab) \,, \quad
  \mbox{where} \quad {\mathcal C}(z)= \frac{z}{2\pi} \sum_{k=0}^{\infty} \frac{\mu_k d_k^2}{\mu_{k} + z}
   \,.
\end{equation}
In the formulation \eqref{mfpt:C_steklov}, the pairs $(\mu_k,d_k)$ are
determined by the exterior Steklov eigenvalue problem for the unit
disk $\Gamma = \{(y_1,y_2) \in\mathbb{R}^2\ | \ y_1^2 + y_2^2 \leq
1\}$.  This is defined \cite[Eqn.~B1]{GrebenkovWard2026} by
\bsub\label{eqn:Steklov}
\begin{align}
    \Delta \Psi_k &= 0, \quad \y\in\mathbb{R}^3_+;&\\
    \partial_{n} \Psi _k &= \mu_k\Psi_k, \quad y_3 = 0 ,\ (y_1,y_2)\in\Gamma;&\\
    \partial_{n} \Psi _k &= 0, \quad  y_3 = 0,\quad (y_1,y_2) \notin\Gamma;\\
    \Psi_k (\y) &= \mathcal{O}(1/|\y|), \quad \mbox{as} \quad |\y|\to\infty.
\end{align}
\esub
This problem (see \cite{GrebenkovWard2026big}) admits infinitely many
nontrivial solutions $\{\mu_k,\Psi_k\}_{k=0}^{\infty}$ such that the
eigenvalues form an ordered sequence
\[
0 = \mu_0<\mu_1\leq \mu_2\leq \ldots
\]
while the restriction of the eigenfunctions $\Psi_k|_{\Gamma}$ form a
complete orthonormal basis $L^2(\Gamma)$ (see
\cite{Arendt15,Bundrock25} for more details).  The complementary
values $d_k$ are defined by
\[
d_k = \int_{\Gamma} \Psi_k(\y) d\y.
\]
Algorithms to numerically compute the parameters $\mu_k$ and $d_k$
were provided in \cite{Grebenkov24,GrebenkovWard2026}.  Extensions to
patches of general shape were developed in
\cite{Grebenkov-Maurette}.

The monopole coefficient $E_{+}$ was determined in \cite[Appendix
E]{GrebenkovWard2026} to be
\bsub\label{eqn:monopolE_num}
\begin{equation}
    E_{+}(\kappa) = -\frac{C^2}{2} \log a + a^2 \mathcal{E}(a\kappa),
\end{equation}
where the function $\mathcal{E}(z)$ is defined by
\begin{equation}\label{eqn:E+Int}
  \mathcal{E}(z) = 2 \int_0^1 \frac{1}{r} \left( \int_0^r r^{\prime}
  a q(a r^{\prime};z/a )dr^{\prime} \right)^2 dr \,,
\end{equation}
\esub
where $q(\y;\kappa)$ is the charge density defined in
\eqref{berg:wc_charge}. By representing the charge density
$q(\y;\kappa)$ in terms of the Steklov problem \eqref{eqn:Steklov},
the calculation of \eqref{eqn:E+Int} is reduced to quadrature.

\subsection{Convergence of the MFRT for prolate spheroids}

To validate the asymptotic analysis, we consider a prolate spherical
geometry with boundary given by
\begin{equation}\label{eq:prolate}
  {\partial \Omega} = \Big\{ (x,y,z)\in\mathbb{R}^3 \ \Big|
  \ \frac{x^2 + y^2}{a_e^2} + \frac{z^2}{c_e^2} = 1 \ \Big\} \, ,
\end{equation}
where $c_e>a_e>0$.  The prolate spheroidal coordinate system for
$\x=(x,y,z)^T$ is
\begin{equation}
    x = f \sqrt{(\zeta^2-1) (1 - \chi^2)} \cos\phi\,, \quad
    y = f \sqrt{(\zeta^2-1) (1 - \chi^2)} \sin\phi\,, \quad 
    z = f \, \zeta\, \chi \,,
\end{equation}
where $f = \sqrt{c_e^2-a_e^2}$ is half the inter-focal distance. Here
$\zeta\in[1,\infty)$ is the radial variable, $\chi \in[-1,1]$ is the
elevation coordinate and $\phi \in [0,2\pi)$ is the azimuthal
variable. The surface of the prolate spheroid is defined by points
$\x = (\chi,\phi,\zeta_b)$ where $\zeta_b = {c_e/f}$. In our 
experiments below we fill fix $a_e= 1$ and $c_e= \sqrt{2}$ so that
$f=1$ and $\zeta_b = \sqrt{2}$.

To apply (\ref{mfpt_b:one_circular}) for the prolate spheroid case, we
note that the mean curvature of this surface at the point
$(\chi,\phi,\zeta_b)$ is
\begin{equation}\label{eqn:meanCurvature}
  \HC = \frac{\zeta_b(2\zeta_b^2 - \chi^2 -1)}{2(\zeta_b^2-1)^{\frac12}
    (\zeta_b^2 - \chi^2)^{\frac32}}\,.
\end{equation}
The evaluation of the regular part $R_{s}(\x_1)$ satisfying
(\ref{mfpt:green_int_sing}) is performed numerically using the methods
described in \cite{lindsay20263D_GFun}. The values of $\HC$ and $R_s$ as
the elevation coordinate $\chi\in[-1,1]$ is varied are shown in
Fig.~\ref{fig:Prolate_a}.

\begin{figure}[htbp]
\centering
\subfigure[$R_s(\x_1)$ (left axis) and $\HC_1$ (right axis).]{\includegraphics[width = 0.495\textwidth]{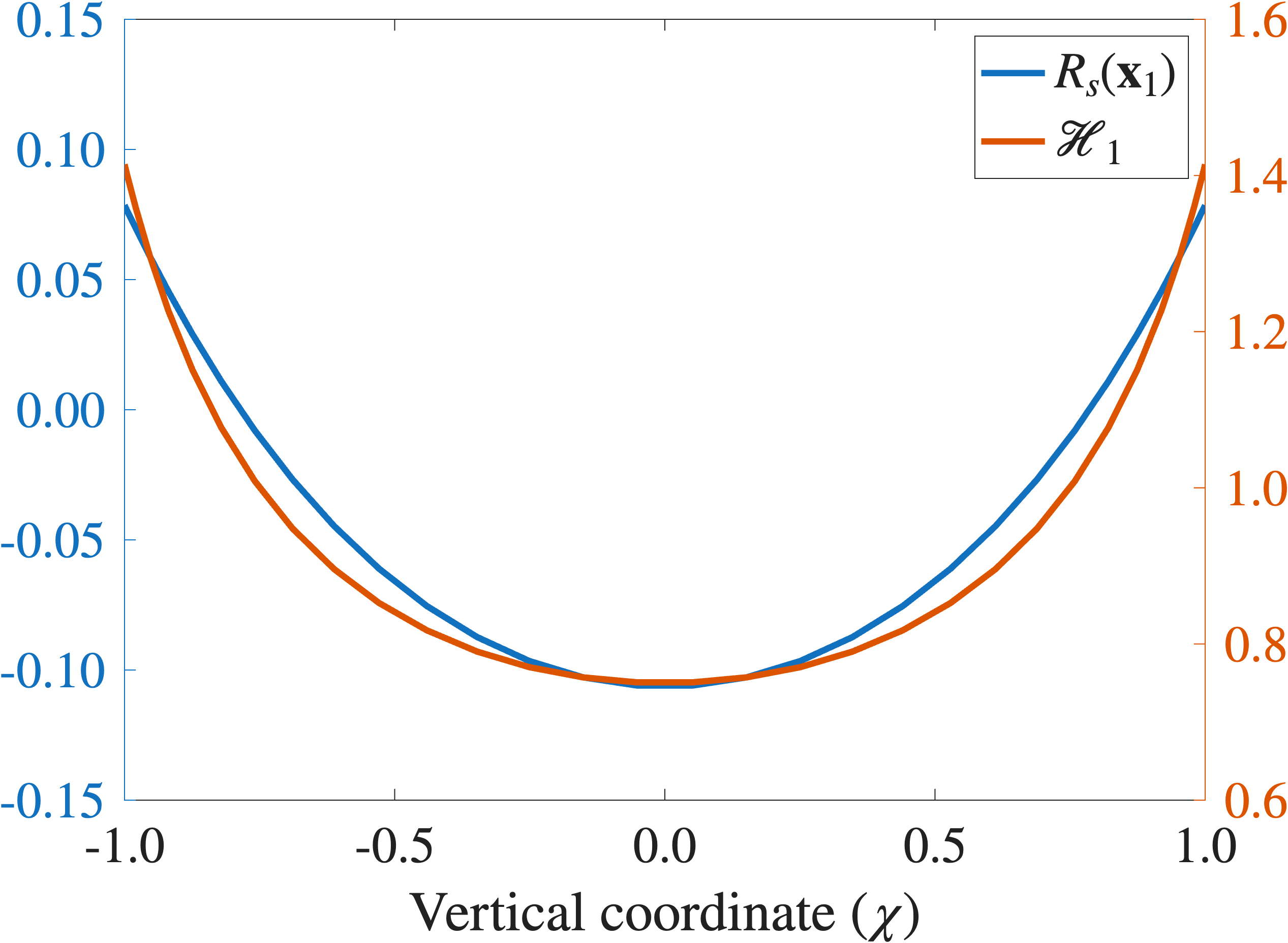}\label{fig:Prolate_a}}\quad
\subfigure[Agreement between $\overline{u}$ and FEM vs $\chi_0$.]{\includegraphics[width = 0.45\textwidth]{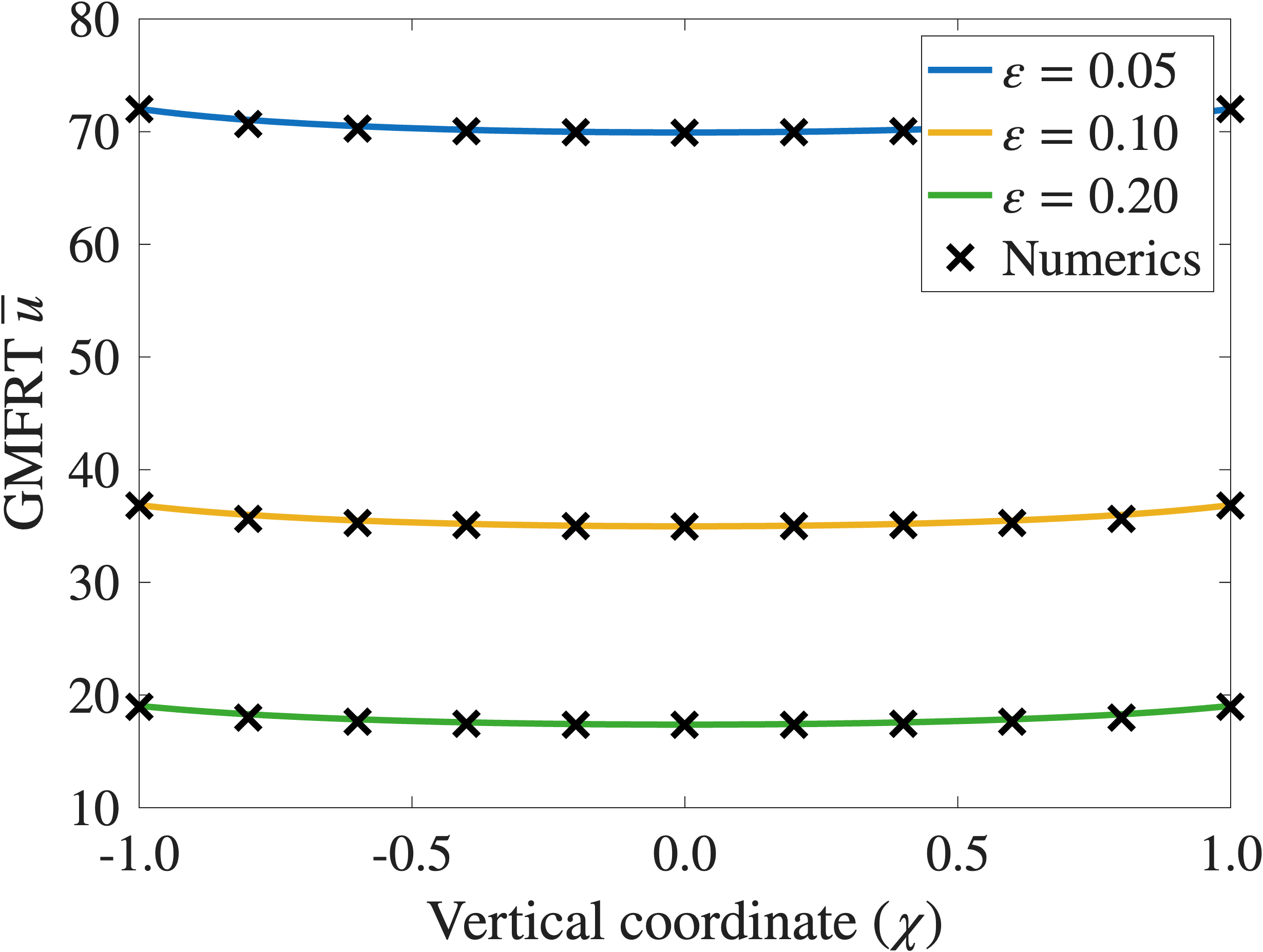}\label{fig:Prolate_b}}\\[5pt]
\subfigure[Rescaled GMFRT $\eps\overline{u}/U_0$ for a single polar and equator patch vs. $\eps$.]{\includegraphics[width = 0.475\textwidth]{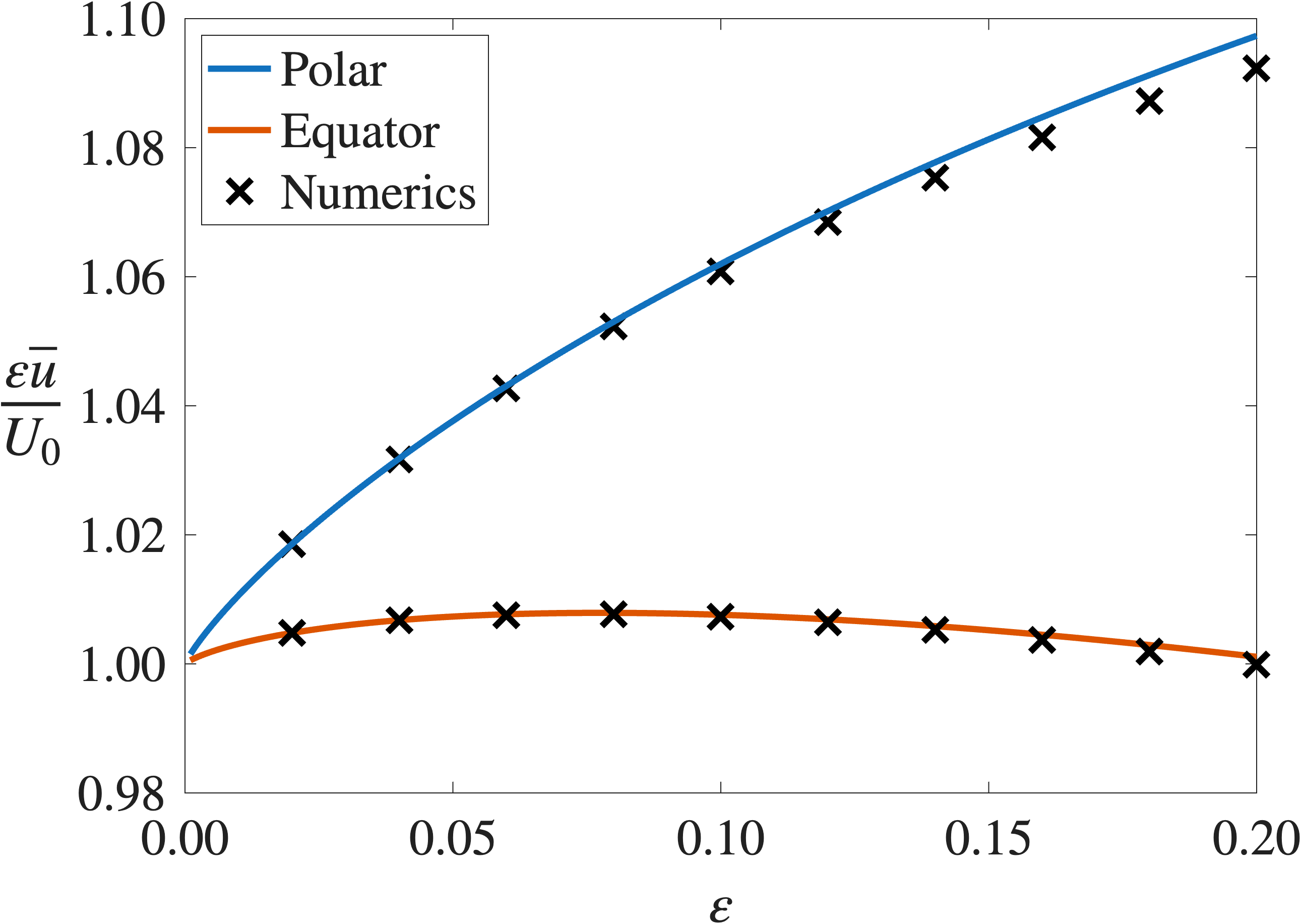}\label{fig:Prolate_c}}\quad
\subfigure[Relative error in $\overline{u}$ against $\eps$.]{\includegraphics[width = 0.475\textwidth]{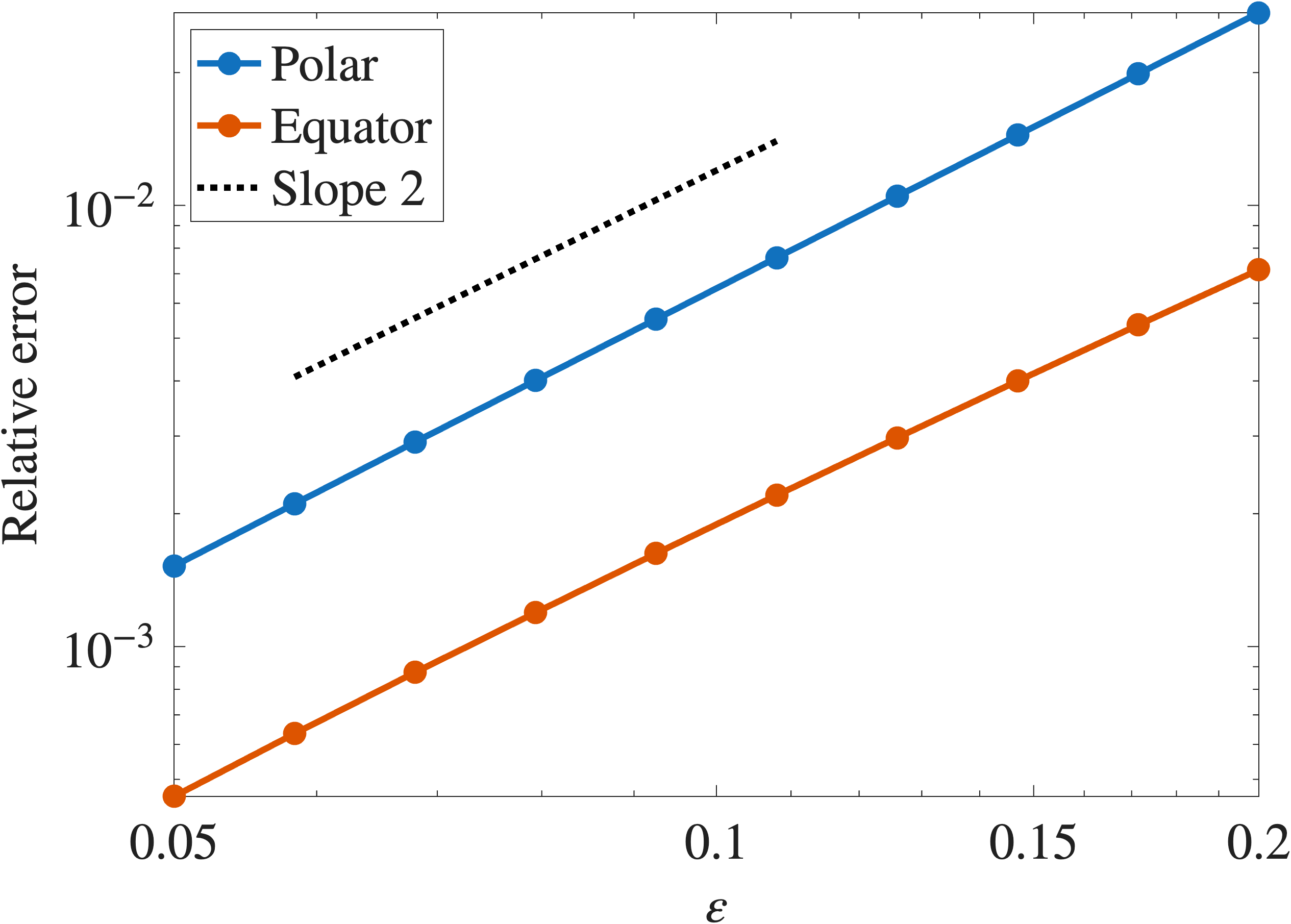}\label{fig:Prolate_d}}
\caption{
Validation of the GMFRT calculation for a single perfectly absorbing
circular patch centered at $\x_1 = (\chi,\phi,\zeta_b)$. Panel (a):
The regular part $R_{s}(\x_1)$ and $\HC_1$ versus $\chi$. Panel (b):
Comparison of asymptotics (\ref{mfpt_b:one_circular}) and the
numerical solution (FEM) as the patch center moves from the south pole
$(\chi= -1)$ to the north pole ($\chi=1$) for various values of
$\eps$. Panel (c): Comparison between numerical solution and
asymptotics for a single patch centered at the north pole and the
equator.  Panel (d): Convergence of the relative errors
(\ref{eqn:RelativeError}) between the asymptotic approximation and
numerical solution as $\eps\to0$ with line of slope $2$
overlaid. }
\label{fig:Prolate}
\end{figure}

One of our primary observations from
Fig.~\ref{fig:Prolate_b}-\ref{fig:Prolate_c} is that the GMFRT
$\overline{u}$ is minimized when the boundary patch is placed at the
equator.  This location corresponds to the global minimum of both
$R_{s}(\x_1)$ and $\HC_1$ as seen in Fig.~\ref{fig:Prolate_a} and, as
predicted by our asymptotic result (\ref{mfpt_b:one_circular}), should
correspond to the location where $\overline{u}$ is minimized when
$\eps\ll 1$.  This expected behavior is confirmed in
Fig.~\ref{fig:Prolate_b} where we show, for several values of $\eps$,
that the minimizer occurs for a patch centered at the
equator. Conversely, the global MFRT $\overline{u}$ is maximized when
the single patch is centered at a pole as shown in
Fig.~\ref{fig:Prolate_c}.  In Fig.~\ref{fig:Prolate_d}, we calculate
the relative error
\begin{equation}\label{eqn:RelativeError}
  \mathcal{E}_{\text{rel}} = \left|\frac{\overline{u}_{\text{num}}-
      \overline{u}_{\text{asy}}}{\overline{u}_{\text{num}}}\right|\,,
\end{equation}
where $\overline{u}_{\text{num}}$ is the GMFRT from the numerical
solution (see \S\ref{sec:BIEMethod}) and $\overline{u}_{\text{asy}}$
is the asymptotic result (\ref{mfpt_b:one_circular}).  This result is
consistent with the predicted behavior of the relative error
$\mathcal{O}(\eps^2\log\eps)$ as $\eps\to 0$.

\subsection{Validation of the Berg-Purcell problem}

In this subsection we consider several special cases of our result in
Proposition \ref{berg:main_res} for the Berg-Purcell problem.  We
again use the prolate spheroid described by \eqref{eq:prolate}.

\paragraph{A single {\clb perfectly} absorbing patch}
We first consider perfectly absorbing circular patches of a common
radius $\eps$ so that $C_i={2/\pi}$ and $E_{i+}=-2{\big(\log4 -
\frac{3}{2}\big)\pi^2}$. In this case, \eqref{berg:main_res_1} for
$C_{T}$ was given previously in (\ref{eq:intro_CT}). For a single
patch $(N=1)$, (\ref{eq:intro_CT}) becomes
\begin{equation}\label{BP_reducedN=1}
  \frac{1}{C_T} \sim \frac{\pi}{\eps }\left[ 1 -
    \frac{\eps\HC_1}{\pi} \log(4 \eps) + 4\eps \left( R_{e}(\x_1) +
    \frac{3\HC_1}{8\pi}\right) + {\mathcal O}(\eps^2\log\eps)\right] \,,
\end{equation}
where the regular part of the exterior surface Neumann Green's function
$R_{e}(\x_1)$ was defined in (\ref{berg:green_int_sing}). This
expression shows how $C_T$ depends on local geometric factors through
the mean curvature $\HC_1$ and the global contribution $R_{e}(\x_1)$.
When $\Omega$ is the unit sphere, $\HC_1=-1$ and $R_e=-{\log(2)/(4\pi)}$
for any $\x_1\in \partial\Omega$ (see (\ref{berg:gs_loc1}) of Appendix
\ref{app:explicit}), so that (\ref{BP_reducedN=1}) becomes
\begin{equation}\label{BP_reducedN=1_sphere}
  \frac{1}{C_T} \sim \frac{\pi}{\eps }\left[ 1 +
    \frac{\eps}{\pi} \log(2\eps) -\frac{3\eps}{2\pi}+
    {\mathcal O}(\eps^2\log\eps)\right] \,,
\end{equation}
which is in agreement with Eq. (3.39) of \cite{LWB2017}.

Based on our sign convention for $\HC$, to apply (\ref{BP_reducedN=1})
for a spheroidal geometry one needs only to replace $\HC$ in
(\ref{eqn:meanCurvature}) with $-\HC$.

With this minor change, in Fig.~\ref{fig:BPexample_a} we plot both
$\HC_1$ and $R_e(\x_1)$ on the {\clb spheroid} with semi-axes
$(a_e,a_e,c_e) = (1,1,\sqrt{2})$.  The value of $R_{e}(\x_1)$ has been
obtained numerically from a boundary integral method described
\cite{lindsay20263D_GFun}.  We observe that both these quantities
attain their global minimum at the poles of the {\clb spheroid}
indicating that a single patch placed at either $\x = (0,0,\pm
\sqrt{2})$ will generate the largest capacitance $C_T$.

\paragraph{Two Robin patches}

As a further test of \eqref{eq:intro_CT}, we consider two circular
patches of common radius $\eps$ and reactivity $b=1$.  We calculate
that
\bsub\label{eqn:TwoPatchRobin}
\begin{equation}\label{eqn:TwoPatchRobin_a}
    C = 0.271628\,, \qquad E_{+} = 0.008360\,, \qquad E_{-} = 0 \,.
\end{equation}
In this case, the formula \eqref{eq:intro_CT} for $C_T$ reduces to
\begin{equation}\label{eqn:TwoPatchRobin_b}
    \frac{1}{C_T} \sim \frac{2}{NC \eps} \left[ 1 - \frac{C\bar{\HC}}{2N} \eps\log\eps 
+ \eps\Big(2\pi C p_e(\x_1,\x_2) + \frac{\bar{\HC} E_{+}}{NC}\Big) \right],
\end{equation}
where $\bar{\HC} = {\HC}_1 + {\HC}_2$ and the interaction term is given by
\begin{equation} \clb
    p_e(\x_1,\x_2) = 2G_{e}(\x_1;\x_2) + R_{e}(\x_1) + R_{e}(\x_2) \,.
\end{equation}
\esub We consider patches that are aligned symmetrically about either
the poles or the equator (see
Figs.~\ref{fig:BPexample_a}-\ref{fig:BPexample_b}).  For this case, we
numerically calculate the regular parts and interaction terms to be
\begin{align*}
  R_e(\x_{1,2}) = -0.03232 \quad& (\text{Polar case})\,,
  \qquad R_e(\x_{1,2}) =-0.05982&(\text{Equatorial case})\,; \\
  G_e(\x_1;\x_2) = 0.01992 \quad &(\text{Polar case})\,, \qquad
  G_e(\x_1;\x_2) = 0.02269 &(\text{Equatorial case}) \,.
\end{align*}

As a function of $\eps$, in Fig.~\ref{fig:BPexample_d} we observe
clear agreement between our asymptotic formula (\ref{eq:intro_CT}) and
numerical simulations. We observe as expected that patches centered at
the polar regions provide a larger capacitance. As a further
validation of our asymptotic result (\ref{eq:intro_CT}), we calculate
the relative error between the numerically computed capacitance {\clb
($C_T^{\text{num}})$} and the asymptotic approximation ($C_T$) defined
by
\begin{equation}\label{eqn:rel_err_C} \clb
    \mathcal{E}_{\text{rel}} = \frac{|C_T-C_T^{\text{num}}| }{C_T^{\text{num}}}.
\end{equation}
In Fig.~\ref{fig:BPexample_e} we plot $\mathcal{E}_{\text{rel}}$ for a
range of $\eps\to0$ and observe the predicted decay rate.

\begin{figure}[htbp]
    \centering
    \subfigure[Equator patches.]{\includegraphics[width = 0.22\textwidth]{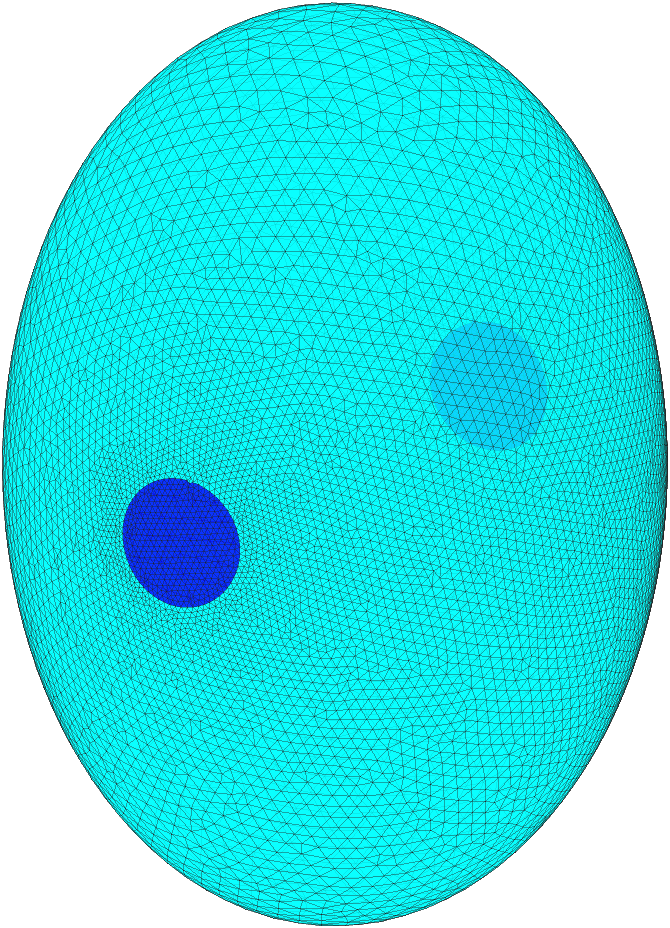}\label{fig:BPexample_a}}
    \subfigure[Polar patches.]{\includegraphics[width = 0.25\textwidth]{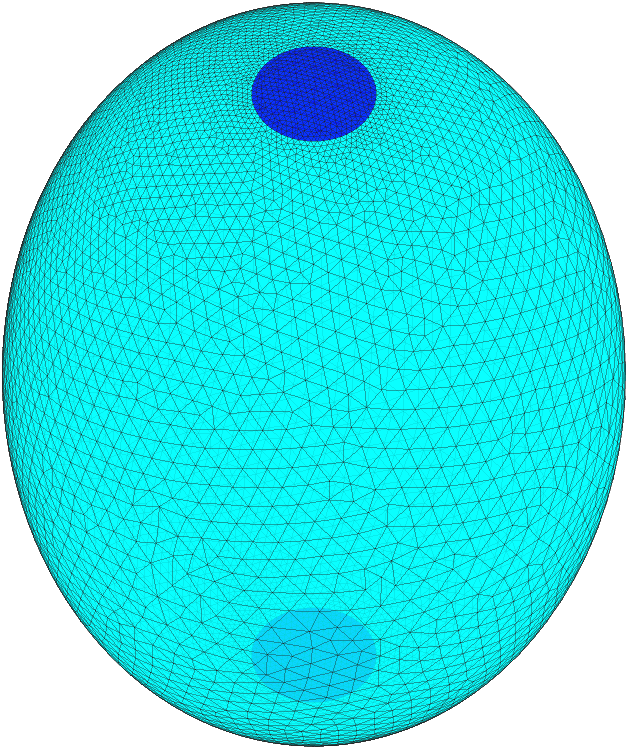}\label{fig:BPexample_b}}\qquad
     \subfigure[$R_e(\x_1)$ (left axis) and $\HC_1$ (right axis)]{\includegraphics[width = 0.45\textwidth]{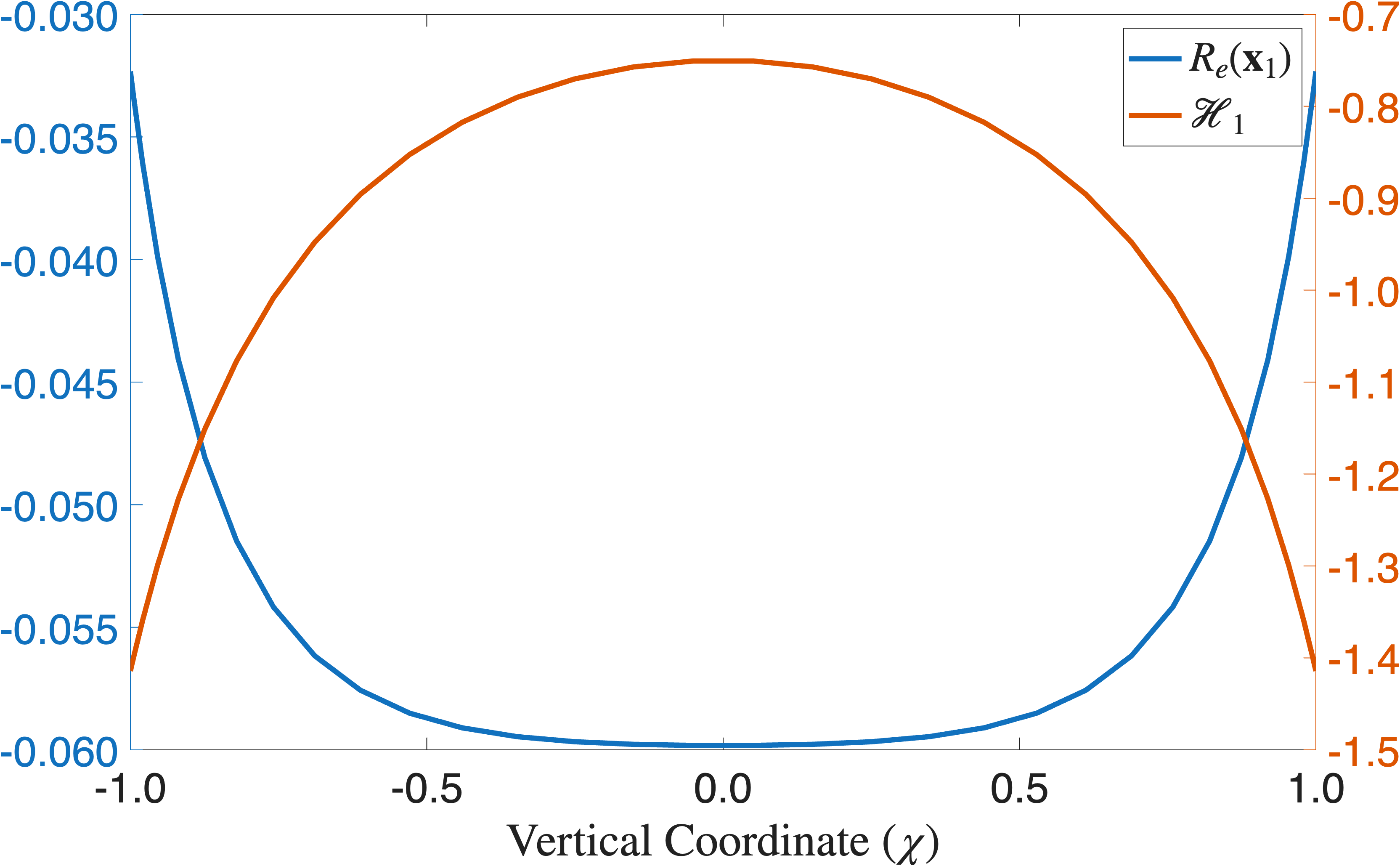}\label{fig:BPexample_c}}\\
       \subfigure[Rescaled $C_T$ against $\eps$.]{\includegraphics[width = 0.45\textwidth]{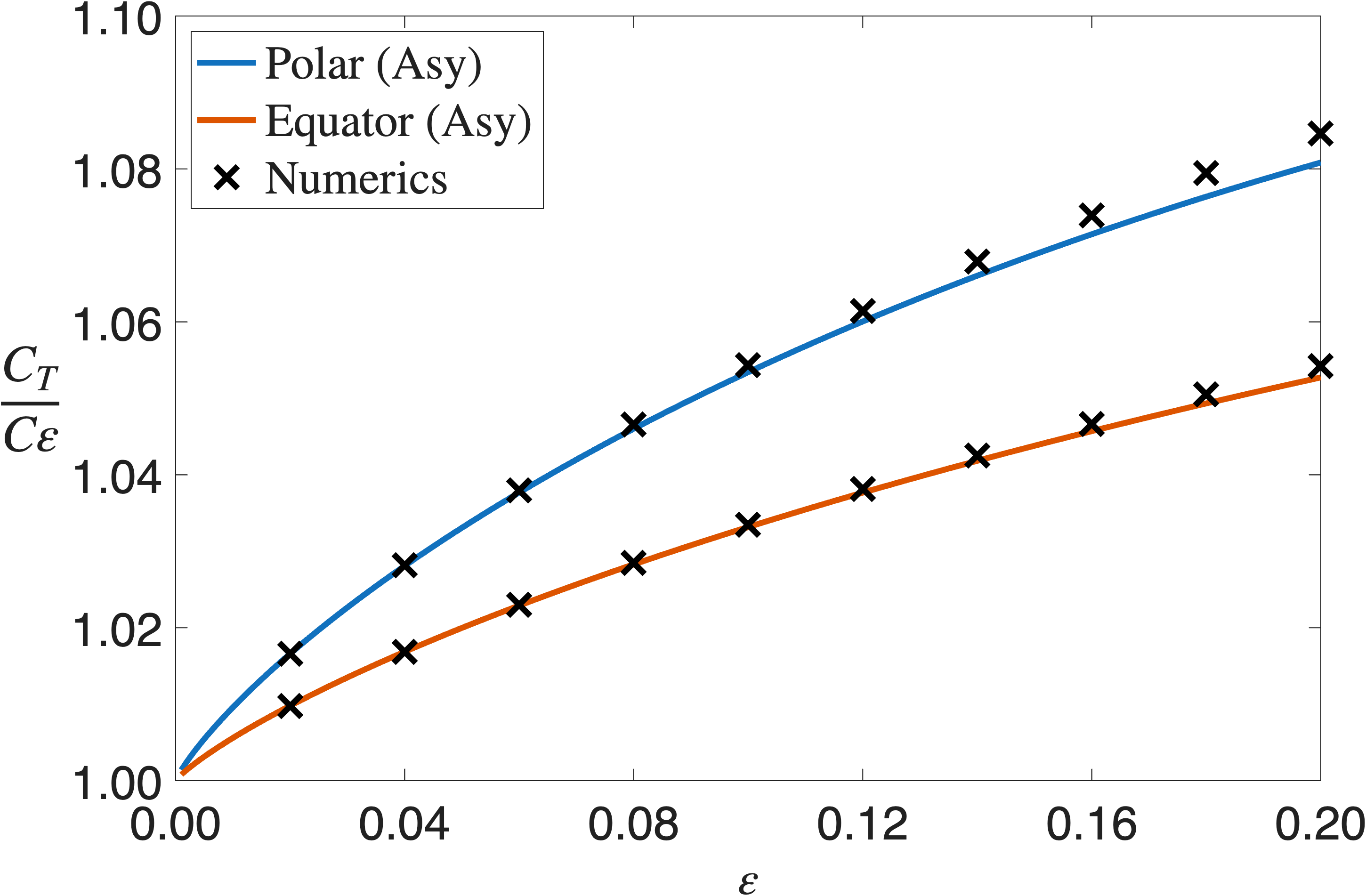}\label{fig:BPexample_d}}
       \qquad
        \subfigure[Relative error in $C_T$ against $\eps$.]{\includegraphics[width = 0.45\textwidth]{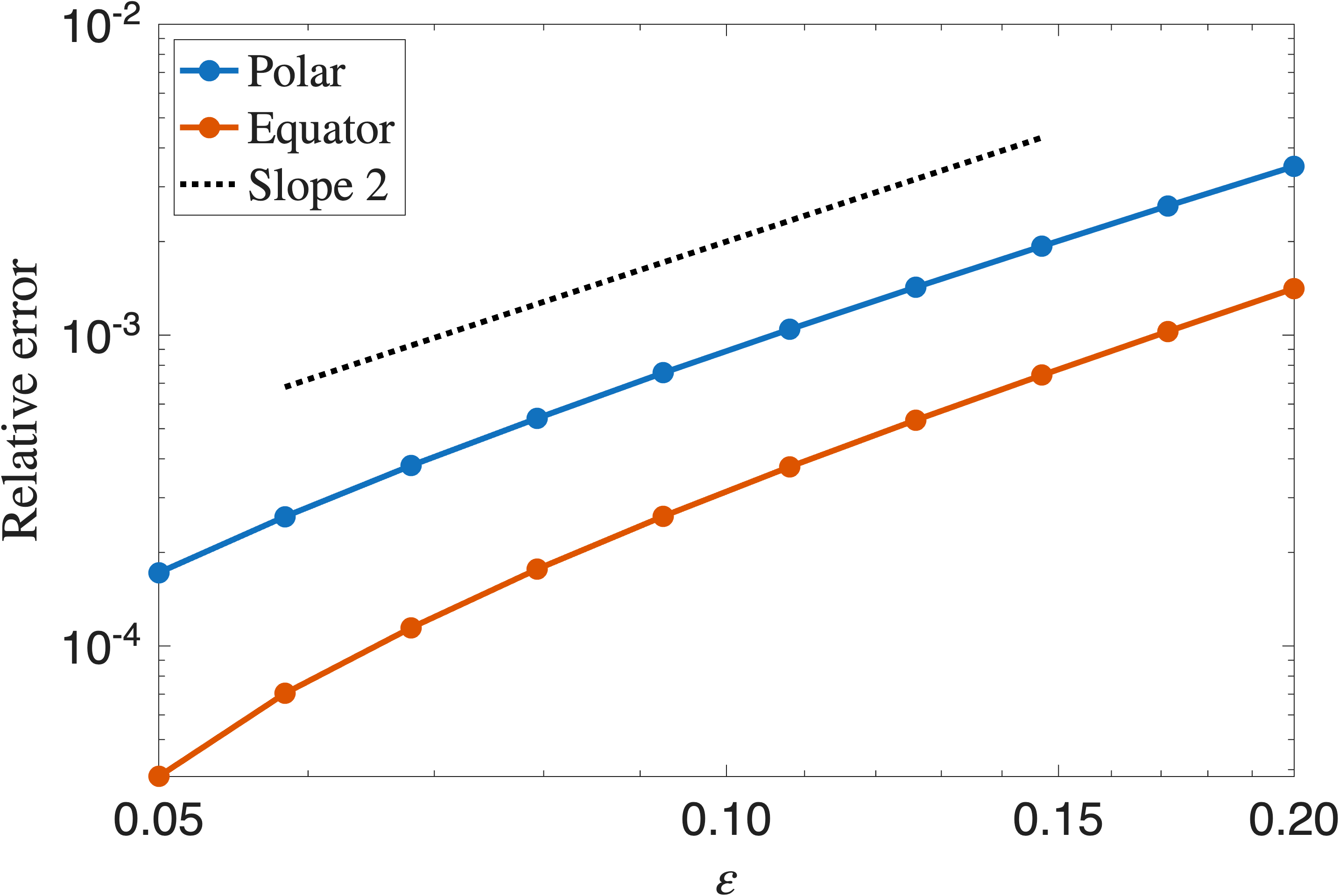}\label{fig:BPexample_e}}
\caption{
Berg-Purcell problem for the {\clb spheroid} with semi-axes
$(a_e,a_e,c_e)=(1,1,\sqrt{2})$.  The capacitance of a configuration with two
patches of common reactivity $b=1$ and radius $\eps$ situated either
on the equator (panel (a)) or on the poles (panel (b)) is calculated.
Panel (c): regular part $R_e(\x_1)$ of the surface Neumann Green's
function computed by the numerical method of
\cite{lindsay20263D_GFun}.  Panel (d): comparison between our
asymptotic result (\ref{eq:intro_CT}) and numerical computations for
shrinking patch radius $\eps$.  The polar configuration provides the
larger capacitance.  Panel (e): the relative error
\eqref{eqn:rel_err_C} in $C_T$ decays according to the rate predicted
in \eqref{eq:intro_CT}. }
\label{fig:BPexample}
\end{figure}

\subsection{Comparison with points on the Fibonacci spiral}

In this example, we consider the Berg-Purcell formula for the case of
$N$ identical patches arranged on nodes of the Fibonacci spiral
points.  The Fibonacci spiral points are a simple and highly optimized
set of points for generating {\clb near-uniform} coverings of the
sphere \cite{Gonzalez2009,Swinbank2006}.

The $N$ patch centers are placed at the nodes of a Fibonacci (spherical)
spiral, which gives a near-uniform, well-separated configuration for any $N$.
With the golden angle $\gamma_g \equiv \pi\,(3-\sqrt{5})$, define for
$k = 0,1,\dots,N-1$
\begin{equation}\label{eq:fib_angles}
  \chi_k = 1 - \frac{2k+1}{N} \;\in\; (-1,1),
  \qquad
  \phi_k = k\,\gamma_g \pmod{2\pi},
\end{equation}
where $\chi_k$ is the elevation and $\phi_k$ the azimuth.  The
half-integer offset $(2k+1)/N$ spaces the nodes equally in $\chi$
(hence by equal area in latitude) and keeps them strictly off the
poles, while the golden-angle increment $\gamma_g$ distributes the
azimuths quasi-uniformly.  The corresponding centers on the {\clb
  spheroidal boundary} $\partial\Omega$ with semi-axes $(a_e,a_e,c_e)
= (1,1,\sqrt{2})$ are
\begin{equation}\label{eq:fib_points}
  \mathbf{x}_k =
  \Big(\, a\sqrt{1-\chi_k^2}\,\cos\phi_k,\;\;
          a\sqrt{1-\chi_k^2}\,\sin\phi_k,\;\;
          c\,\chi_k \,\Big),
  \qquad k = 0,\dots,N-1 .
\end{equation}

\begin{figure}
    \centering
    \includegraphics[width=0.975\textwidth]{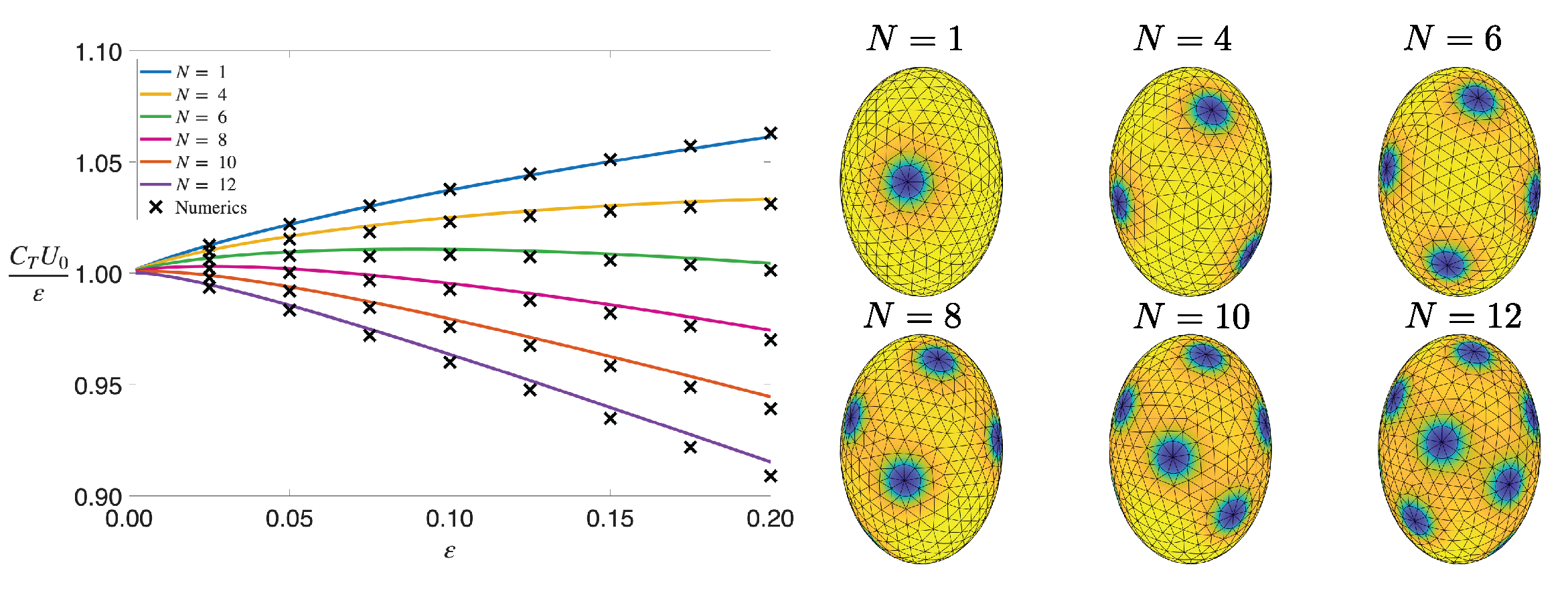}
\caption{
Solution of the Berg-Purcell problem \eqref{berg_bp:ssp} with $N$
identical Robin patches of radius $\eps$ and reactivity $b=1$ on the
{\clb spheroid} $(a_e,a_e,c_e) = (1,1,\sqrt{2})$.  For a range of patches
centered at the Fibonacci spiral points \eqref{eq:fib_points}, we plot
the rescaled capacitance $C_TU_0/\eps$ as $\eps\to0$ from asymptotic
expansions \eqref{berg:main_res_1} and full numerical solution (black
crosses).  The surface solution $u|_{\partial\Omega}$ of
\eqref{berg_bp:ssp} obtained from the boundary integral method
(cf.~Sec.~\ref{sec:BIEMethod}) is shown together with the mesh which
resolves the solution near each patch.\label{fig:Fibonacci} }
\end{figure}

In Fig.~\ref{fig:Fibonacci}, we plot the rescaled capacita nce
$(C_TU_0)/\eps$ against $\eps$ for various values of $N$. The solid
lines are the asymptotic approximation \eqref{berg:main_res_1} where
the entries of the Green's matrix are constructed using the methods in
\cite{lindsay20263D_GFun}.  The corresponding rescaled capacitance
determined from the boundary integral method (see
Sec.~\ref{sec:BIEMethod}) are shown with black crosses.  The close
agreement between these curves in the case of many patch numbers
serves as a highly non-trivial validation of our solution
methodologies applied to the problem \eqref{berg_bp:ssp}.  This
encompasses our high-order asymptotic expansions (Propositions
\ref{berg:main_res} and \ref{mfpt_b:main_res}), numerical computation
of the path-dependent singular Green's function
\eqref{berg:green_ext_full} and our boundary integral method
(\S\ref{sec:BIEMethod}).  In Fig.~\ref{fig:Fibonacci}, we plot the
solution $u|_{\partial\Omega}$ of \eqref{berg_bp:ssp} for $\eps=0.2$.
We note that the surface is well resolved by the mesh which conforms
to each of the $N$ boundary patches.

\subsection{Optimization of $C_T$ for a perturbation of a
  sphere}\label{sec:OptimNearSphere}

For the Berg-Purcell problem, we now investigate how to optimize
configurations of patches for domains that are small perturbations of
the unit sphere.  More specifically, we consider $N$ partially
reactive circular boundary patches of a common radius $\eps$ and
reactivity $b$.  Upon setting $C_i=C(b)$, $E_{i+}=E_{+}(b)$ and
$E_{i-}=0$ in (\ref{berg:main_res_1}), we obtain that
\begin{equation}\label{optim:main_f}
  \frac{1}{C_{\rm T}} = \frac{2}{N C \eps} \left[ 1 +
    \frac{\eps C}{2 N} \left(4\pi p_e - {\clb \overline{\HC}} \log\eps 
	+ {\clb \overline{\HC}} \frac{2 E_{+}}{C^2} \right)
    +  {\mathcal O}(\eps^2\log\eps) \right],
\end{equation}
where $p_e=p_e(\x_1,\ldots,\x_N)$ is the discrete inter-patch
interaction energy defined in (\ref{eq:intro_CT_b}), {\clb and
$\overline{\HC}$ is the sum of mean curvatures (see (\ref{eq:HCsum})).}
By recalling (\ref{berg:Ei_general0}) and (\ref{berg:wc_charge}) for
$E_{+}$ and $C$, we write (\ref{optim:main_f}), after neglecting the
${\mathcal O}(\eps^2\log\eps)$ error term, as
\bsub \label{optim:main}
\begin{equation}\label{optim:main_1}
  \frac{1}{C_{\rm T}} \sim \frac{2}{N C \eps} + \frac{1}{N^2} 
    {\mathcal P}(\x_1,\ldots,\x_N;b)  \,,
\end{equation}
where the objective function ${\mathcal P}={\mathcal
P}(\x_1,\ldots,\x_N;b)$ is defined by
\begin{equation}\label{optim:main_2}
  {\mathcal P} \equiv \left(-\log\eps + \gamma(b)\right)
  {\clb \overline{\HC}}
  + 4\pi \sum_{i=1}^N\Big[R_{e}(\x_i) +
  \sum_{\substack{j=1\\j\neq i}}^N G_{e}(\x_j;\x_i)\Big]\,. 
\end{equation}
Here $\gamma(b)$, depending on the common patch reactivity $b$, is
defined by
\begin{equation}\label{optim:main_3}
  \gamma(b)\equiv \frac{2E_{+}}{C^2} = -\left( \frac{
    \int_{\mathbb{D}} \int_{\mathbb{D}} q(|\y|;b) q(|\y^{\prime}|;b)\,
    \log|\y-\y^{\prime}| \, d\y \, d\y^{\prime}}{
    \left(\int_{\mathbb{D}} q(|\y|;b)\, d\y\right)^2} \right) \,,
\end{equation}
\esub
where $\mathbb{D}$ is the unit disk centered at the origin. By recalling
the limiting asymptotics for $C$ and $E_{+}$ in (\ref{berg:Cj_asy}) and
(\ref{berg:Ej_asy}), respectively, we have
\bsub \label{optim:gamma}
\begin{equation}\label{optim:gamma_low}
  \gamma(b)\sim \frac{1}{4} \,, \quad \mbox{as} \quad b\to 0 \,; \qquad
  \gamma(b) \sim \frac{3}{2}-\log{4}\approx 0.114 \,, \quad \mbox{as} \quad
  b \to \infty  \,.
\end{equation}
The numerical evaluation of $\gamma(b)$ is discussed in
\S\ref{sec:numerics_consts}.  A heuristic approximation of $\gamma(b)$
for all $b>0$ is given by the monotonically decreasing function
\cite{GrebenkovWard2026big}
\begin{equation}\label{optim:gamma_all}
  \gamma(b) \approx \frac{3}{2} -\log{4} +
  \frac{2}{\frac{1}{\log{2} - 5/8} + 5.17 \, b^{0.81}}\,.
\end{equation}
\esub This leads to the formulation of a {\em discrete optimization
problem.}

\noindent {\em Optimization: For a given domain $\Omega$ with smooth
  and closed boundary $\partial\Omega$, let {\clb us fix the number
    $N>1$ of circular boundary patches, their common radius $\eps\ll
    1$ and reactivity $b>0$.  We seek to determine the configuration
    $\lbrace{\x_1,\ldots,\x_N\rbrace} \subset \pa$ of the centers of
    the patches} that minimizes ${\mathcal P}$ in
  (\ref{optim:main_2}).  Such an optimal patch arrangement maximizes
  the capacitance $C_T$.}

\vspace{0.1cm} From (\ref{optim:main_2}), the optimal arrangement
depends on $\eps$, the patch reactivity through $\gamma(b)$, and the
domain shape via the Green's function $G_e$ and its regular part
$R_e$. Since $\gamma(b)>0$, but is monotone decreasing in $b$,
partially reactive patches increase the influence of the mean
curvature in the discrete energy, but to a lesser extent as $b$
increases. The curvature term and the inter-patch interaction effect
in (\ref{optim:main_2}) will balance in an optimum arrangement since
$G_{e}(\x_i;\x_j)\to +\infty$ as $\x_i\to\x_j$, which prevents the
optimally located patches from concentrating too close to minima of
the mean curvature.

We now apply (\ref{optim:main_2}) to a domain $\Omega$ that is a
perturbation of the unit sphere whose boundary is
$r=1+\mu F(\phi,\theta)$, where $\mu\ll 1$, $r=|\x|$, and with
$0\leq\phi<2\pi$ and $0\leq \theta\leq \pi$ being the azimuthal and
polar angles, respectively. For this geometry, the mean curvature of
$\partial\Omega$ for $\mu\ll 1$ is (noting our sign convention)
\begin{equation}\label{optim:hpert}
  {\mathcal H} \sim -1 +\frac{\mu}{2} \left( \Delta_{S^2} F + 2F \right)
  \quad \mbox{with} \quad \Delta_{S^2} F \equiv
  \frac{1}{\sin\theta} \left(\sin\theta \, \partial_{\theta}
    F\right)_{\theta} + \frac{1}{\sin^{2}\theta} \partial_{\phi\phi} F \,.
\end{equation}
We will consider two cases. {\bf Case I:} For a perturbed sphere
bulging out at the poles, we take $F=\cos^{2}\theta$, which vanishes
on the equator $\theta={\pi/2}$. In {\bf Case II} we take
$F=\sin^{2}\theta$, so that the perturbed sphere instead bulges out on the
equator. From (\ref{optim:hpert}) we calculate for these two cases
that
\begin{equation}\label{optim:pole_equator}
  {\mathcal H} = -1 + \mu \tilde{{\mathcal H}}(\theta) \,, \quad
  \mbox{where} \quad \tilde{{\mathcal H}}(\theta) \equiv \begin{cases}
    1-2\cos^{2}\theta \,, &\text{Case I}; \\
    2 \cos^{2}\theta \,, &\text{Case II}.
    \end{cases} 
\end{equation}
For our two choices of boundary deformations, the centers of the
patches on the boundary of the perturbed sphere are labeled as
$\x_j=(1+\mu F(\theta_j))\x_{j0}$ where
$\x_{j0}\equiv \left(\cos\phi_j \sin\theta_j,
  \sin\phi_j\sin\theta_j,\cos\theta_j \right)^{T}$ for
$j=1,\ldots,N$. We will choose the size of the boundary deformation to
satisfy $\mu={\mathcal O}\left({-1/\log\eps}\right)$ so that the
curvature and inter-patch interaction terms in (\ref{optim:main_2})
are of the same order. More specifically, we will set
$\mu={-\beta/\log\eps}$, where $\beta>0$ is independent of
$\eps$. Upon substituting (\ref{optim:pole_equator}) into
(\ref{optim:main_2}) we obtain with this choice of $\mu$ that
\begin{align}\label{optim:prelim_1}
 \nonumber {\mathcal P} &\equiv N\left( -\log\eps +
    \gamma(b)\right)\\
    & \qquad +
  \frac{\beta}{2} \sum_{j=1}^{N} \tilde{\mathcal H}(\theta_j) +
  4\pi \sum_{i=1}^N\Big[R_{e}(\x_{i0}) +
  \sum_{\substack{j=1\\j\neq i}}^N G_{e}(\x_{j0};\x_{i0})\Big] +
  {\mathcal O}\left(\frac{-1}{\log\eps}\right) .
\end{align}
Finally, we explicitly evaluate the surface Neumann Green's function in
(\ref{optim:prelim_1}) on the unit sphere using
(\ref{berg:gs_loc1}) of Appendix \ref{app:explicit}. In this way, we
conclude that
\bsub \label{optim:example}
\begin{equation}\label{optim:example_1}
  {\mathcal P}(\x_1,\ldots,\x_N) \sim N\left(\log\left(\frac{\eps}{2}\right)
    -\gamma(b) \right) + {\mathcal P}_{0}(\x_{10},\ldots,\x_{N0}) \,,
\end{equation}
where, with the two choices of $\tilde{\mathcal H}(\theta_j)$ given in
(\ref{optim:pole_equator}), we have
\begin{equation}\label{optim:example_2}
  {\mathcal P}_{0}\equiv   \frac{\beta}{2} \sum_{j=1}^{N}
  \tilde{\mathcal H}(\theta_j) +\sum_{i=1}^{N} \sum_{j=i+1}^{N} \left(
    \frac{4}{|\x_{i0}-\x_{j0}|}-2\log\left(1 +
      \frac{2}{|\x_{i0}-\x_{j0}|}\right)\right) \,.
\end{equation}
\esub

In Figs.~\ref{fig:optim_case1}-\ref{fig:optim_case2} we present
results obtained from the \textsc{Matlab} optimization routine {\tt
fmincon} called with the {\tt interior-point} algorithm.  The
objective function is given by ${\mathcal P}_{0}$ defined in
\eqref{optim:example_2} with the two choices of $\tilde{\mathcal
H}(\theta_j)$ described in \eqref{optim:pole_equator}.  We consider
$N=1,\ldots,12$ patches and initialize the algorithm at points with
angular coordinates chosen with uniform random distribution in
$(\phi,\theta)\in[0,2\pi)\times[0,\pi)$.

In case I where the choice $F = \cos^2\theta$ yields an elongated
domain, we see in Fig.~\ref{fig:optim_case1} that the optimizing
configurations have a tendency to cluster patches around the poles of
the domain.  In case II the choice $F = \sin^2\theta$ results in a
domain which is more bulged in the equatorial region.  In
Fig.~\ref{fig:optim_case2}, we observe a trend for the optimizing
patch locations to congregate around the equator.  For $N=1,2,3,4$, we
observe that the optimizers are coplanar on $z=0$.  For $N\geq5$, the
optimizing configurations occupy patch centers with $z\neq0$.  Similar
geometric bifurcations have been observed in related optimization
problems for narrow capture in two \cite{TargetSearch2024} and three
dimensions \cite{lindsay20263D_GFun}.

\begin{figure}
    \centering
    \includegraphics[width =0.15\textwidth]{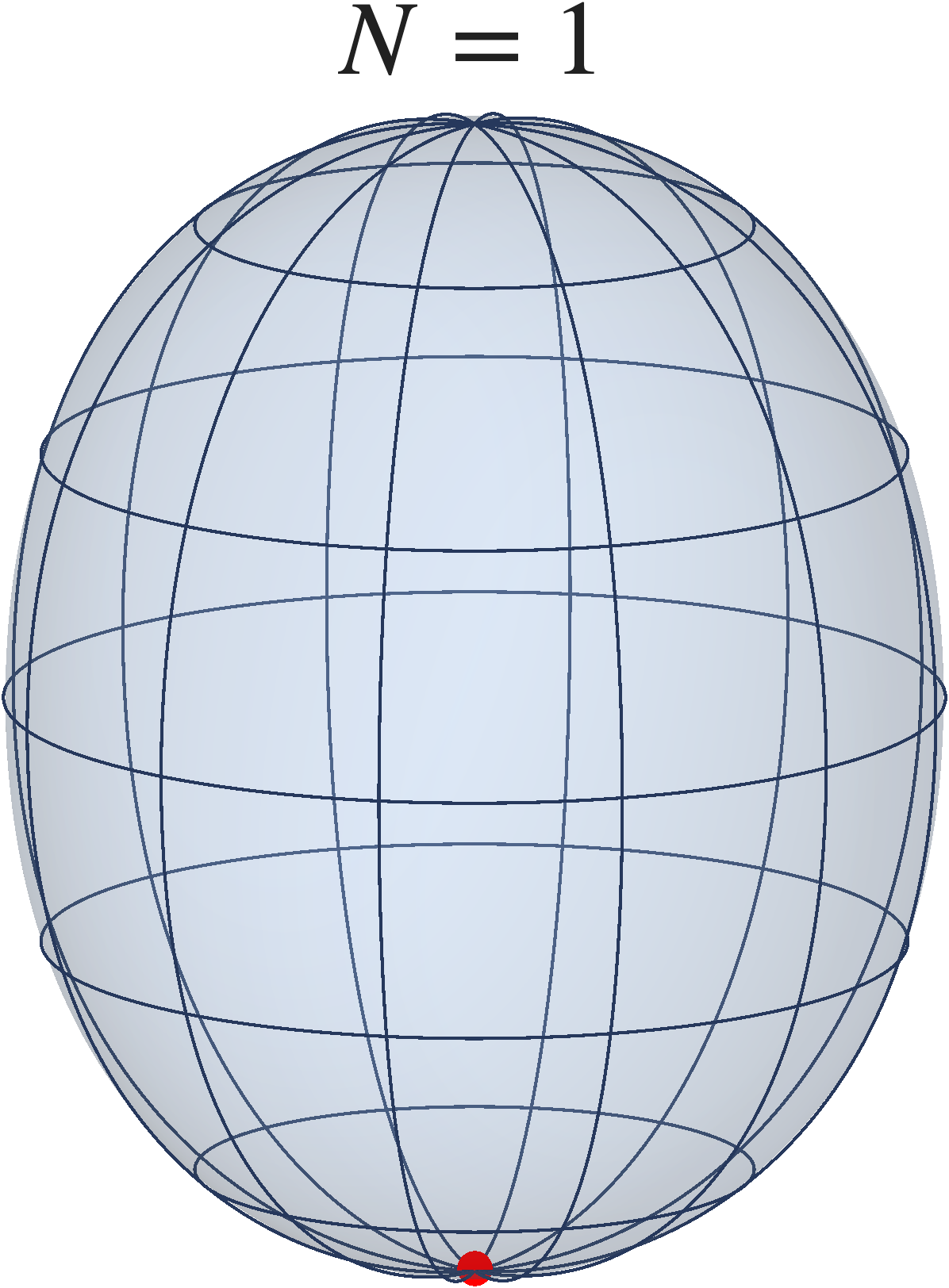}
    \includegraphics[width =0.15\textwidth]{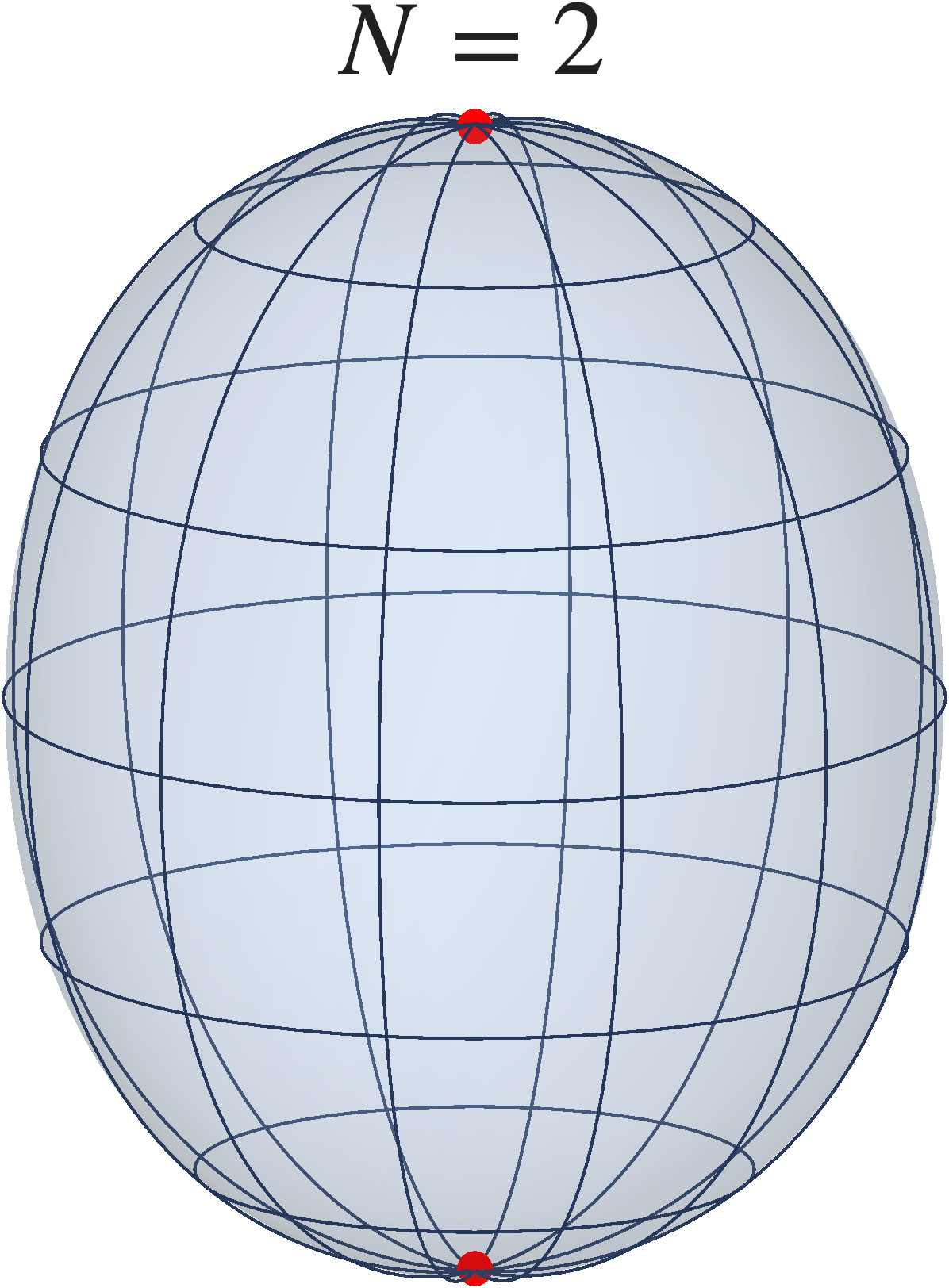}
    \includegraphics[width =0.15\textwidth]{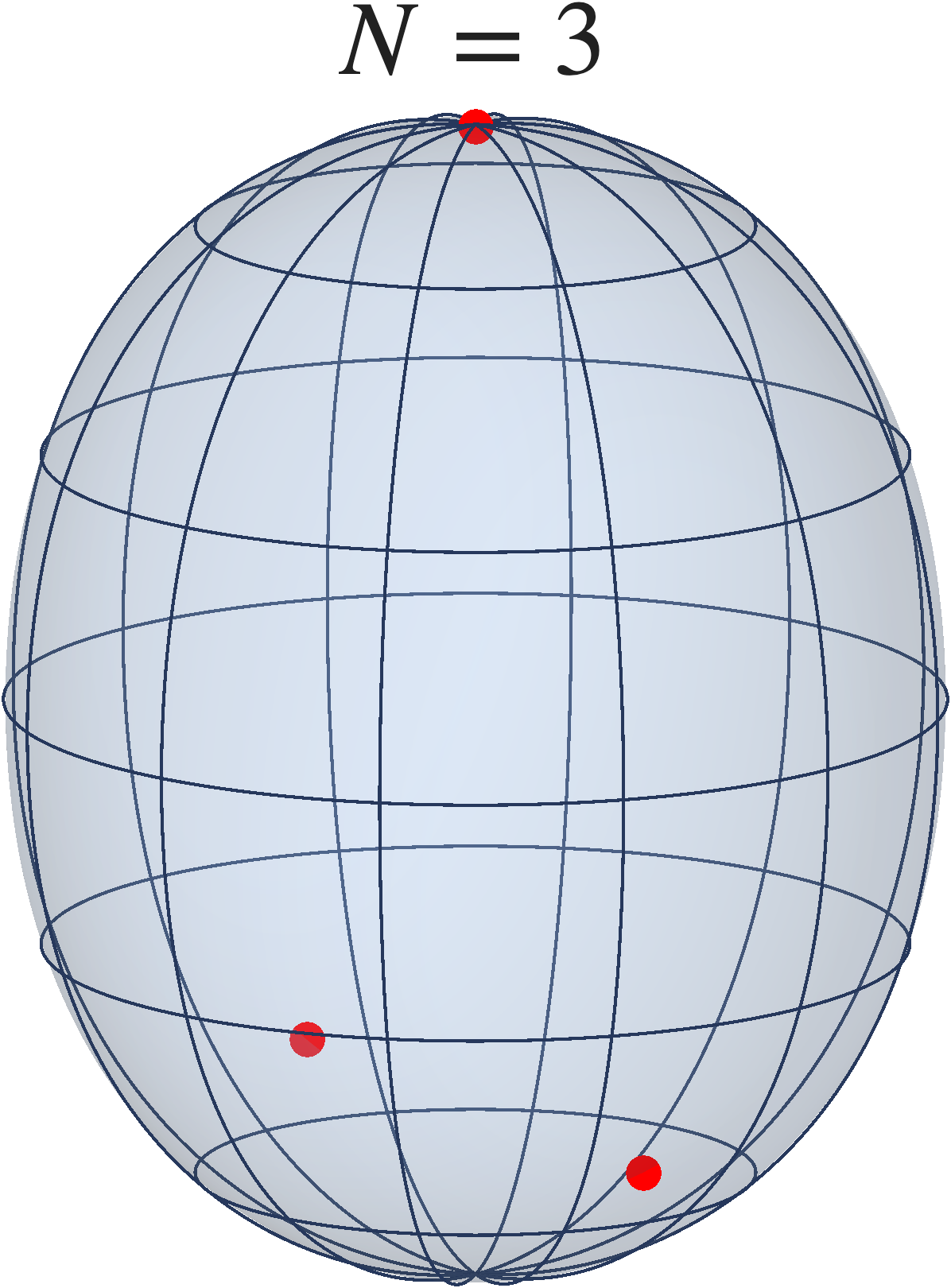}
    \includegraphics[width =0.15\textwidth]{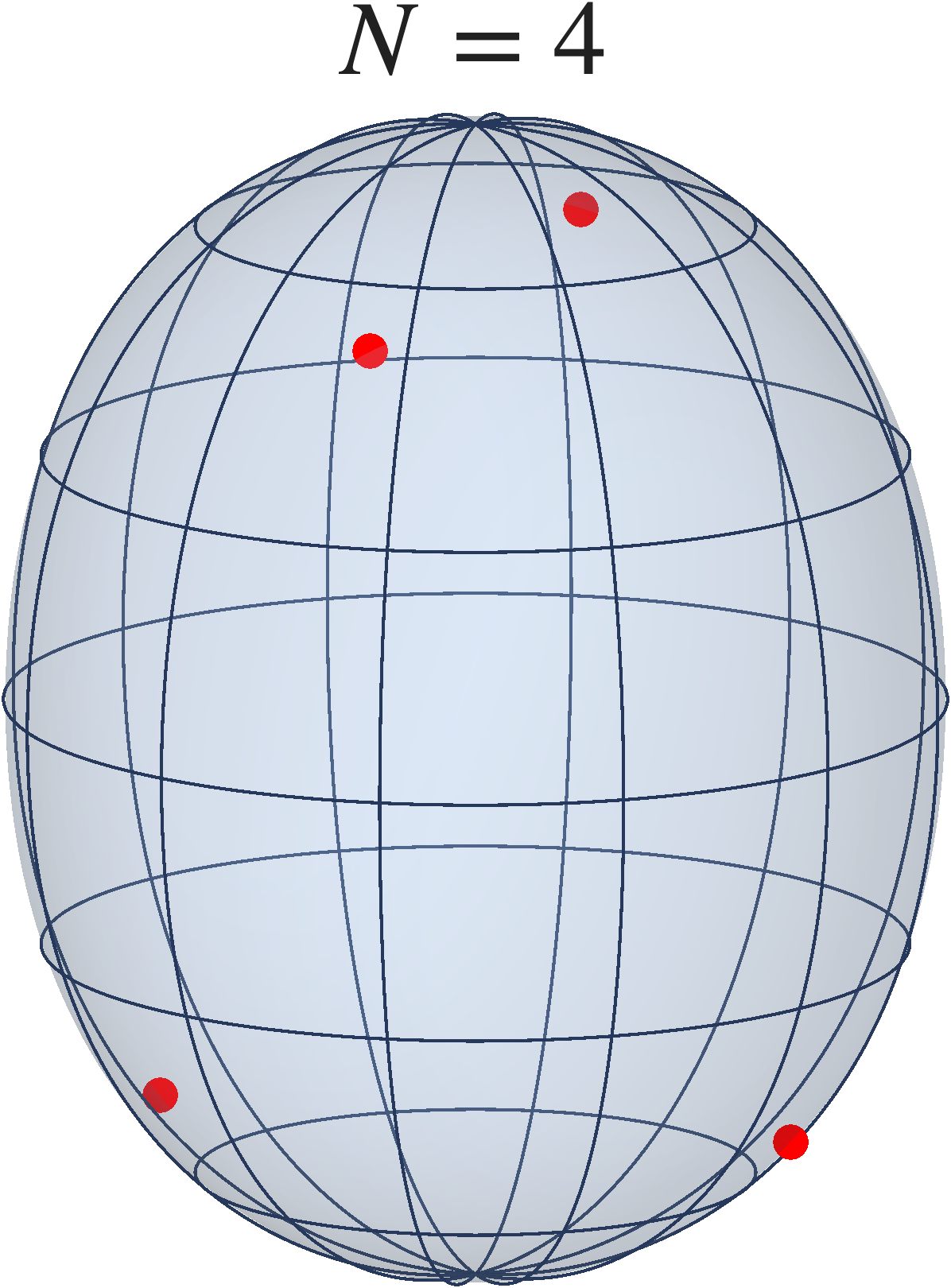}
    \includegraphics[width =0.15\textwidth]{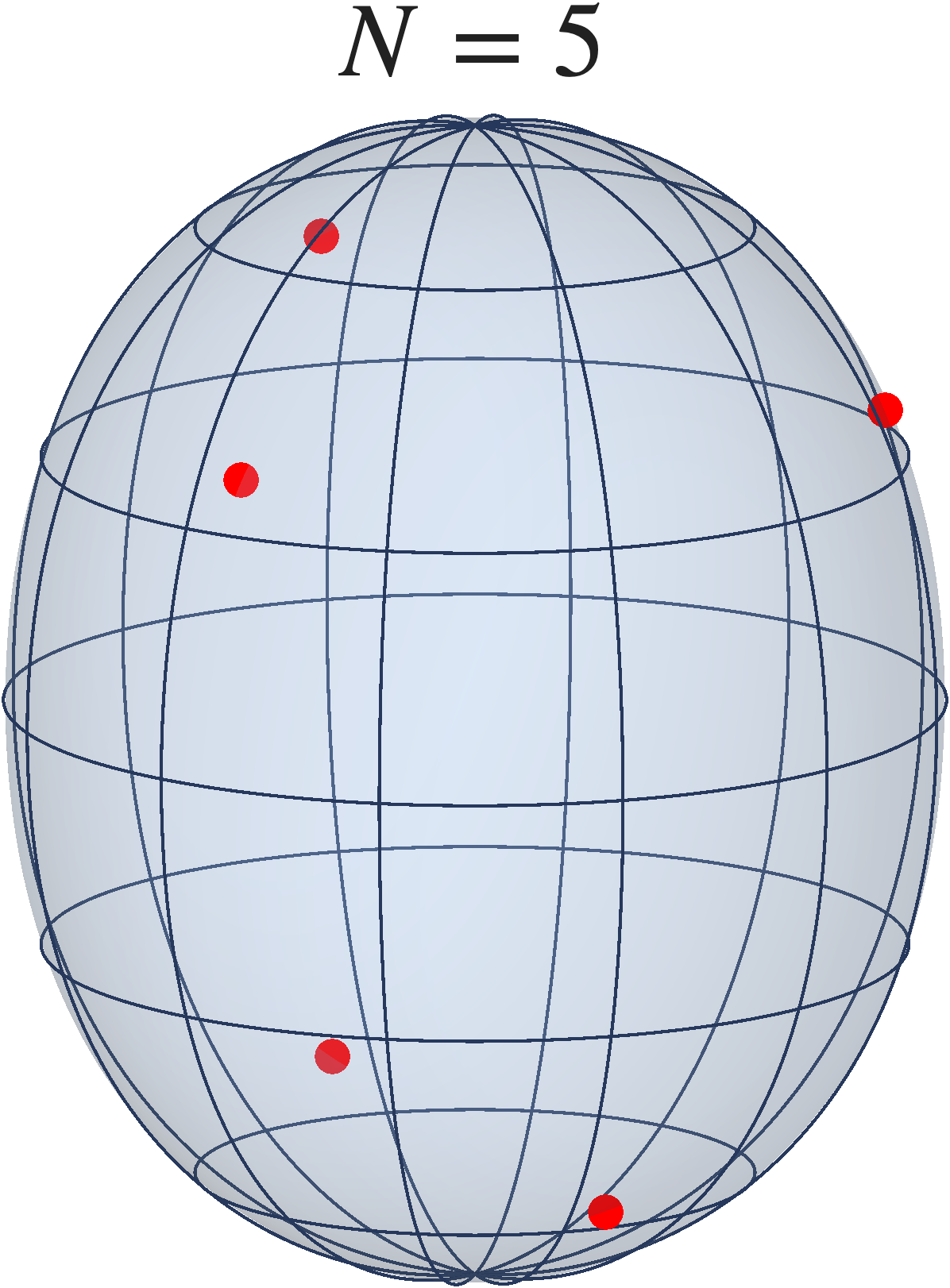}
    \includegraphics[width =0.15\textwidth]{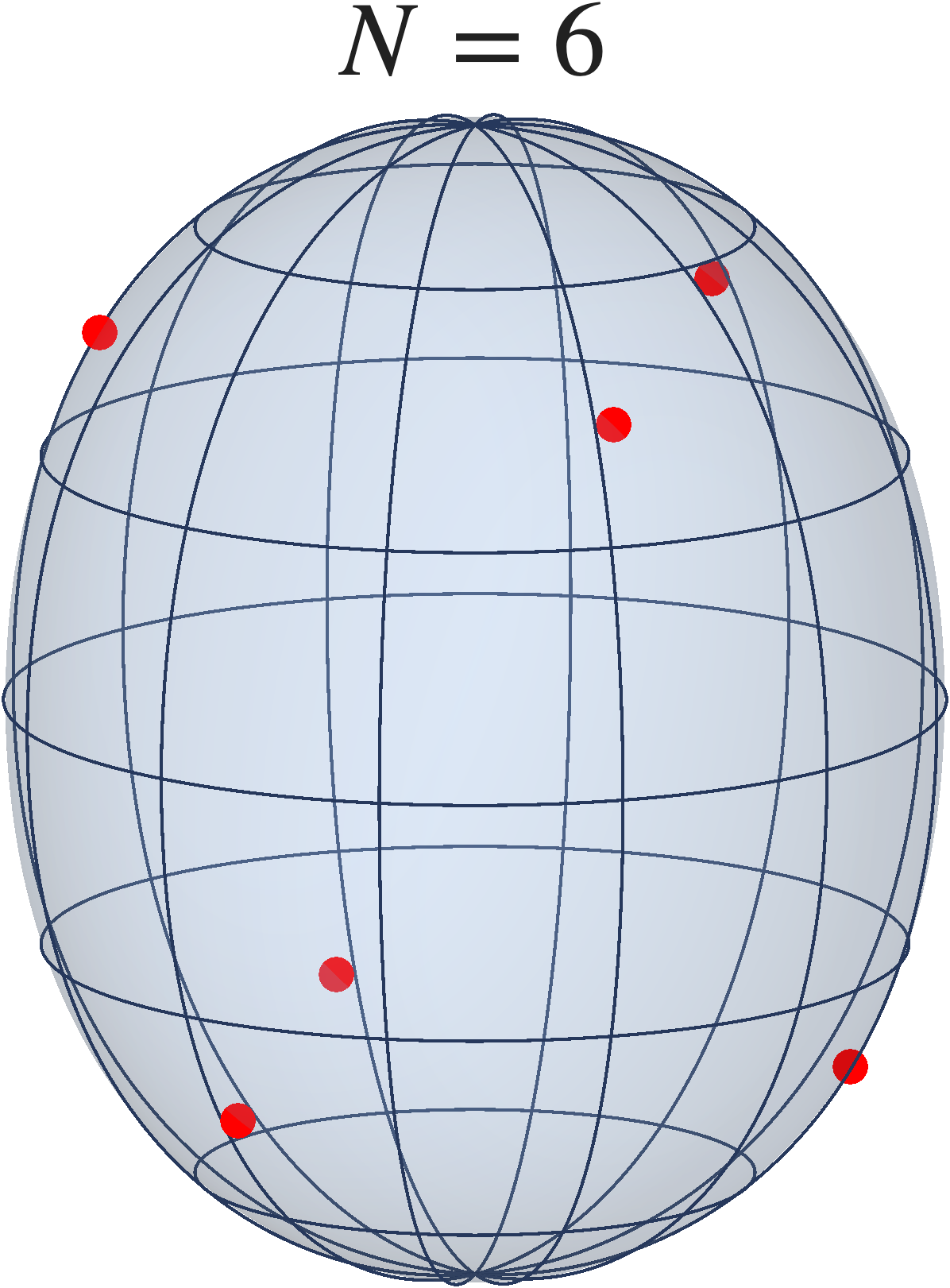}\\[5pt]
    \includegraphics[width =0.15\textwidth]{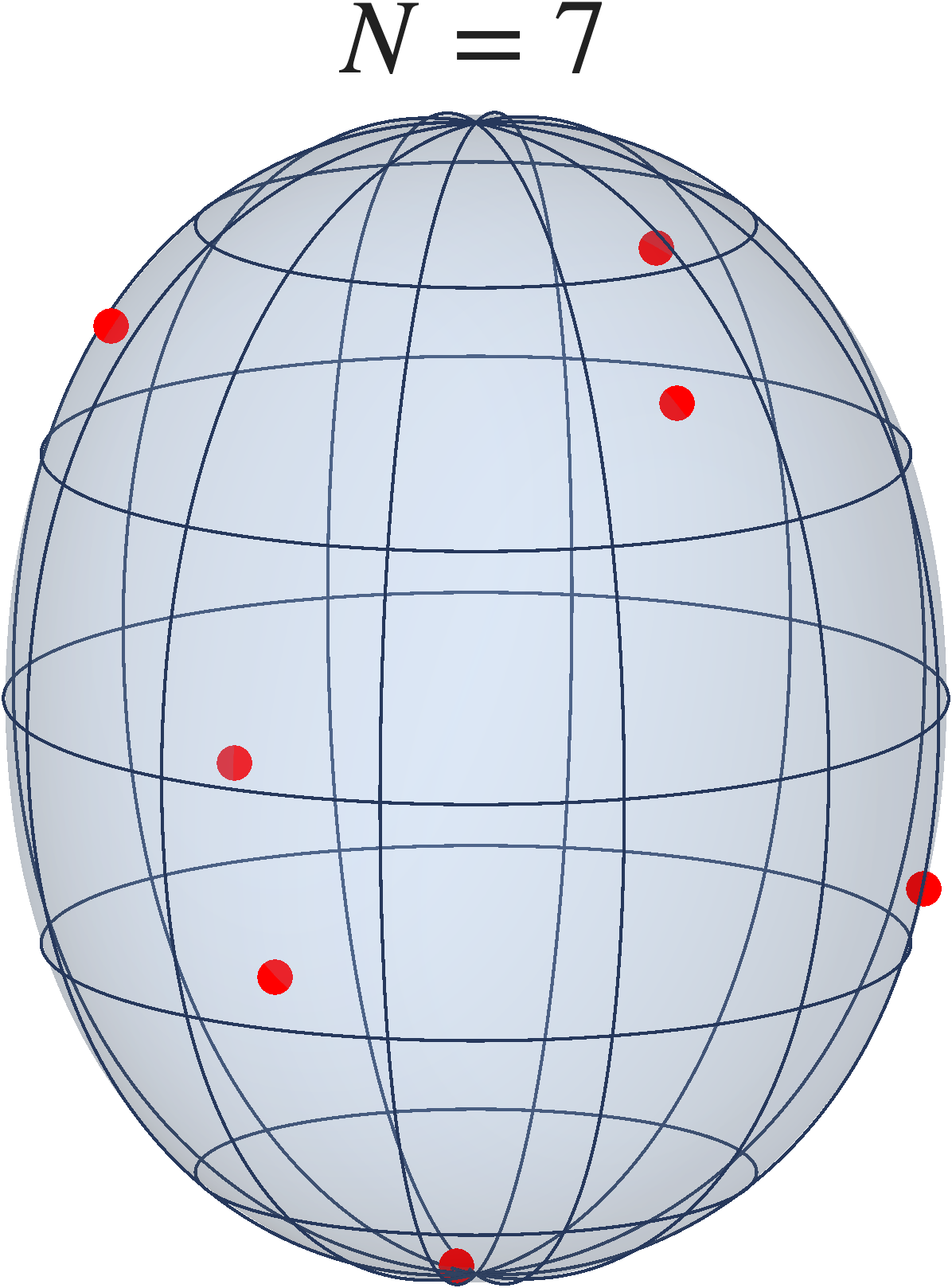}
    \includegraphics[width =0.15\textwidth]{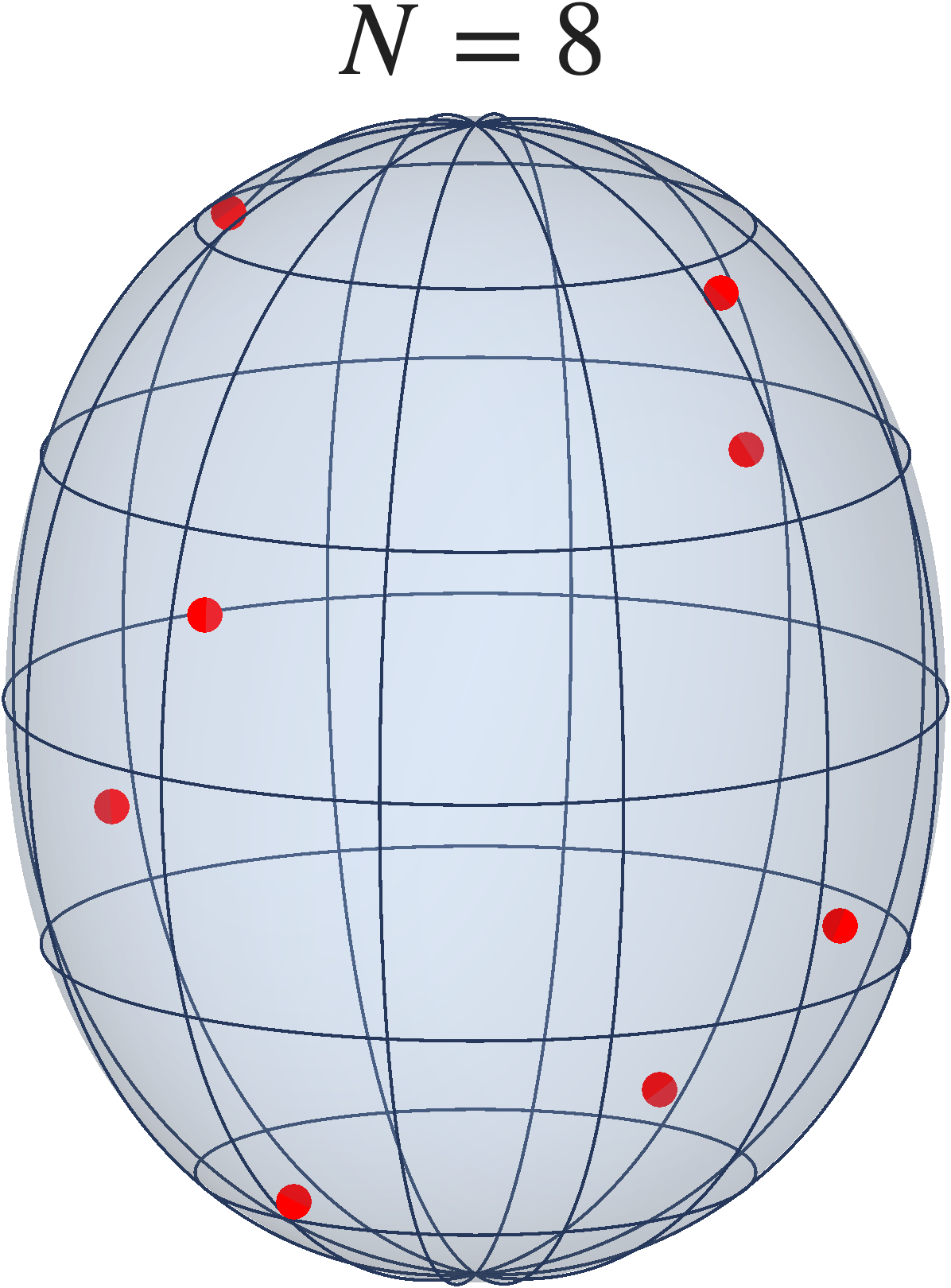}
    \includegraphics[width =0.15\textwidth]{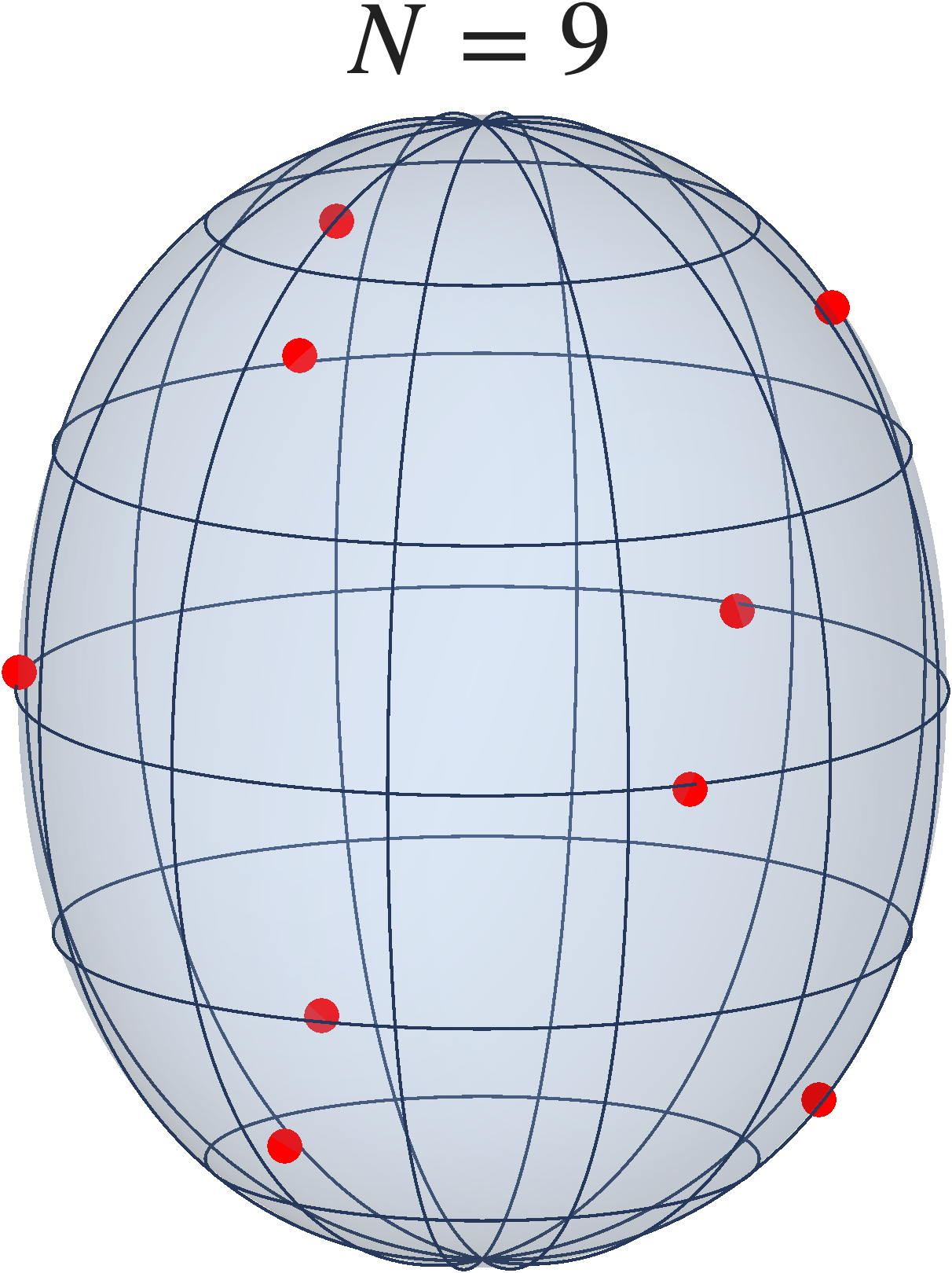}
    \includegraphics[width =0.15\textwidth]{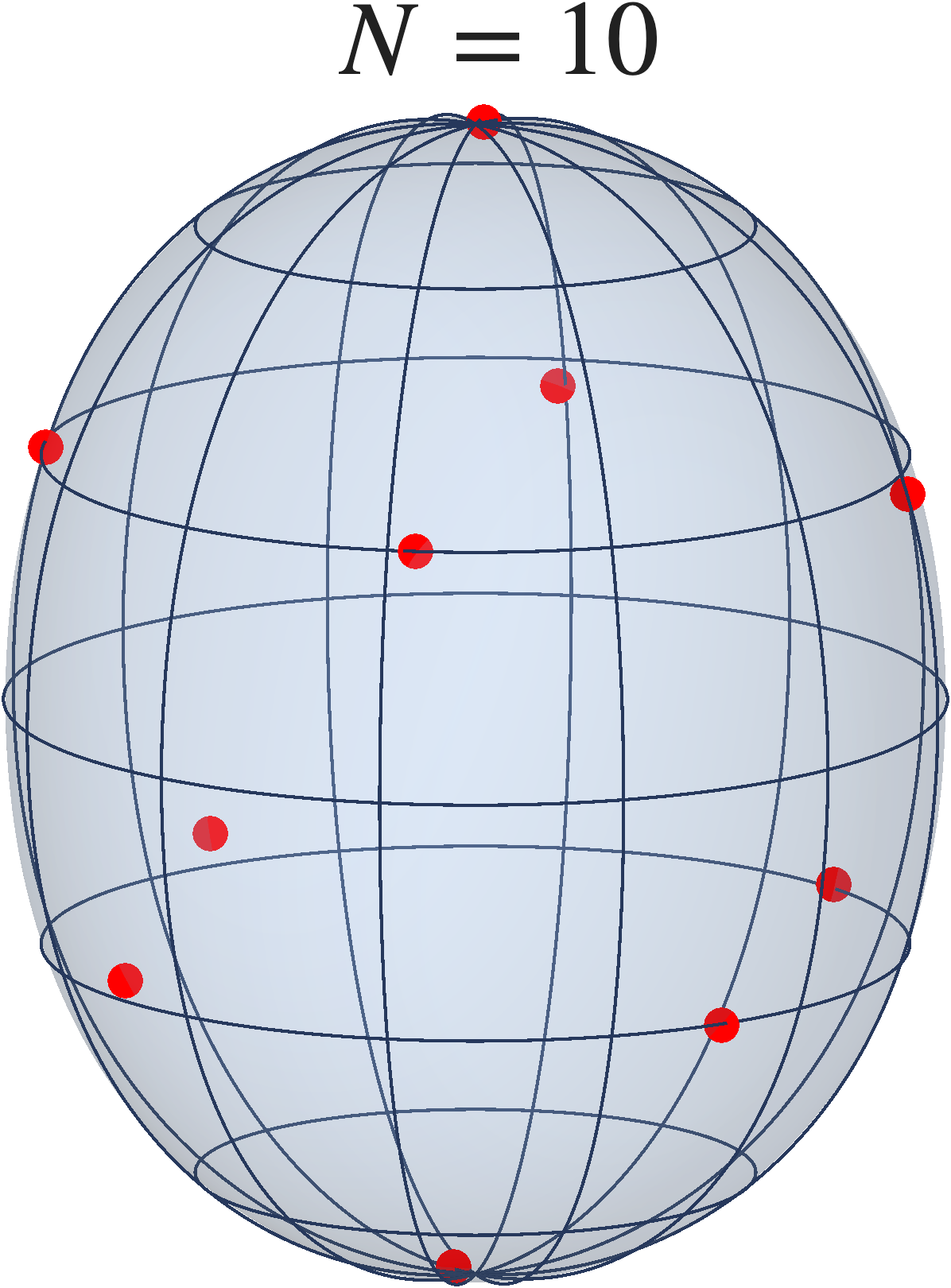}
    \includegraphics[width =0.15\textwidth]{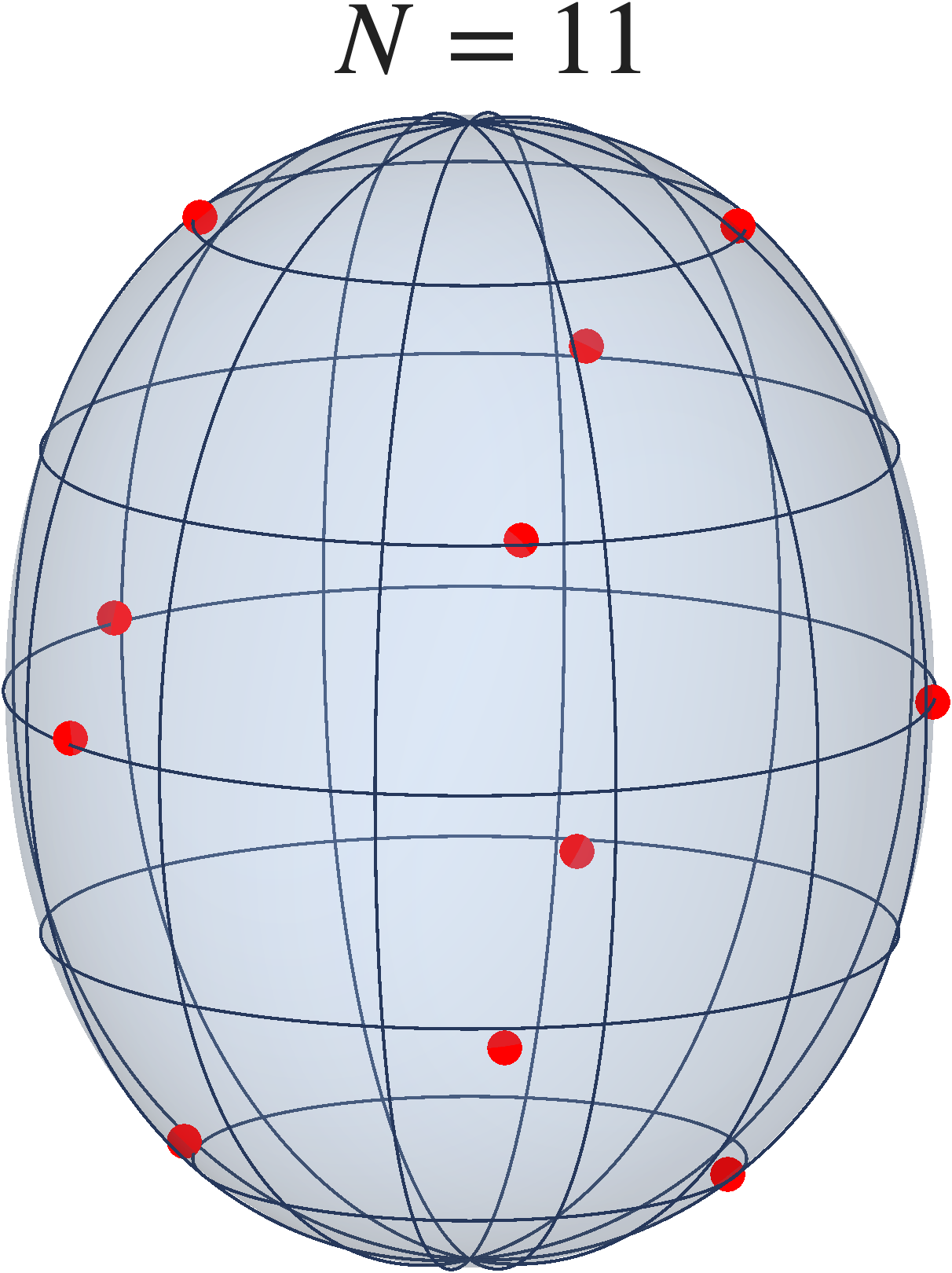}
    \includegraphics[width =0.15\textwidth]{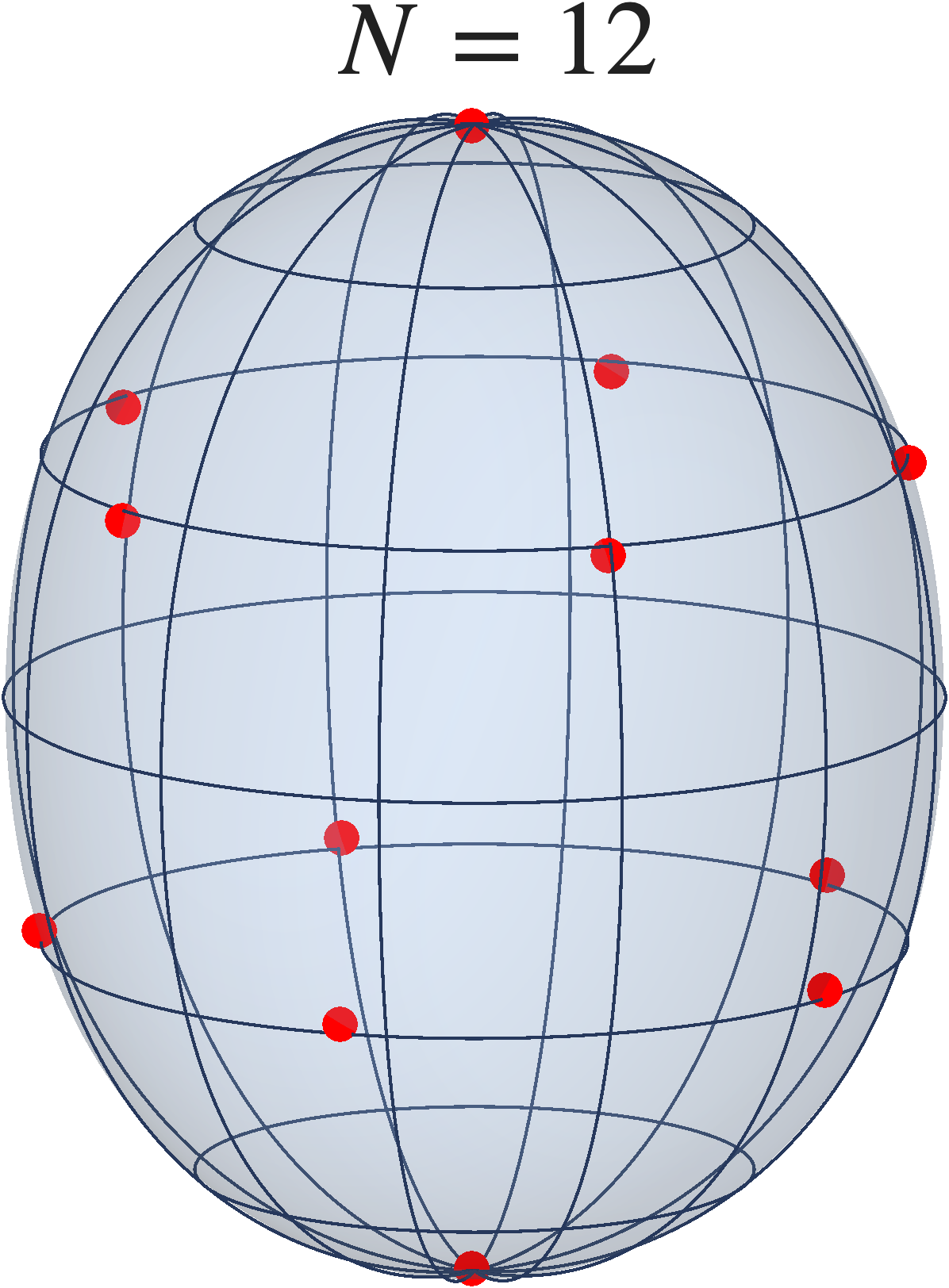}
\caption{
Optimization of the capacitance $C_T$ for $N=1,\ldots,12$ surface
patches in the Berg-Purcell problem, Case I.  We consider {\clb a
nearly spherical boundary} defined by radial coordinate $r = 1+
\mu\cos^2\theta$ for $\mu = 0.25$.  For $N=1,\ldots,12$ points, we
numerically calculate the points $\{\x_j\}_{j=1}^N$ that minimize the
discrete energy $\eqref{optim:example}$.}
\label{fig:optim_case1}
\end{figure}

\begin{figure}
    \centering
    \includegraphics[width =0.15\textwidth]{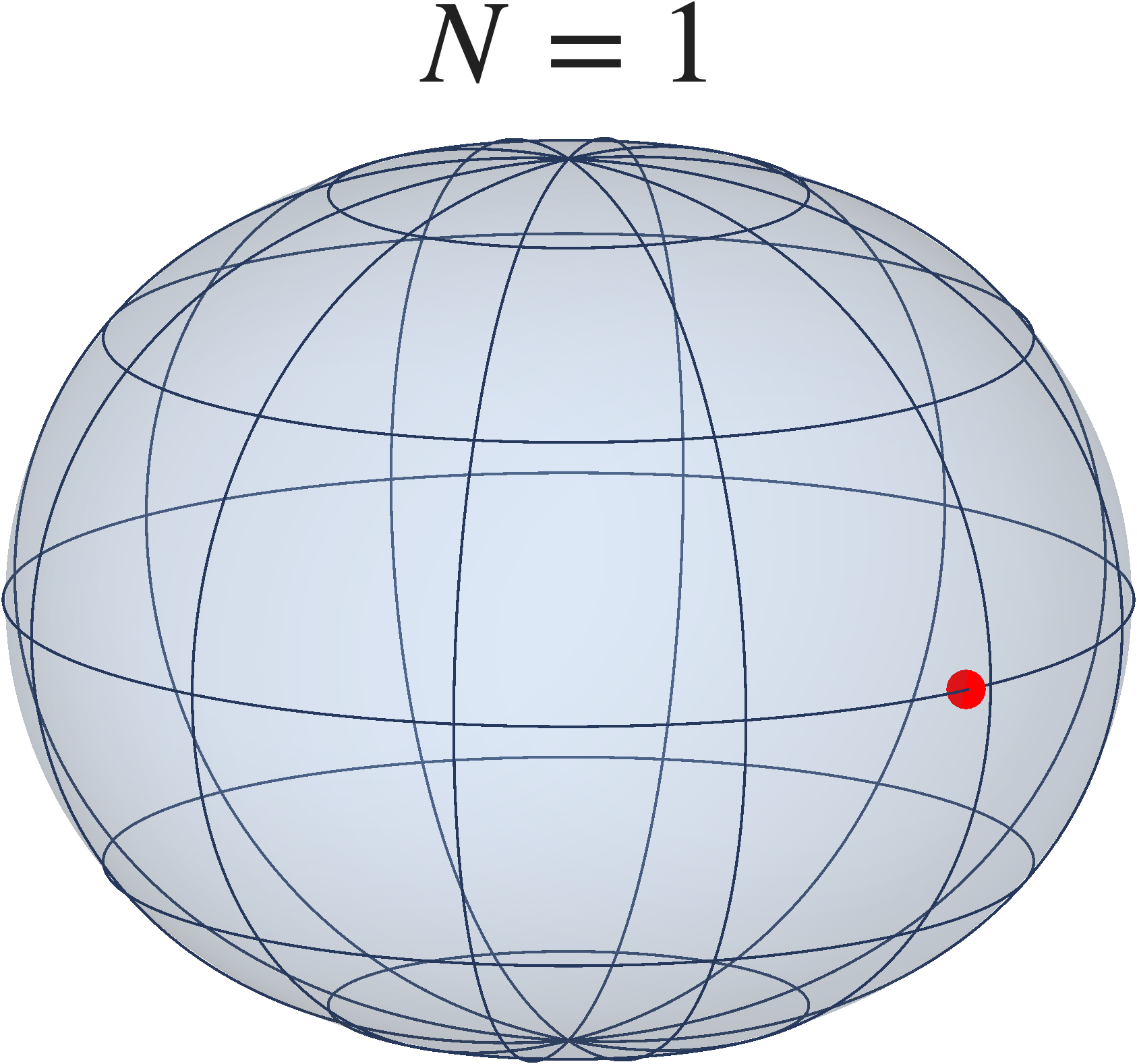}
    \includegraphics[width =0.15\textwidth]{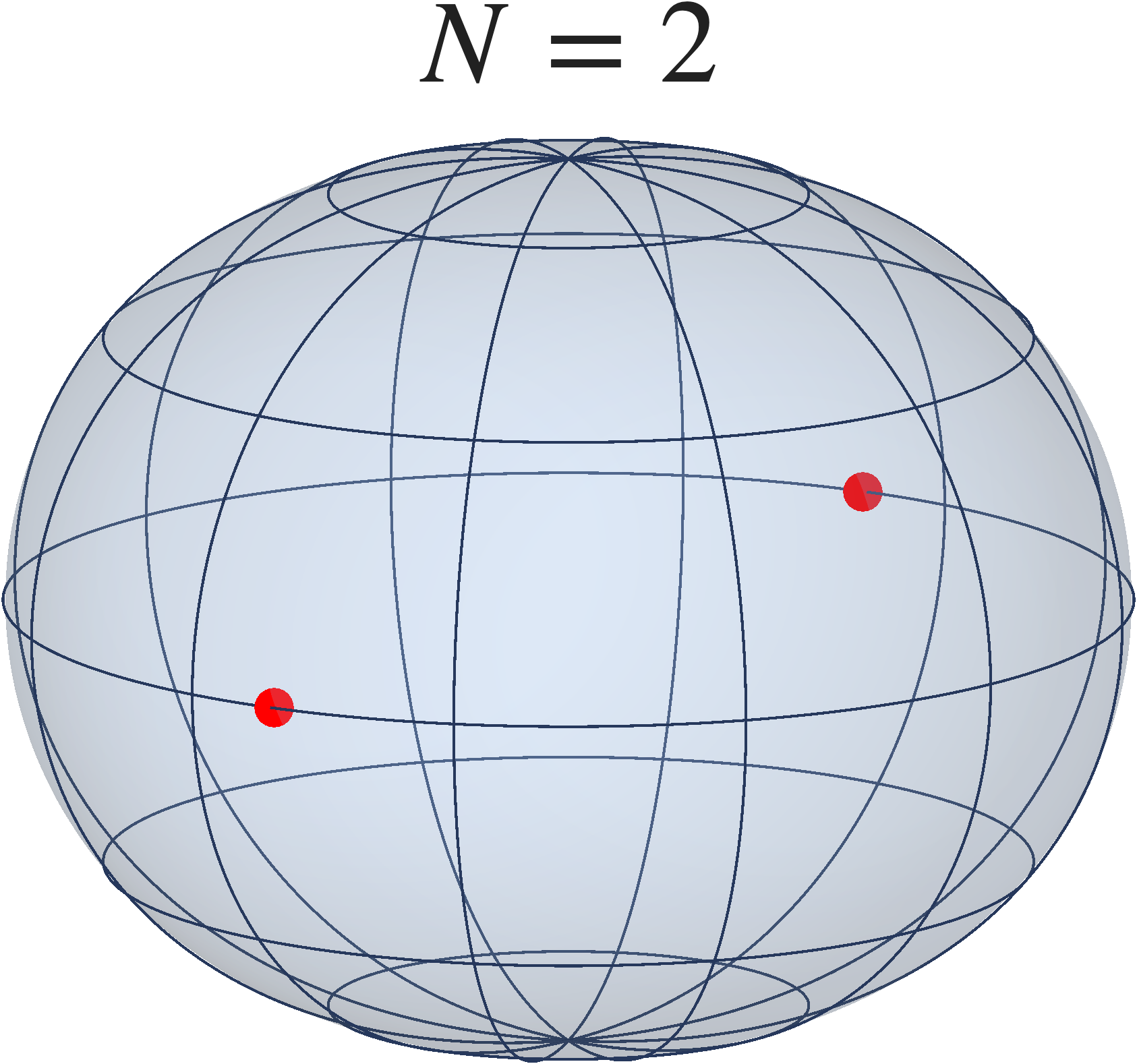}
    \includegraphics[width =0.15\textwidth]{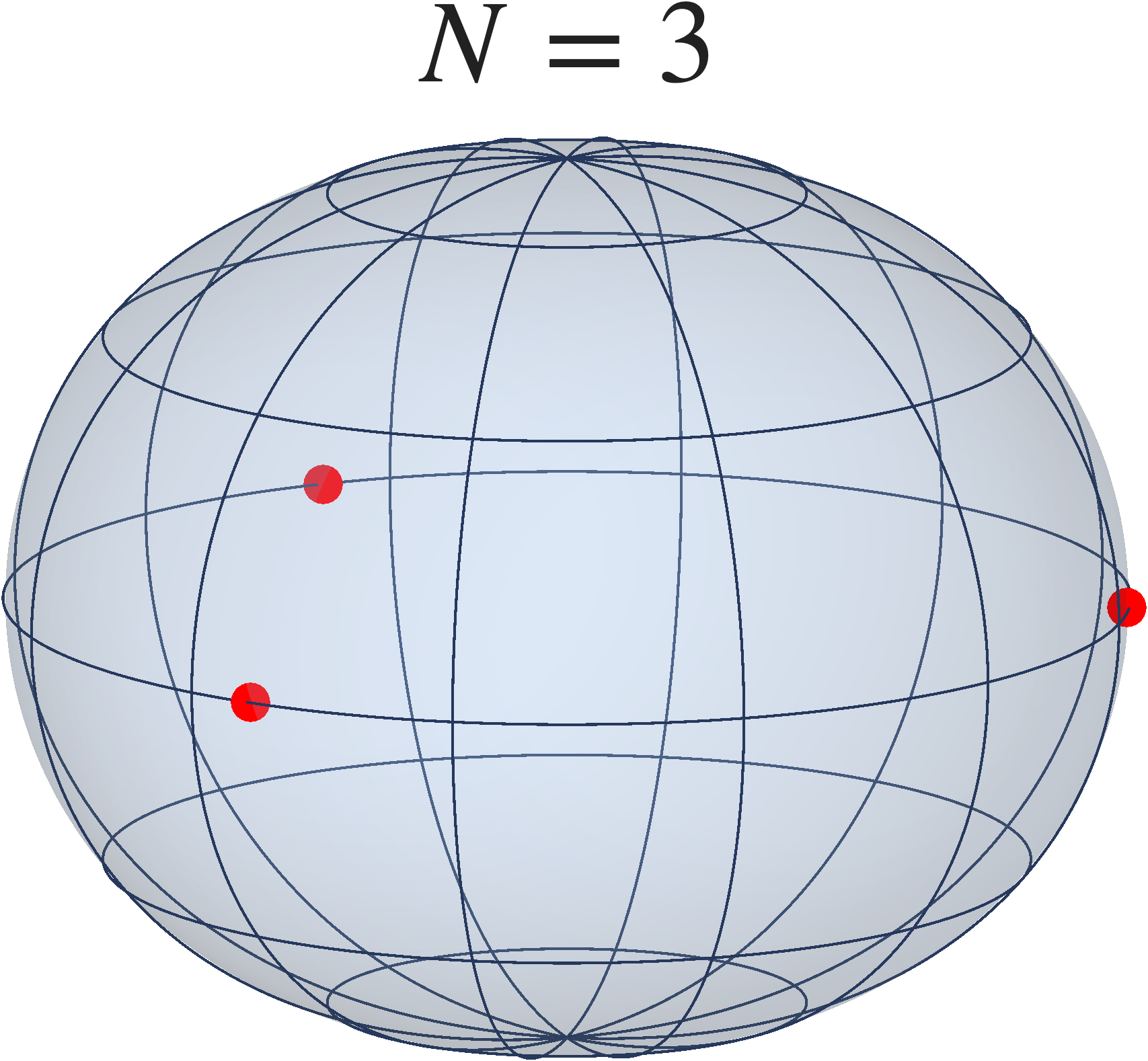}
    \includegraphics[width =0.15\textwidth]{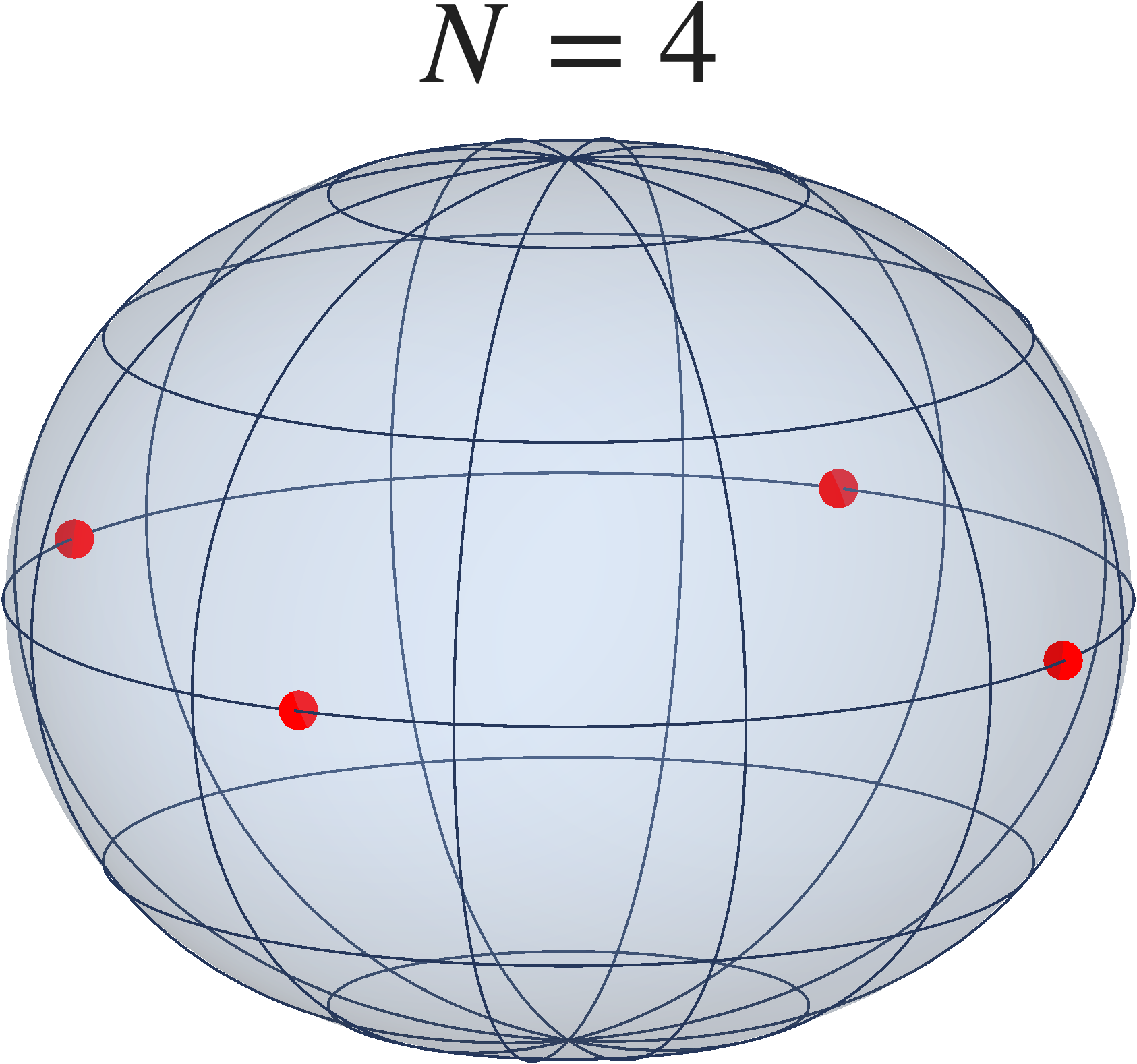}
    \includegraphics[width =0.15\textwidth]{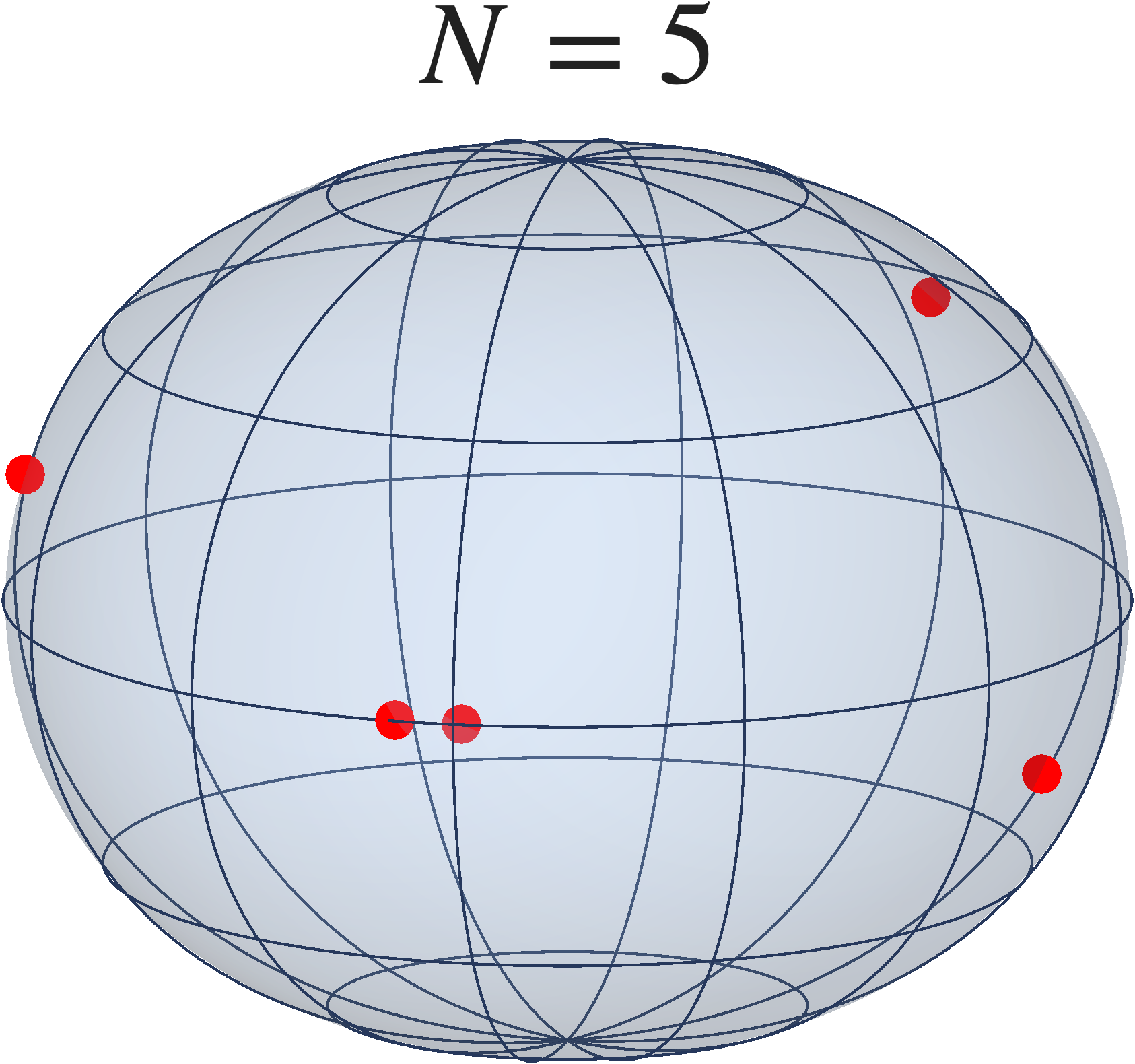}
    \includegraphics[width =0.15\textwidth]{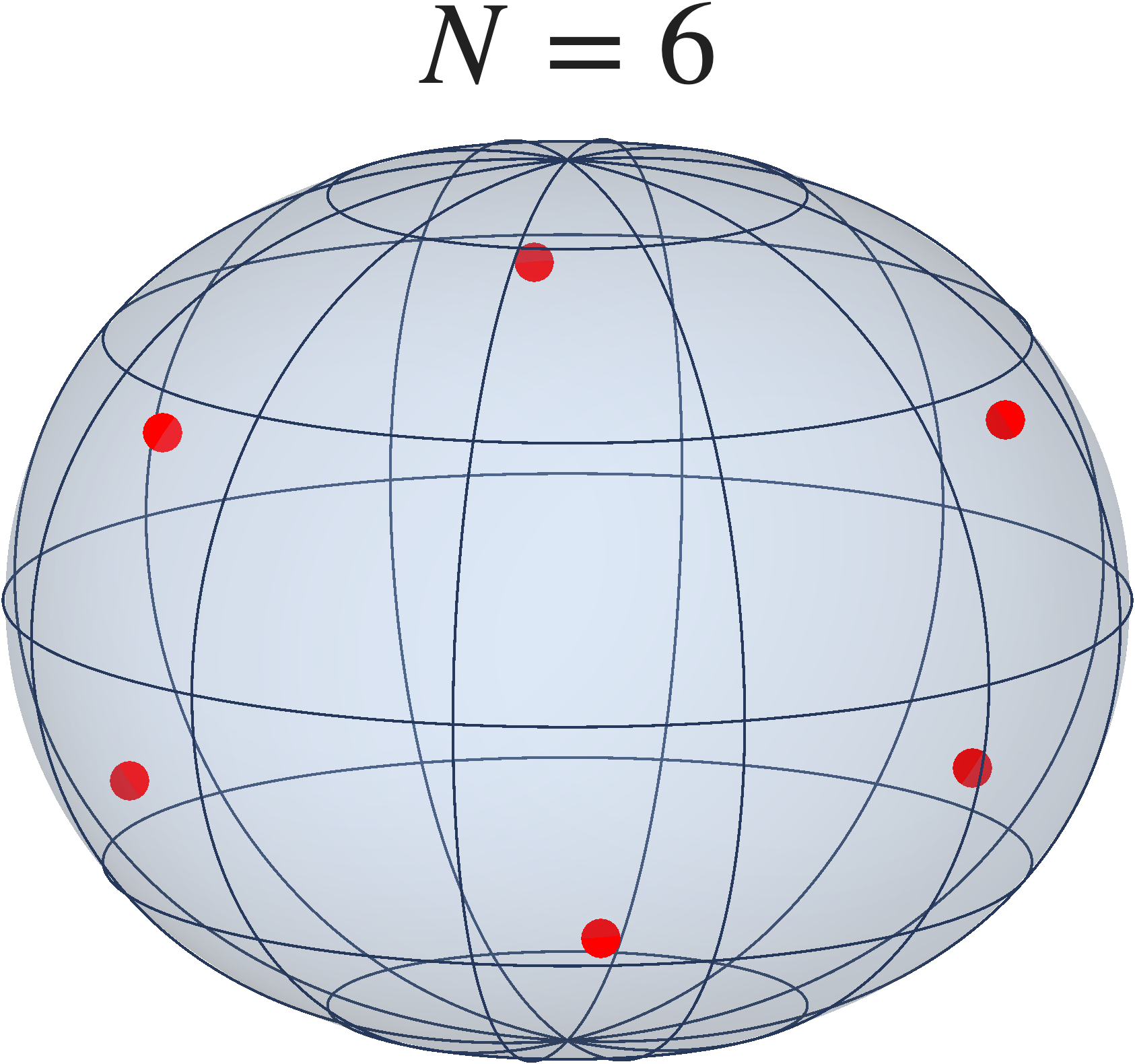}\\[5pt]
    \includegraphics[width =0.15\textwidth]{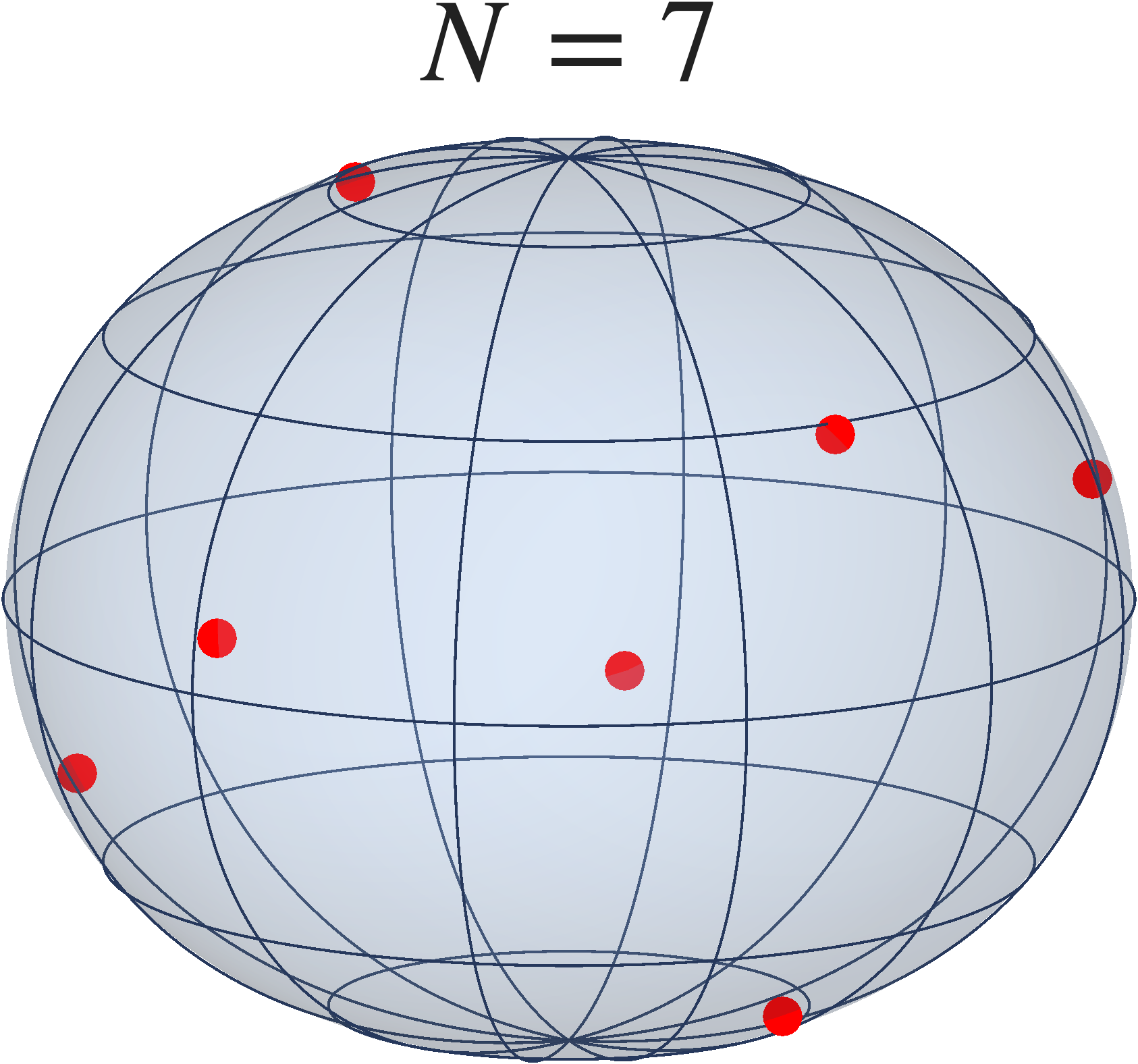}
    \includegraphics[width =0.15\textwidth]{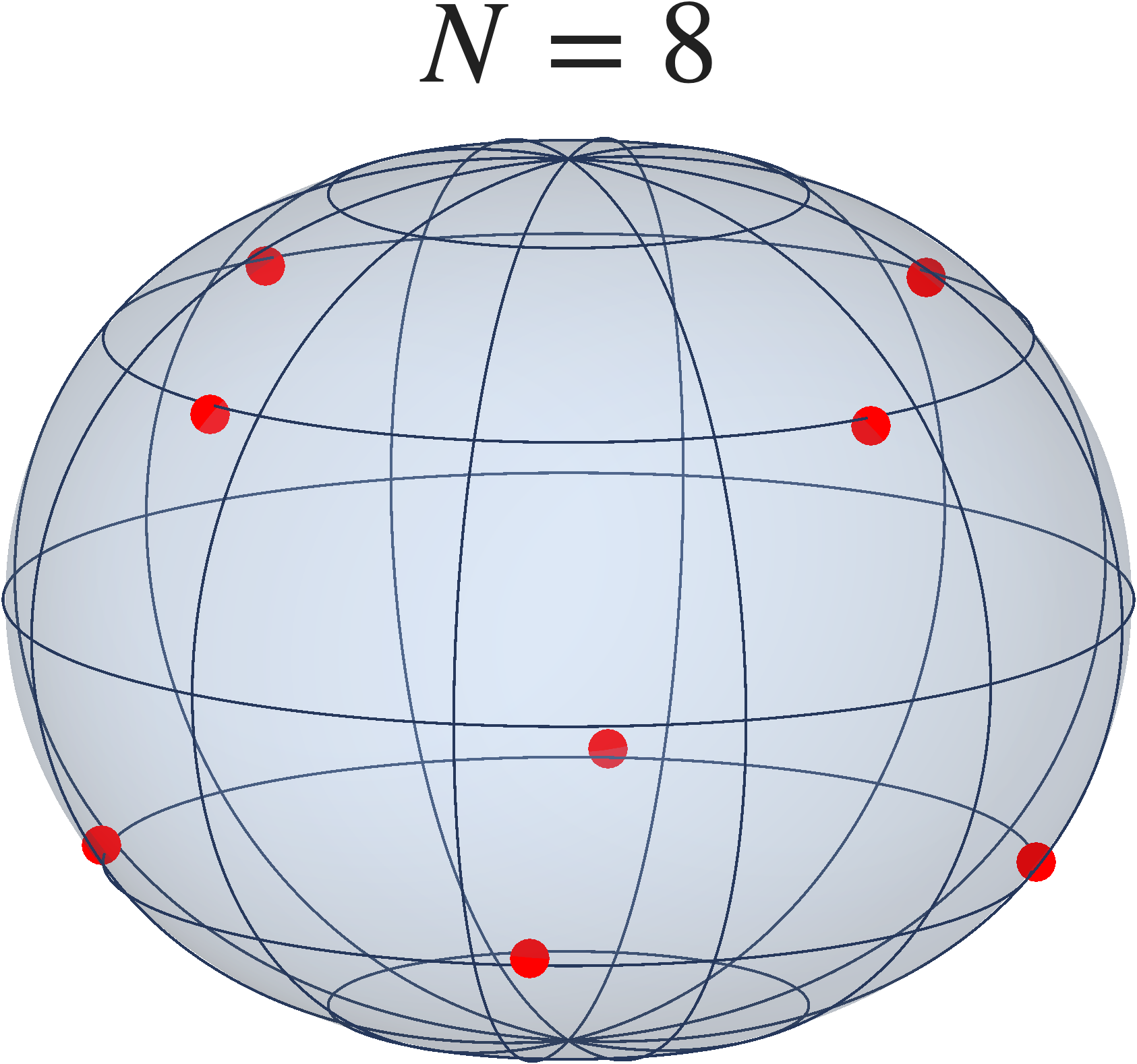}
    \includegraphics[width =0.15\textwidth]{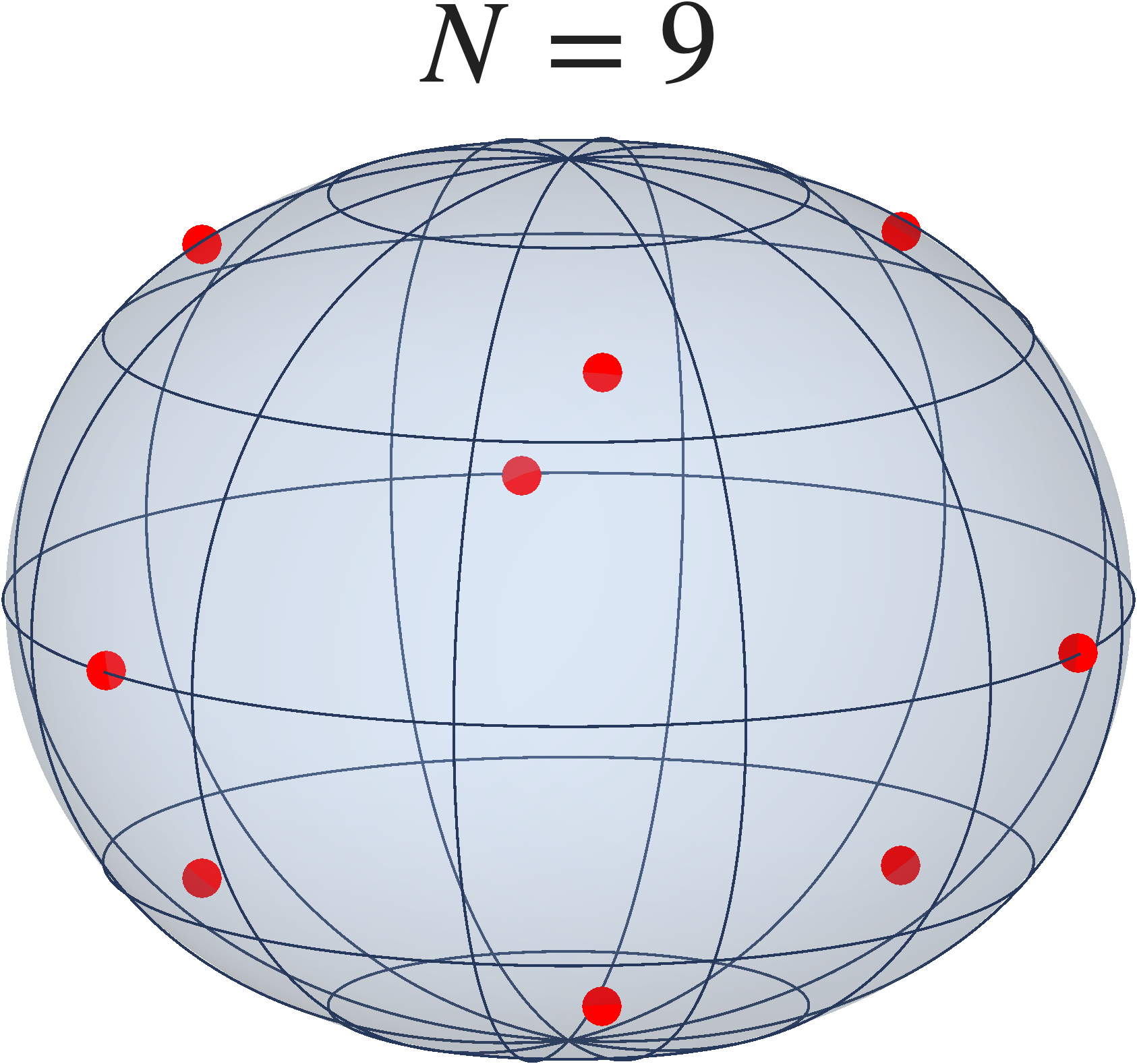}
    \includegraphics[width =0.15\textwidth]{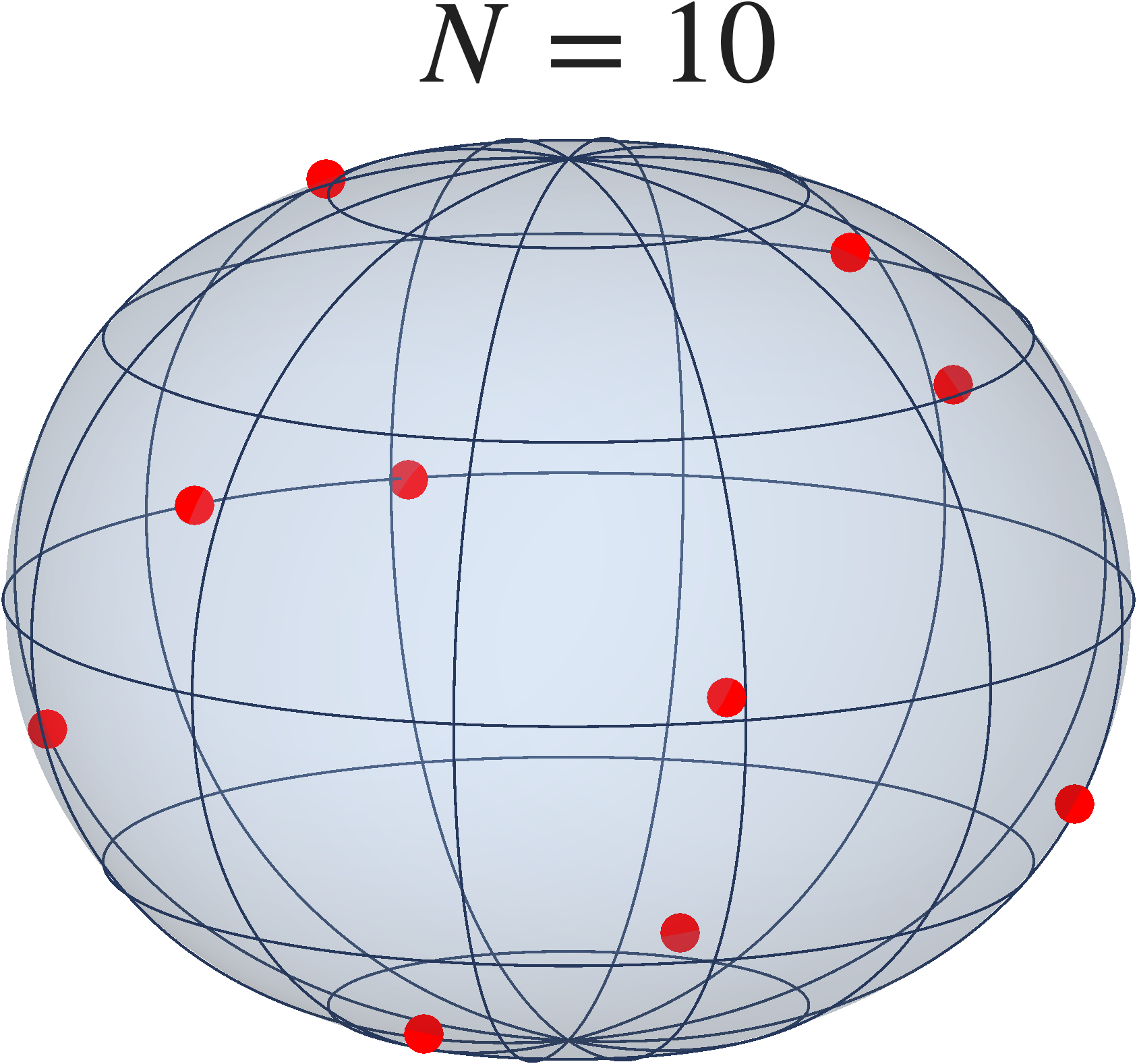}
    \includegraphics[width =0.15\textwidth]{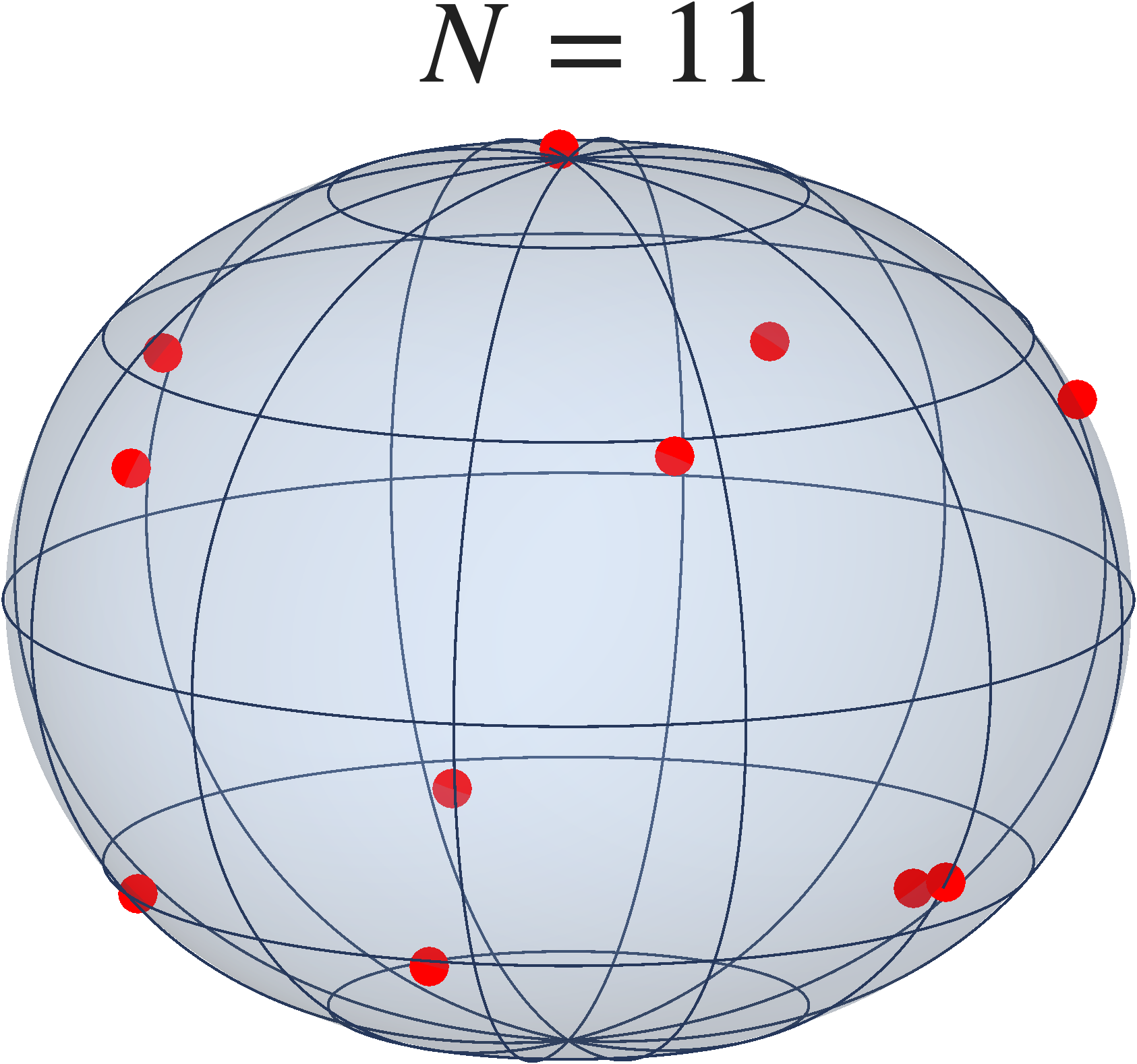}
    \includegraphics[width =0.15\textwidth]{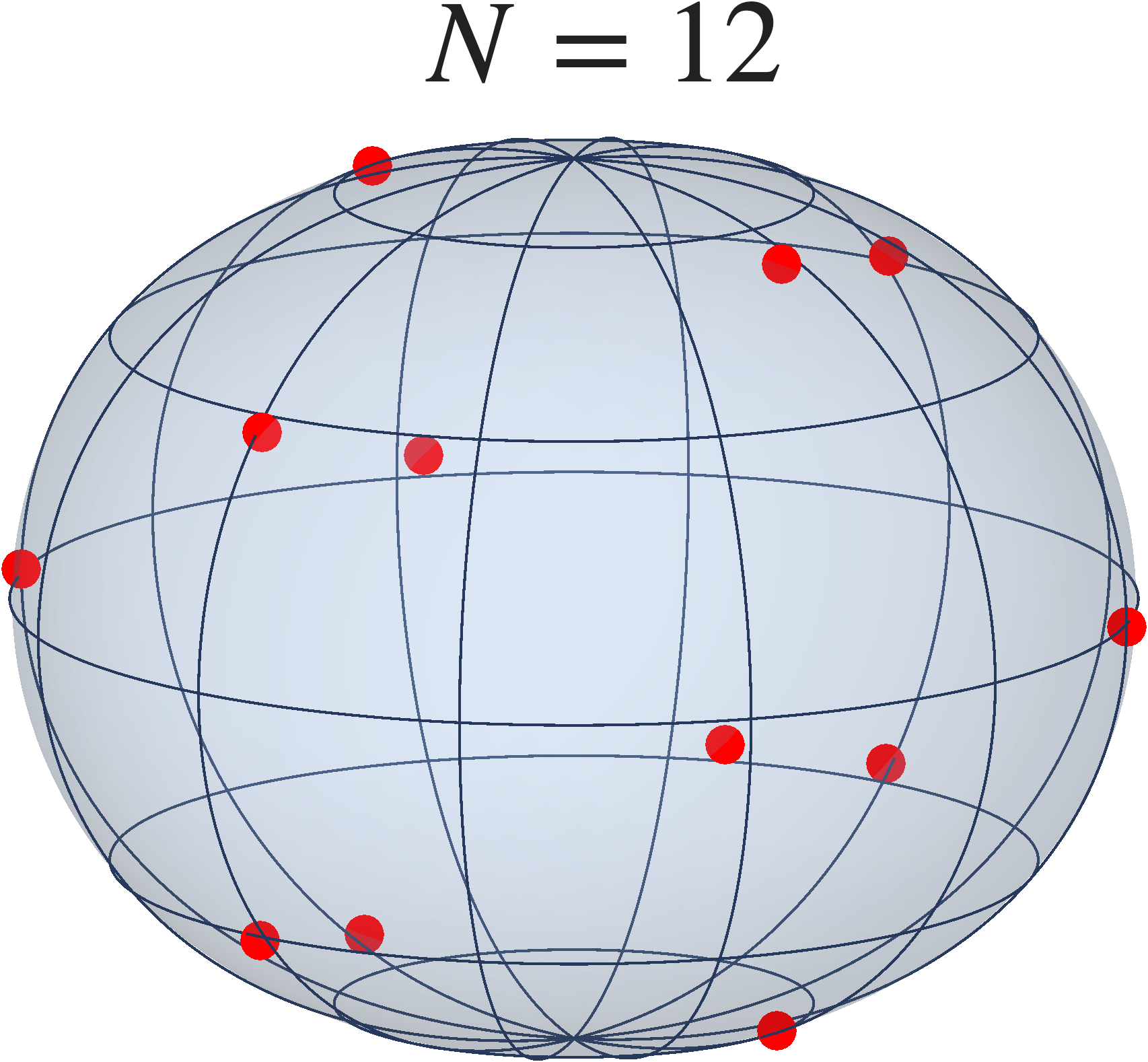}
\caption{
Optimization of the capacitance $C_T$ for $N=1,\ldots,12$ surface
patches in the Berg-Purcell problem, Case II.  We consider  {\clb a
nearly spherical boundary} defined by radial coordinate $r = 1+
\mu\sin^2\theta$ for $\mu = 0.25$.  For $N=1,\ldots,12$ points, we
numerically calculate the points $\{\x_j\}_{j=1}^N$ that minimize the
discrete energy $\eqref{optim:example}$.}
\label{fig:optim_case2}
\end{figure}

\section{Discussion}\label{sec:discussion}

We have developed a unified asymptotic theory for two canonical
diffusive capture problems in general smooth closed three-dimensional
domains whose boundary consists of small non-overlapping partially
reactive boundary patches of arbitrary shape on an otherwise
reflecting surface.  First, we considered the exterior Berg-Purcell
problem of receptor-mediated absorption.  Then, we adapted the
analysis to treat the interior narrow escape problem for the mean
first-reaction time for a diffusing particle to react on any of small
boundary patches.  Previous treatments were largely confined to the
sphere, where the surface Neumann Green's function is known explicitly
and the boundary is a surface of constant curvature.  By performing a
matched asymptotic analysis relying on a local orthogonal coordinate
system, we have obtained three-term expansions (Propositions
\ref{berg:main_res} and \ref{mfpt_b:main_res}) for the capacitance
$C_T$ and the global (or volume-averaged) MFRT $\overline{u}$.
The central outcome of our analysis consists in accurate closed-form
expressions for the key quantities that are determined in terms of
both local {\clb properties} of each patch and those representing
long-range interactions.  In particular, the local geometry near each
patch enters through the principal curvatures $\kappa_{1i},
\kappa_{2i}$ and the patch reactivity $b_i$, as encoded in the
reaction capacitance $C_i$ and certain monopole coefficients
$E_{i\pm}$.  These latter quantities are determined up to various
quadratures that involve the charge density on each patch.  In
contrast, the long-range interactions that encode the global domain
geometry enter through a surface Neumann Green's matrix.

Resolving the weak, path-dependent singularity of the surface Neumann
Green's function on a curved boundary was the key technical step that
enabled our general-geometry analysis. For the special case of $N$
identical circular patches, the results collapse to compact formulas
where the curvature appears only through the mean curvature
$\HC_i=\frac{1}{2}(\kappa_{1i}+ \kappa_{2i})$, and where they reduce to
previously derived results for a spherical domain in the appropriate
limit.

A natural application of these formulas that awaits to be explored is
the optimization of capture over different surface geometries. More
specifically, given a manifold where should a fixed number of
receptors be placed, and how should their sizes and reactivities be
apportioned, to maximize the absorption rate for capture of diffusing
ligands. Our derived asymptotic results reduce this question to a
finite-dimensional optimization problem in the patch locations,
involving the Green's interaction energy
\[
  p(\x_1,\ldots,\x_N) = \sum_{i=1}^N\Big[ R(\x_i) +
  \sum_{\substack{j=1\\i\neq j}}^NG(\x_j;\x_i)\Big]\,,
\]
together with the local mean-curvature contributions $\HC_i$. In
practice, optimization will require coupling the asymptotic formulas
to a numerical solver to build the Green's interaction matrix on the
surface of interest, for instance through the boundary-integral
methods \cite{lindsay20263D_GFun}. We note one qualitative feature
that distinguishes the three-dimensional setting from its
two-dimensional counterpart
\cite{chakraborty2025fastintegralmethodsneumann,GrebWard2D}: the spatially
varying corrections enter at the $\mathcal{O}(\eps\log\eps)$ term due
to local mean curvature effects while global effects are encoded in
the $\mathcal{O}(\eps)$ term. This differs from planar and spherical
narrow-capture problems where the equivalent terms are independent of
the trap configuration in the limit $\eps\to0$. We hypothesize that
the set of optimizing locations will depend on the patch radius
$\eps$ with smaller values being associated with clustered
configurations around critical points of mean curvature.
 
A second direction for future work is to derive effective scaling laws
in the dense-receptor limit $N \to \infty$ while keeping the receptor area
$\mathcal{O}(N\eps^2)$ fixed. On the sphere, the leading Berg-Purcell
rate together with a packing correction can be captured by a
homogenized effective reactivity to be imposed on the entire domain
boundary (see \cite{LWB2017,GrebenkovWard2026}). We
conjecture that on a general surface the analogous
law must involve some average of the local curvatures and an
interaction energy whose growth with $N$ reflects the geometry of
optimal point configurations on the manifold. Establishing this law
and identifying the effective conditions that summarize a densely
patterned surface as a single homogenized boundary condition would
connect the discrete asymptotics developed here to a continuum
description of partially absorbing curved boundaries.
 
Finally, the asymptotic framework developed herein unlocks the
potential for examining a broader class of geometry-dependent
first-passage questions.  The same inner/outer structure and the same
curved-surface Green's function apply, with minor modification, to
studying splitting probabilities among competing receptors
\cite{LLM2020,GrebenkovWard2026big}, to directional sensing
\cite{BJNL2023} and source localization by a cell reading out the
spatial pattern of capture events, and to the extreme-value statistics
of the fastest few arrivals that govern the timing of many cellular
decisions.  {\clb We remark that the splitting probability is readily
  obtained from our analysis by calculating the local flux to a specific receptor and dividing by the
  total flux.}  Yet another application concerns the asymptotic
analysis of mixed Steklov eigenvalue problems (see
\cite{GrebenkovWard2026big}).  In each case the local curvature and
the global Green's interactions are expected to play the organizing
role they do here, and the methods of this paper provide a systematic
route to quantifying how geometric effects shape stochastic processes.

\section*{Acknowledgments}
AEL acknowledges support from the NSF under grant {DMS-2052636}.  JGH
was supported in part by a Sloan Research Fellowship.  MJW is grateful
for the support of the NSERC Discovery Grant Program (Canada).  DSG
acknowledges the Simons Foundation for supporting his sabbatical
sojourn in 2024 at the CRM (CNRS -- University of Montr\'eal, Canada),
and the Alexander von Humboldt Foundation for support within a Bessel
Prize award.  This research was supported in part by grants from the
NSF (DMS-2235451) and Simons Foundation (MPS-NITMB-00005320) to the
NSF-Simons National Institute for Theory and Mathematics in Biology
(NITMB).

\appendix
\renewcommand{\theequation}{\Alph{section}.\arabic{equation}}

\section{Boundary-fitted local orthogonal coordinates}\label{app_b:bernoff}

We introduce a local coordinate system defined near $\x_i$ using the
intrinsic geometry of the surface as was done in
\cite{GrebenkovWard2026}.  We let $t_1$ and $t_2$ be the local
orthogonal surface coordinates, centered at $\x_i$, that correspond to
the two principal directions through $\x_i\in \partial\Omega$ with
principal curvatures $\kappa_1$ and $\kappa_2$, respectively.  Locally
the surface is described by the quadratic form
\begin{equation}\label{app_b:quad}
  z^{\prime}= {\mathcal B}(t_1,t_2) \equiv
  \frac{1}{2}\kappa_1 t_1^2 + \frac{1}{2} \kappa_2 t_{2}^2 +
  o(t_1^2+t_2^2) \,.
\end{equation}
In terms of a suitable orthogonal {\clb (rotation)} matrix
${\orthomat}$, we introduce the local change of coordinates defined by
\bsub \label{app_b:coord}
\begin{equation}\label{app_b:coord_1}
  \x^{\prime}\equiv {\orthomat}(\x-\x_i)=
  \left(t_1,t_2,{\mathcal B}(t_1,t_2)
  \right)^{T} + d  {\bf \hat{n}} \,,
\end{equation}
where $d\ge 0$ and ${\bf \hat{n}}$ is the inward pointing unit normal
for the local parameterization $z^{\prime}={\mathcal B}(t_1,t_2)$ of
the surface, which is given for $|t_1|\ll 1$ and $|t_2|\ll 1$ by
\begin{equation}\label{app_b:coord_2}
  {\bf \hat{n}} =
  \frac{ \left(-{\mathcal B}_{t_1}, -{\mathcal B}_{t_2}, 1\right)^T}{
    \sqrt{1 + {\mathcal B}_{t_1}^2 + {\mathcal B}_{t_2}^2 }} =
  \left(-\kappa_1t_1,-\kappa_2t_2,1\right)^{T} \left[ 1 +
    {\mathcal O}(t_1^2 + t_2^2) \right]\,.
\end{equation}
\esub {\clb For the validity of this coordinate system locally we
  require that $d<{1/\max(0,\kappa_1,\kappa_2)}$.}

To transform the Laplacian of a generic function $V(\x)$ from Cartesian to
the local coordinates $(t_1,t_2,d)$ of (\ref{app_b:coord}), we use the
fact that ${\orthomat}$ preserves lengths, and so
\begin{equation}\label{app_b:scale_fac}
  \begin{split}
    \nu_{t_1} &\equiv |{\partial \x/\partial t_1}| = 1-
    d \kappa_1 + {\mathcal O}(t_1^2+t_2^2+d^2) \,,\\
    \nu_{t_2} &\equiv |{\partial \x/\partial t_2}| = 1-
    d \kappa_2 + {\mathcal O}(t_1^2+t_2^2+ d^2) \,,\\
    \nu_{d} &\equiv |{\partial \x/\partial d}| = 1
    + {\mathcal O}(t_1^2+t_2^2) \,.
  \end{split}
\end{equation}
In terms of these scale factors, the Laplacian is transformed to
\begin{equation}\label{app_b:laplace_scale}
  \begin{split}
    \Delta_{\x} V &= \frac{1}{\nu_{t_1}\nu_{t_2}\nu_{d}} \left[
      \frac{\partial}{\partial t_1}\left( \frac{\nu_{t_2}\nu_{d}}
        {\nu_{t_1}} V_{t_1} \right) +
      \frac{\partial}{\partial t_2}\left( \frac{\nu_{t_1}\nu_{d}}
        {\nu_{t_2}} V_{t_2} \right) +
            \frac{\partial}{\partial d}\left( \frac{\nu_{t_1}\nu_{t_2}}
              {\nu_{d}} V_{d} \right) \right] \,,\\
          &= V_{dd} - \left(\frac{\kappa_1}{1-\kappa_1 d} +
            \frac{\kappa_2}{1-\kappa_2 d}\right) V_{d} \\
          & \qquad 
          + \frac{1}{(1- d \kappa_1)(1- d \kappa_2)} \left[
            \frac{\partial}{\partial t_1} \left(
              \frac{1- d \kappa_2}{1- d \kappa_1} V_{t_1} \right) +
            \frac{\partial}{\partial t_2} \left(
              \frac{1- d \kappa_1}{1- d \kappa_2} V_{t_2} \right)\right] \,.
          \end{split}
\end{equation}          

Next, we introduce the $\eps$-localized variables
$(\txione,\txitwo,\teta)$ defined by
\begin{equation}\label{app_b:innvar}
  \txione\equiv {t_1/\eps} \,, \quad  \txitwo\equiv {t_2/\eps} \,, \quad
  \teta\equiv {d/\eps} \,, \quad \tdelta \equiv \partial_{\txione\txione}
+  \partial_{\txitwo\txitwo} + \partial_{\teta\teta} \,.
\end{equation}
In terms of these local variables, for $\eps\to 0$ we find that the
Laplacian (\ref{app_b:laplace_scale}) reduces to
\begin{equation}\label{app_b:local_laplace}
  \Delta_{\x}V = \frac{1}{\eps^2} \tdelta V +
  \frac{1}{\eps} \bigl[-2\HC V_{\teta} + 2 \HC \teta \left(V_{\txione\txione}
      + V_{\txitwo\txitwo}\right)  + 2\kapdif \teta
    \left(V_{\txione\txione}- V_{\txitwo\txitwo}\right) \bigr] + {\mathcal O}(1)\,,
\end{equation}
where $\HC$ is the mean curvature and $\kapdif$ is the curvature difference
defined in terms of the two principal curvatures $\kappa_1$ and $\kappa_2$
of $\partial\Omega$ at $\x=\x_i$ by
\begin{equation}\label{app_b:curv_def}
  \HC = \frac{\kappa_1+\kappa_2}{2} \,, \qquad \kapdif=
  \frac{\kappa_1-\kappa_2}{2} \,.
\end{equation}

Next, to derive a two-term approximation for the Euclidian distance
$|\x-\x_i|$, we use (\ref{app_b:innvar}) in (\ref{app_b:coord}) and
collect powers of $\eps$. Since ${\orthomat}$ is an orthogonal
matrix, we obtain
\bsub \label{app_b:vab_all}
\begin{equation}
 {\orthomat}(\x-\x_i) = \eps \y + \eps^2 \vb + {\mathcal O}(\eps^3)\,, \qquad
  |\x-\x_i|^2=\eps^2 \y^T\y + 2 \eps^3 \y^T \vb  +
  {\mathcal O}(\eps^4)\,,
\end{equation}
where $\y$ and $\vb$ are defined by
\begin{equation}\label{app_g:vab}
  \y \equiv \left(\txione,\txitwo,\teta\right)^{T}\,, \qquad
  \vb \equiv \left(-\kappa_1 \txione\teta,-\kappa_2\txitwo\teta,
    \frac{1}{2}\kappa_1\txione^2 + \frac{1}{2}\kappa_2\txitwo^2\right)^{T}\,.
\end{equation}
\esub
By calculating
$\y^T\y =\txione^2+\txitwo^2+\teta^2 \equiv \rho^2$ and
$\y^T\vb={-\teta\left(\kappa_1\txione^2 + \kappa_2 \txitwo^2\right)/2}$,
we get
\begin{equation}\label{app_g:loc}
  \begin{split}
  |\x-\x_i| &\sim \eps \rho\left( 1  - \frac{\eps \teta}{2\rho^2}
    \left[ \HC (\txione^2 +\txitwo^2) + \kapdif(\txione^2-\txitwo^2)
    \right] \right) + {\mathcal O}(\eps^3), \\
  \frac{1}{|\x-\x_i|} &\sim \frac{1}{\eps \rho} \left( 1 +
    \frac{\eps \teta}{2\rho^2} \left[ \HC (\txione^2 +\txitwo^2)
  + \kapdif(\txione^2-\txitwo^2) \right]
    \right) + {\mathcal O}(\eps) \,.
\end{split}
\end{equation}

{\clb Finally, we observe from (\ref{app_g:loc}) that $|\x-\x_i|=\eps
  \rho + {\mathcal O}(\eps^3)$ and $\rho=(\txione^2+\txitwo^2)^{1/2}$
  when $\teta=0$. As a result, if the orthogonal projection of the
  Robin patch $\partial\Omega_{i}^{\eps}$ onto the tangent is a disk
  of radius $\eps a_i$, then the rescaled surface patch $\PT_{i}$ is
  described by $\left(\txione^2+\txitwo^2\right)^{1/2}\leq a_i +
  {\mathcal O}(\eps^2)$.  This small ${\mathcal O}(\eps^2)$ difference
  between the the orthogonally-projected and on-surface patches does
  not influence our three-term asymptotic result in Propositions
  \ref{berg:main_res} and \ref{mfpt_b:main_res}.}

\section{The surface Neumann Green's function}\label{app_g:int_green}

We now derive (\ref{berg:green_int_sing}) and show that the
unspecified $e(\x;\x_i)$ term in (\ref{berg:green_int_sing}), which is
determined solely by the local behavior of $\partial\Omega$ near
$\x=\x_i$, exhibits a mild singularity in that the value as
$\x\to\x_i$ depends on the specific path of approach to $\x_i$.  This
detailed analysis below produces a result equivalent to that derived
in \cite{NURSULTANOV2021202} from microlocal analysis; here, we
present an alternative approach based on a local coordinate system
defined in Appendix \ref{app_b:bernoff} which is suited for the
high-order asymptotic analysis of (\ref{berg_bp:ssp}) and
(\ref{mfpt:ssp}).  In our analysis below, for ease of notation we omit
the subscript $e$ for this Green's function.

In terms of the $\eps$-localized coordinates (\ref{app_b:innvar}), we
apply (\ref{app_b:local_laplace}) to the problem
(\ref{berg:green_int}) for $G$. The boundary condition in
(\ref{berg:green_int}) on the surface $\teta=0$ transforms locally to
\begin{equation*}
  \partial_{n} G = -\frac{1}{\eps\nu_{d}} \partial_{\teta} G = \delta(\x-\x_i)=
  \delta( {\orthomat}^{T}\x^{\prime})= \frac{1}{\eps^2 |\det{{\orthomat}}|}
  \delta(\txione) \delta(\txitwo)\,.
\end{equation*}
Since $\nu_{d}=1+{\mathcal O}(\eps^2)$ and $|\det{{\orthomat}}|=1$,
we obtain that
\begin{equation}\label{app_g:bc_loc}
  G_{\teta} \sim - \frac{1}{\eps} \delta(\txione) \delta(\txitwo) \,, \quad
  \mbox{on} \quad \teta=0\,.
\end{equation}
On the domain $\R_{+}^{3}$, as defined in (\ref{berg:r3+}), we substitute
the expansion
\begin{equation}\label{app_g:gexpan}
  G= \frac{G_0}{\eps} + G_1 + {\mathcal O}(\eps)
\end{equation}
into both (\ref{app_b:local_laplace}) applied to $G$ and the boundary
condition (\ref{app_g:bc_loc}). By collecting powers of $\eps$, we 
obtain the leading-order problem
\begin{equation}\label{app_g:g0_prob}
    \tdelta G_0 = 0 \,, \quad \mbox{in} \,\,\, \R_{+}^{3} \,; \qquad
    \partial_{\teta} G_0 =-\delta(\txione)\delta(\txitwo) \,, \quad \mbox{on}
    \,\,\, \teta=0 \,,
\end{equation}
where $\tdelta G_0\equiv G_{0,\txione\txione}+G_{0,\txitwo\txitwo}+G_{0,\teta\teta}$.
At next order we obtain that $G_1$ satisfies
\begin{equation}\label{app_g:g1_prob}
  \begin{split}
    \tdelta G_1 &= 2\HC G_{0,\teta} -2\HC\teta \left(G_{0,\txione\txione}
      +G_{0,\txitwo\txitwo}\right)-2\kapdif\teta\left(G_{0,\txione\txione}-
      G_{0,\txitwo\txitwo}\right) \,, \quad \mbox{in} \,\,\, \R_{+}^{3} \,, \\
    \partial_{\teta} G_1 &=0 \,, \quad \mbox{on} \,\,\, \teta=0 \,.
  \end{split}
\end{equation}

The solution to (\ref{app_g:g0_prob}) is
\begin{equation}\label{app_g:g0_sol}
  G_{0} = \frac{1}{2\pi \rho} \,, \quad \mbox{where} \quad
  \rho \equiv \left(s^2 +\teta^2\right)^{1/2} \,, \quad \mbox{and} \quad
  s^2= \txione^2+\txitwo^2 \,,
\end{equation}
and we conveniently decompose $G_1$ as
\begin{equation}\label{app_g:g1_sol}
  G_{1}=2 \HC G_{+} - 2\kapdif G_{-} \,.
\end{equation}
By using $G_{0,\txione\txione}+G_{0,\txitwo\txitwo}=-G_{0,\teta\teta}$ for
$(\txione,\txitwo,\teta)\ne (0,0,0)$ in (\ref{app_g:g1_prob}), the
decomposition (\ref{app_g:g1_sol}) yields that $G_{+}$ satisfies
\begin{equation}\label{app_g:g1+_prob}
    \tdelta G_{+} = G_{0,\teta}+\teta G_{0,\teta\teta}=\left[\teta G_{0,\teta}
    \right]_{\teta} \,, \quad \mbox{in} \quad \R_{+}^{3} \,; \qquad
     \partial_{\teta} G_{+} =0 \,, \quad \mbox{on} \quad \teta=0 \,,
\end{equation}
while $G_{-}$ satisfies
\begin{equation}\label{app_g:g1-_prob}
    \tdelta G_{-} = \teta \left( G_{0,\txione\txione}-G_{0,\txitwo\txitwo}\right)
    \,, \quad \mbox{in} \quad \R_{+}^{3} \,; \qquad
    \partial_{\teta} G_{-} =0 \,, \quad \mbox{on} \quad \teta=0 \,.
\end{equation}

We can determine $G_{+}$ by simply using the PDE that $G_0$
solves. The result is as follows:

\begin{lemma}\label{lemma:G1+sol}
  For $\big(\txione,\txitwo,\teta\big)\ne (0,0,0)$, the solution to
  (\ref{app_g:g1+_prob}) that has no singularity in $\R_{+}^{3}$ has
  the form
\begin{equation}\label{app_g:g1+decomp}
  G_{+}=\frac{\teta^2}{4}G_{0,\teta} + \frac{\teta}{4}G_0 -\frac{1}{4}
  \int^{\teta} G_{0} \, d\teta + {\mathcal F}(\txione,\txitwo)\,,
\end{equation}
where ${\mathcal F}(\txione,\txitwo)$ is a 2-D harmonic solution
satisfying
$\Delta_{\xi}{\mathcal F}\equiv {\mathcal F}_{\txione\txione} +
{\mathcal F}_{\txitwo\txitwo}=0$.
\end{lemma}

\begin{proof} For $(\txione,\txitwo,\teta)\ne (0,0,0)$ we readily
  calculate that
  \begin{equation}\label{app_g:decomp_1}
  G_{+,\teta} = \frac{3}{4}\teta G_{0,\teta} + \frac{\teta^2}{4}
  G_{0,\teta\teta} \,, \qquad
  G_{+,\teta\teta} = \frac{\teta^2}{4} G_{0,\teta\teta\teta} + \frac{5}{4}
  \teta G_{0,\teta\teta} + \frac{3}{4}G_{0,\teta}\,.
\end{equation}
Upon defining $\Delta_{\xi} \equiv \partial_{\txione\txione}+
\partial_{\txitwo\txitwo}$, we use the PDE that $G_0$ solves to get
\begin{equation}\label{app_g:decomp_2}
  \begin{split}
    \Delta_{\xi} G_{+} &= \frac{\teta^2}{4}
    \left[\Delta_{\xi}G_0\right]_{\teta} + \frac{\teta}{4} \Delta_{\xi} G_0
    - \frac{1}{4} \int^{\teta} \Delta_{\xi} G_0 \, d\teta + \Delta_{\xi}
    {\mathcal F} \,, \\
    & = - \frac{\teta^2}{4} G_{0,\teta\teta\teta} - \frac{\teta}{4}
    G_{0,\teta\teta} + \frac{1}{4} \int^{\teta} G_{0, \teta\teta} \, d\teta +
    \Delta_{\xi} {\mathcal F} \,, \\
    & = -\frac{\teta^2}{4} G_{0,\teta\teta\teta} - \frac{\teta}{4} G_{0,\teta\teta}
    + \frac{1}{4} G_{0,\teta} + \Delta_{\xi} {\mathcal F} \,.
  \end{split}
\end{equation}
Adding the last line in (\ref{app_g:decomp_2}) to $G_{+,\teta\teta}$
from (\ref{app_g:decomp_1}), and setting
$\Delta_{\xi} {\mathcal F}=0$ for $(\txione,\txitwo)\ne (0,0)$, we
conclude that $G_{+}$ satisfies (\ref{app_g:g1+_prob}) on
$\R_{+}^{3}$ as required. From (\ref{app_g:decomp_1}) it follows that
the boundary condition $G_{+\teta}=0$ on $\teta=0$ for
$(\txione,\txitwo)\ne (0,0)$ holds.
\end{proof}

To determine $G_{+}$ explicitly, we use (\ref{app_g:g0_sol}) for
$G_0$ to obtain
\begin{equation*}
  G_{0,\teta} = -\frac{\teta}{2\pi \rho^3} \,, \qquad G_{0,\teta\teta}=
  -\frac{1}{2\pi \rho^3} + \frac{3\teta^2}{2\pi \rho^5} \,.
\end{equation*}
Replacing the integral in (\ref{app_g:g1+decomp}) with a definite
integral, and recalling that $\int^{x} (a^2+x^2)^{-1/2}\, dx =
\log\left(x+\sqrt{x^2+a^2}\right)$, we get from  (\ref{app_g:g1+decomp}) that
\begin{equation*}
  \begin{split}
  G_{+} &= \frac{\teta}{8\pi\rho}\left(1 - \frac{\teta^2}{\rho^2}\right)
  -\frac{1}{8\pi}\int_{0}^{\teta} \frac{dt}{\sqrt{\txione^2+\txitwo^2+
      t^2}} + {\mathcal F}(\txione,\txitwo)\,, \\
  & = \frac{\teta}{8\pi \rho^3} \left(\txione^2+\txitwo^2\right) -
  \frac{1}{8\pi} \log\left( \teta + \rho\right) + \frac{1}{8\pi}
  \log\left(\sqrt{\txione^2+\txitwo^2}\right) + {\mathcal F}(\txione,\txitwo)
  \,.
  \end{split}
\end{equation*}
To ensure that $G_{+}$ has no singularity in $\R_{+}^{3}$, we must choose,
up to an additive constant $K$, that
${\mathcal F}=-(8\pi)^{-1}\log\left(\sqrt{\txione^2+\txitwo^2}\right)$. In this
way, we conclude that
\begin{equation}\label{app_g:g1+_sol}
  G_{+} = \frac{\teta}{8\pi \rho^3} \left(\txione^2+\txitwo^2\right)-
  \frac{1}{8\pi} \log(\teta + \rho) + K \,.
\end{equation}

Next, we solve (\ref{app_g:g1-_prob}) for $G_{-}$. By using 
(\ref{app_g:g0_sol}) for $G_0$, (\ref{app_g:g1-_prob}) becomes
\begin{equation}\label{app_g:ng1-_prob}
 \tdelta G_{-} = \frac{3\teta \left(\txione^2-\txitwo^2\right)}{2\pi \rho^5}
       \,, \quad \mbox{in} \quad \R_{+}^{3} \,; \qquad
    \partial_{\teta} G_{-} =0 \,, \quad \mbox{on} \quad \teta=0 \,.
\end{equation}
We first seek a particular solution $G_{-p}$ to (\ref{app_g:ng1-_prob}) by
using the following simple lemma.

\begin{lemma}\label{lemma:homog}
  Let $\y=(\txione,\txitwo,\teta)^{T}$, $\rho=|\y|$, and let
  $P_{n}(\y)$ be a homogeneous polynomial in $\y$ of degree $n$, so
  that $P_n(t\y)=t^{n}P_{n}(\y)$ for any scalar $t$. If $P_{n}(\y)$ is
  also harmonic in 3-D, then
  \begin{equation}\label{app_g:lemma_g1-}
    \tdelta \left(\frac{P_{n}(\y)}{\rho^3}\right) =
    \frac{6(1-n)P_n(\y)}{\rho^5} \,.
  \end{equation}
\end{lemma}

\begin{proof}
We calculate
\begin{equation*}
  \tdelta\left( \frac{P_n}{\rho^3} \right) = \frac{\tdelta P_n}{\rho^3}
  + P_n \tdelta\left( \frac{1}{\rho^3}\right) + 2 \nabla P_n
  {\bf \cdot} \nabla \left( \frac{1}{\rho^3} \right) \,.
\end{equation*}
Since $\nabla \rho^{-3}=-{3\y/\rho^5}$ and
$\tdelta \rho^{-3}=6\rho^{-5}$ we get
\begin{equation}\label{app_g:poly}
  \tdelta\left( \frac{P_n}{\rho^3} \right) = \frac{\tdelta P_n}{\rho^3}
  +\frac{6P_n}{\rho^5} - \frac{6}{\rho^5} \nabla P_n {\bf \cdot} \y \,.
\end{equation}
By using Euler's result $\y {\bf \cdot} \nabla P_n = nP_n$ for
homogeneous polynomials of degree $n$ together with $\tdelta P_n=0$,
we observe that (\ref{app_g:poly}) reduces to (\ref{app_g:lemma_g1-}).
\end{proof}

To find a particular solution for (\ref{app_g:ng1-_prob}) we simply
choose the third-degree homogeneous polynomial
$P_{3}(\y)=\alpha \teta\big(\txione^2-\txitwo^2\big)$, where
$\alpha$ is a parameter. By equating the right-hand sides of
(\ref{app_g:lemma_g1-}) and the PDE in (\ref{app_g:ng1-_prob}), we
identify that $\alpha=-{1/(8\pi)}$. This yields a particular solution
for (\ref{app_g:ng1-_prob}) given by
\begin{equation}\label{app_g:g1-p_sol}
  G_{-p}= - \frac{\teta}{8\pi}  \frac{\left(\txione^2-\txitwo^2\right)}{\rho^3}
  \,.
\end{equation}
To satisfy the boundary condition on $\teta=0$ in
(\ref{app_g:ng1-_prob}) we decompose $G_{-}$ as
\begin{equation}\label{app_g:g1-_decomp}
  G_{-} = G_{-p} + G_{-c} \,,
\end{equation}
where the complementary solution $G_{-c}$ satisfies
\begin{equation}\label{app_g:ng1c-_prob}
  \begin{split}
    \tdelta G_{-c} &= 0 \,, \quad \mbox{in} \,\,\, \R_{+}^{3} \,, \\
    \partial_{\teta} G_{-c} &=-\partial_{\teta}G_{-p} = \frac{1}{8\pi}
    \frac{ \left(\txione^2-\txitwo^2 \right)}{\rho^3}
     = \frac{1}{8\pi s} \frac{ \left(\txione^2 -
        \txitwo^2\right)}{s^2} \,, \quad \mbox{on} \,\,\, \teta=0 \,,
\end{split}
\end{equation}
with $\rho^2=\teta^2+s^2$ and $s^2\equiv \txione^2+\txitwo^2$.

To determine the solution to (\ref{app_g:ng1c-_prob}) that has no
singularity in $\R_{+}^{3}$, it is convenient to use spherical coordinates
$(\rho,\theta,\phi)$, where $\theta$ and $\phi$ are the polar and
azimuthal angles, respectively, so that
\begin{equation}\label{app_g:sphere_coord}
  \teta=\rho \cos\theta \,, \quad s=\rho \sin\theta \,, \quad
   \frac{\txione^2-\txitwo^2}{s^2}= \cos(2\phi) \,.
\end{equation}
Since $\partial_{\theta} G_{-c}\vert_{\theta=\pi/2}=-s \partial_{\teta} G_{-c}
\vert_{\teta=0}$, we write $\tdelta$ in spherical coordinates to obtain that
(\ref{app_g:ng1c-_prob}) becomes
\begin{equation}\label{app_g:ng1c-_sphere}
  \begin{split}
    \tdelta G_{-c} &= 0\,, \quad
    \mbox{in} \,\,\, \rho>0\,, \quad 0\leq \theta\leq {\pi/2} \,, \quad
     0\leq \phi<2\pi \,, \\
     \partial_{\theta} G_{-c} &=-\frac{1}{8\pi} \cos(2\phi) \,, \quad
      \mbox{on} \,\,\, \theta={\pi/2} \,.
   \end{split}
\end{equation}
By seeking a solution to (\ref{app_g:ng1c-_sphere}) with no $\rho$ dependence
of the form
\begin{equation}\label{app_g:comp_1}
  G_{-c}= v(\mu) \cos(2\phi) \,,
\end{equation}
with $\mu=\cos\theta$, we conclude that $v(\mu)$ satisfies the ODE
\begin{equation}\label{app_g:comp_2}
  \left[ (1-\mu^2) v^{\prime}\right]^{\prime} - \frac{4}{1-\mu^2} v=0 \,,
  \quad \mbox{on} \quad 0\leq \mu < 1 \,; \quad v^{\prime}(0)=\frac{1}{8\pi}\,.
\end{equation}
It is readily verified that the unique solution to (\ref{app_g:comp_2})
that has no singularity at $\mu=1$ has the form $v=c{(1-\mu)/(1+\mu)}$, where
$c$ is a constant. By satisfying $v^{\prime}(0)={1/(8\pi)}$ we determine
this constant, and conclude that
\begin{equation}\label{app_g:comp_3}
  v = -\frac{1}{16\pi} \left( \frac{1-\mu}{1+\mu} \right) \,.
\end{equation}
Upon replacing $\mu=\cos\theta={\teta/\rho}$, we obtain from
(\ref{app_g:comp_3}) and (\ref{app_g:comp_1}) that 
\begin{equation}\label{app_g:comp_4}
  G_{-c}=-\frac{1}{16\pi} \left( \frac{\rho-\teta}{\rho+\teta}\right)
  \cos(2\phi) \,.
\end{equation}
Then, upon substituting (\ref{app_g:comp_4}) and (\ref{app_g:g1-p_sol})
into (\ref{app_g:g1-_decomp}), the solution to (\ref{app_g:ng1-_prob}) is
\bsub\label{app_g:g-_final_all}
\begin{equation}\label{app_g:g-_final}
  G_{-} = -\frac{1}{16\pi} \left[ \frac{2\teta s^2}{\rho^3} +
    \left( \frac{\rho-\teta}{\rho+\teta} \right) \right] \cos(2\phi)\,.
\end{equation}
Equivalently, upon using $\cos(2\phi)={(\txione^2-\txitwo^2)/s^2}$ and
$\rho^2-\teta^2=s^2$, we have
\begin{equation}\label{app_g:g-_final2}
  G_{-} = -\frac{1}{8\pi} \frac{\teta \left(\txione^2-\txitwo^2\right)}{\rho^3}
  -\frac{1}{16\pi} \frac{ \left(\txione^2-\txitwo^2\right)}{(\rho+\teta)^2}\,,
  \quad \mbox{where} \quad \rho=|\y|\,.
\end{equation}
\esub
Finally, by substituting (\ref{app_g:g-_final}) and
(\ref{app_g:g1+_sol}) into (\ref{app_g:g1_sol}), we obtain an explicit
representation for the two-term asymptotic expansion
(\ref{app_g:gexpan}). {\clb We emphasize that this truncated two-term
expansion describes the local behavior of the solution to Laplace's
equation with a Dirac singularity on the boundary.}  We summarize our
result as follows:

\begin{prop} 
Let $\HC={(\kappa_1+\kappa_2)/2}$ and $\kapdif=
{(\kappa_1-\kappa_2)/2}$, where $\kappa_1$ and $\kappa_2$ are the
principal curvatures of $\partial\Omega$ at $\x=\x_i$. {\clb Then, in terms
of the $\eps$-local boundary-fitted coordinates (\ref{app_b:innvar})
for the local coordinate transformation (\ref{app_b:coord}), and up to
an arbitrary constant $K$, the local singularity behavior as
$\x\to\x_i$ for the exterior surface Neumann Green's function
satisfying (\ref{berg:green_int}) is}
\bsub\label{app_g:glocal}
\begin{equation}\label{app_g:glocal_1}
    G_e\sim \frac{1}{2\pi\eps \rho} + G_1 + {\mathcal O}(\eps) \,,
\end{equation}
where
\begin{equation}\label{app_g:glocal_2}
    G_{1} =\frac{\HC}{4\pi} \left[ \frac{\teta s^2}{\rho^3} - \log(\teta+\rho)
    \right] +  K + \frac{\kapdif}{8\pi} \left[ \frac{2\teta s^2}{\rho^3}
      + \left( \frac{\rho-\teta}{\rho+\teta} \right) \right]
   \frac{(\txione^2-\txitwo^2)}{s^2} \,.
\end{equation}
\esub 
Equivalently, we can write (\ref{app_g:glocal}) as
\begin{equation}\label{app_g:glocal_3} 
G_e \sim \frac{1}{2\pi\eps \rho} + \frac{\HC}{4\pi} \left(
\frac{\teta(\txione^2+\txitwo^2)}{\rho^3} - \log(\teta+\rho)\right) +
K + \frac{\kapdif}{4\pi}\frac{\teta(\txione^2-\txitwo^2)}{\rho^3} +
\frac{\kapdif}{8\pi} \frac{(\txione^2-\txitwo^2)}{(\rho+\teta)^2} +
{\mathcal O}(\eps) \,.  
\end{equation}
\end{prop}

Finally, we identify the constant $e(\x;\x_i)$ in
(\ref{berg:green_int_sing}) and the constant $K$ in
(\ref{app_g:glocal_3}) by relating (\ref{app_g:glocal_3}) to
(\ref{berg:green_int_sing}). To do so, we substitute the two-term
expansion for $|\x-\x_i|$, as given in (\ref{app_g:loc}), into
(\ref{berg:green_int_sing}), while recalling $d=\teta \eps$.  Upon
comparing the resulting expression with (\ref{app_g:glocal_3}), we
obtain that \bsub \label{app_g:glocal_all}
\begin{equation} \label{app_g:glocal_ex}
   G_e  \sim \frac{1}{2\pi\eps \rho} + \frac{\HC}{4\pi} \left(
        \frac{\teta(\txione^2+\txitwo^2)}{\rho^3} - \log(\teta+\rho)\right)
      + R_e - \frac{\HC}{4\pi}\log\eps  +
      \frac{\kapdif}{4\pi} \frac{\teta(\txione^2-\txitwo^2)}{\rho^3} +
      e +  {\mathcal O}(\eps) \,,
\end{equation}
which determines the constant $K$ as $K=R_e-\tfrac{\HC}{4\pi}\log\eps$, and
identifies $e(\x;\x_i)$ as
\begin{equation}\label{app_g:e_1}
  e(\x;\x_i) = \frac{\kapdif}{8\pi} \frac{(\txione^2  -\txitwo^2)}
  {(\rho+\teta)^2}\,.
\end{equation}
\esub
We emphasize that the remaining terms in (\ref{app_g:glocal_ex}) match
identically, and these terms simply represent the effect of our chosen
local parameterization (\ref{app_b:coord_1}) of the surface.

We now write $e$ in (\ref{app_g:e_1}) in terms of the Cartesian
coordinate system.  In (\ref{app_g:e_1}), we replace $\rho\sim
{|\x-\x_i|/\eps}$, $\teta={d/\eps}$, $\txione={t_1/\eps}$, and
$\txitwo={t_2/\eps}$, and use (\ref{app_b:coord_1}) to solve for $t_1$
and $t_2$. In this way, we calculate that
\begin{equation}\label{app_g:e_final}
   e(\x;\x_i) = \frac{\kapdif}{8\pi( |\x-\x_i|+ d )^2}
    \left( \frac{p_1^2}{(1-\kappa_1 d)^2} - \frac{p_2^2}{(1-\kappa_2 d)^2}
        \right) + o(1) \,,
\end{equation}
where $d=\mbox{dist}(\x,\partial\Omega)$,
$\kapdif={(\kappa_1-\kappa_2)/2}$, and $p_1$ and $p_2$ are the
projections of ${\orthomat}(\x-\x_i)$ onto the two principal
directions of $\partial\Omega$ at $\x=\x_i$, defined in terms of the
components of ${\orthomat}(\x-\x_i)$ by $p_j=\left(
{\orthomat}(\x-\x_i)\right)_{j}$ for $j=1,2$.

{\clb We remark that the singularity structure of the interior surface
Neumann Green's function, satisfying (\ref{mfpt:green_int}), has
exactly the same two-term asymptotic behavior as that given in
(\ref{app_g:glocal_3}) for the exterior problem. This is because
the ${1/|\Omega|}$ term in (\ref{mfpt:green_int}) would only appear at
at one higher-order.} However, care is needed to determine the correct signs
of $\kappa_1$ and $\kappa_2$ for the interior problem as they will be
of opposite signs than that for the exterior problem.

\subsection{Explicit results for the sphere}\label{app:explicit}
  
When $\Omega$ is the unit sphere, the surface Neumann Green's $G_e$
satisfying (\ref{berg:green_int}) for the {\em exterior problem} is
(see \cite{SurfaceGreen3D,LLM2020,LWB2017})
\bsub \label{berg:gs_exact}
\begin{equation}\label{berg:gs_exact_1}
G_e(\x;\x_i) = \frac{1}{2 \pi \left|\x-\x_i\right|} -
\frac{1}{4\pi } \log\left( 1 + \frac{2}{|\x-\x_i|+|\x|-1} \right)\,,
\end{equation}
which has the following local behavior as $\x\to\x_i$:
\begin{equation}\label{berg:gs_loc1}
    G_e(\x;\x_i)\sim \frac{1}{2\pi|\x-\x_i|} +\frac{1}{4\pi}
    \log\bigl(|\x-\x_i|+|\x|-1\bigr) + R_{e} + o(1) \,; \quad
 R_{e}\equiv -\frac{\log{2}}{4\pi} \,.
\end{equation}
\esub 
For this case, $\HC_i=-1$ and $\kapdif_i=0$ for any $\x_i\in
\partial\Omega$.

%In addition, when $\Omega$ is the unit sphere, 
The surface Neumann Green's function satisfying (\ref{mfpt:green_int})
for the {\em interior problem} is (see Appendix A of
\cite{ChevWard2010})
\bsub \label{mfpt:gs_exact}
\begin{equation}
G_s(\x;\x_i) = \frac{1}{2 \pi \left|\x-\x_i\right|} +
\frac{|\x|^2 + 1}{8\pi} + \frac{1}{4
\pi } \log\left(\frac{2}{1 - \x {\bf \cdot} \x_i +
  \left|\x-\x_i\right|}\right) - \frac{7}{10 \pi} \,, \label{mfpt:gs_exact_1}
\end{equation}
which has the following local behavior as $\x\to\x_i$:
\begin{equation}\label{mfpt:gs_loc1}
  G_s(\x;\x_i)\sim \frac{1}{2\pi|\x-\x_i|} -\frac{1}{4\pi}
  \log\bigl(|\x-\x_i|+1 -|\x|\bigr) + R_{s} + o(1) \,; \quad
  R_{s}\equiv \frac{\log{2}}{4\pi} -
  \frac{9}{20\pi}\,.
\end{equation}
\esub
For the unit sphere, $\HC_i=1$ and $\kapdif_i=0$ for any
$\x_i\in \partial\Omega$.

\section{Monopole Coefficient $E_{i+}$}\label{app:mono-plus}

{\clb Although the properties of $E_{i+}$ were already reported in
\cite{GrebenkovWard2026,GrebenkovWard2026big}, for completeness we
outline the analysis of \cite{GrebenkovWard2026} (see also
\cite{GrebenkovWard2026big}) for the inner problem (\ref{berg:Phi+}).}
This analysis identifies the monopole coefficient in the refined
far-field behavior (\ref{berg:Phi+_ff}), whose properties were
summarized in Lemma \ref{lemma:monop}. For convenience we omit the
subscript $i$ below for the $i$-th patch.

We begin by decomposing the solution $\Phi_{+}$ to (\ref{berg:Phi+})  as
\begin{equation}\label{app+:decomp}
  \Phi_{+}= \Phi_{p+} + \Phi_{c+} \,,
\end{equation}
where $\Phi_{p+}$ accounts for the inhomogeneous term in the PDE
(\ref{berg:Phi+_1}) and satisfies the Neumann condition
$\partial_{\teta}\Phi_{p+}=0$ on the plane $\teta=0$. The complementary
part $\Phi_{c+}$ will satisfy the homogeneous PDE, but with an
inhomogeneous Robin condition on the patch $\PT$. As similar
to the derivation in Lemma \ref{lemma:G1+sol} for the analogous
component of the surface Neumann Green's function, $\Phi_{p+}$
can be determined analytically in terms of $w$ satisfying
(\ref{berg:wc}). 

\begin{lemma}(Lemma C.1 of \cite{GrebenkovWard2026}) \label{app+:Phip+}
The solution to (\ref{berg:Phi+}) satisfying
$\partial_{\teta}\Phi_{p+}=0$ on $\teta=0$ is
\begin{equation}\label{app+:Phip+_sol}
  \Phi_{p+}=-\frac{\teta^2}{2}w_{\teta} - \frac{\teta}{2}w + \frac{1}{2}
  \int_{0}^{\teta} w(\txione,\txitwo,t)\, dt +
  {\mathcal F}_{+}(\txione,\txitwo;b)\,,
\end{equation}
where ${\mathcal F}_{+}(\txione,\txitwo;b)$, with
$\Delta_{\xi}{\mathcal F}_{+} \equiv {\mathcal F}_{+,\txione\txione}+
{\mathcal F}_{+,\txitwo\txitwo}$, is the unique solution to
\bsub \label{app+:F+}
\begin{gather}
  \Delta_{\xi}{\mathcal F}_{+} = q(\txione,\txitwo;b) I_{\PT} \,; \,\,\,
  q \equiv - \left(\frac{1}{2} w_{\teta}\vert_{\teta=0}\right)\,,
  \,\,\, I_{\PT} \equiv \left\{\begin{array}{ll}
        1 \,, & (\txione,\txitwo) \in \PT \\
         0 \,, & (\txione,\txitwo) \notin \PT\,,  \end{array}\right.
                             \label{app+:F+1}\\
                            {\mathcal F}_{+} \sim \frac{C}{2}\log{s} +
                            o(1)\,,  \quad \mbox{as} \quad
                            s\equiv (\txione^2+\txitwo^2)^{1/2}\to \infty
                            \,. \label{app+:F+2}
\end{gather}                 
\esub The $o(1)$ condition in the far-field (\ref{app+:F+2})
specifies ${\mathcal F}_{+}$ uniquely. From the identity
(\ref{berg:wc_charge}) relating $C$ and $q$, an application of the
divergence theorem shows that the logarithmic growth specified in
(\ref{app+:F+2}) is consistent with the PDE
(\ref{app+:F+1}). Finally, the far-field behavior for
(\ref{app+:Phip+_sol}) is
\begin{equation}\label{app+:Phi+ff}
  \Phi_{p+}\sim \frac{C}{2} \left( \log(\teta+\rho) -
    \frac{\teta(\txione^2+\txitwo^2)}{\rho^3}\right) +
  o\left({1/\rho}\right)\,, \quad \mbox{as} \quad
  \rho=(\txione^2+\txitwo^2+\teta^2)^{1/2} \to\infty\,. 
\end{equation}
\end{lemma}

\begin{proof}
We calculate that
\begin{equation}\label{app+:phi2p_1}
  \Phi_{p+,\teta}=-\frac{\teta^2}{2}w_{\teta\teta}-\frac{3}{2}\teta w_{\teta} \,,
  \quad
  \Phi_{p+,\teta\teta} = - \frac{\teta^2}{2} w_{\teta\teta\teta} - \frac{5}{2}
  \teta w_{\teta \teta} - \frac{3}{2} w_{\teta} \,,
\end{equation}
with the first equation yielding $\Phi_{p+,\teta}=0$ on $\teta=0$. Next
by calculating
$\Delta_{\xi}\Phi_{p+}\equiv
\Phi_{p+,\txione\txione}+\Phi_{p+,\txitwo\txitwo}$, and by using
$w_{\teta\teta}=-\Delta_{\xi} w$ from (\ref{berg:wc_1}), we derive
\begin{equation}\label{app+:phi2p_s}
  \begin{split}
  \Delta_{\xi}\Phi_{p+} &= -\frac{\teta^2}{2}\partial_{\teta}\left[
    \Delta_{\xi} w\right] - \frac{\teta}{2}\Delta_{\xi}
  w + \frac{1}{2}\int_{0}^{\teta} \Delta_{\xi} w\, dt +
  \Delta_{\xi} {\mathcal F}_{+}  \\
      &= \frac{\teta^2}{2} w_{\teta\teta\teta} +
  \frac{\teta}{2} w_{\teta\teta} - \frac{1}{2} w_{\teta} + \frac{1}{2} w_{\teta}
  \vert_{\teta=0} + \Delta_{\xi} {\mathcal F}_{+}\,.
  \end{split}
\end{equation}
By adding (\ref{app+:phi2p_1}) and (\ref{app+:phi2p_s}), and by
choosing ${\mathcal F}_{+}$ to satisfy the PDE in (\ref{app+:F+1}), we
get 
\begin{equation}\label{app_h:phi2_add}
  \tdelta \Phi_{p+} \equiv \Phi_{p+,\teta\teta} + \Delta_{\xi}
  \Phi_{p+} =-2\left(\teta w_{\teta\teta} + w_{\teta}\right) \,.
\end{equation}
Therefore, $\Phi_{p+}$ satisfies the inhomogeneous PDE
(\ref{berg:Phi+_1}) subject to $\partial_{\teta}\Phi_{p+}=0$ on
$\teta=0$.

Next, to establish (\ref{app+:Phi+ff}), we use $w\sim C\rho^{-1}$ as
$\rho\to \infty$ with $\rho=(\teta^2+s^2)^{1/2}$ and
$s\equiv (\txione^2+\txitwo^2)^{1/2}$. We readily calculate for $\rho\to\infty$
that
\bsub \label{app+:ffes}
\begin{gather}
  -\frac{1}{2} \teta^2 w_{\teta} - \frac{1}{2} \teta w \sim 
 -\frac{C}{2} \frac{\teta s^2}{\rho^3} \,,
   \label{app+:ffes_1} \\
  \frac{1}{2} \int_{0}^{\teta} w(\txione,\txitwo,t)\, dt \sim 
  \frac{C}{2} \int_{0}^{\teta} \frac{1}{\left(s^2 + t^2\right)^{1/2}}
  \, dt \sim \frac{C}{2} \left[
  \log\left( \teta + \rho \right) - \log{s} \right] 
  \,. \label{app+:ffes_2}
\end{gather}
\esub
We conclude that $\Phi_{p+}$ in (\ref{app+:Phip+_sol}) has the
divergent far-field behavior
\begin{equation}
  \Phi_{p+} \sim  \frac{C}{2}\left( \log\left( \teta + \rho \right)
 - \frac{\teta}{\rho^3} \left(\txione^2+\txitwo^2\right) \right)  
  - \frac{C}{2} \log{s} + {\mathcal F}_{+} \,. \label{app+:phi2p_ff}
\end{equation}
From (\ref{app+:F+2}) we impose
${\mathcal F}_{+}\sim \left({C/2}\right)\log{s} + o(1)$ as
$s\to \infty$ to obtain (\ref{app+:Phi+ff}).

\end{proof}

This lemma shows that the particular solution $\Phi_{p+}$ in
(\ref{app+:Phip+_sol}) accounts for both the inhomogeneous term in the PDE
(\ref{berg:Phi+_1}) as well as the two divergent terms
in the required far-field behavior (\ref{berg:Phi+_4}). There is
no monopole term associated with the far-field behavior of $\Phi_{p+}$.

To determine the complementary solution $\Phi_{c+}$, we substitute
(\ref{app+:decomp}) into (\ref{berg:Phi+}) to conclude that $\Phi_{c+}$
satisfies
\bsub\label{app+c:inn2}
\begin{align}
\tdelta \Phi_{c+} &= 0 \,, \quad \y \in \R_{+}^{3} \,,\label{app+c:inn2_1}\\
  -\partial_{\teta} \Phi_{c+} + b \Phi_{c+} &= -b {\mathcal F}_{+}
                                                 \,, \quad \teta=0 \,,\,
    (\txione,\txitwo)\in \PT\,,  \label{app+c:inn2_2}\\
  \partial_{\teta} \Phi_{c+} &=0 \,, \quad \teta=0 \,,\,
                               (\txione,\txitwo)\notin \PT
    \,, \label{app+c:inn2_3}\\
    \Phi_{c+} & \sim \frac{E_{+}}{\rho}\,,  \quad
    \mbox{as} \quad \rho\to \infty \,, \label{app+c:inn2_4}
\end{align}
\esub 
where ${\mathcal F}_{+}$ satisfies (\ref{app+:F+}).  To determine
$E_{+}$ we apply Green's second identity to $\Phi_{c+}$ and
$\hat{w}=-1+w$, where $w$ satisfies (\ref{berg:wc}), over a hemisphere
$\Omega_{\sigma}$ of radius $\sigma \gg 1$ in the upper half-space
$\R_{+}^{3}$. By letting $\sigma\to\infty$ we get
\begin{equation}\label{app+c:green}
  \begin{split}
    0 = \int_{\R_{+}^{3}} \left(\Phi_{c+}\tdelta \hat{w} - \hat{w}\tdelta
      \Phi_{c+}\right) d\y &=
   2\pi \lim_{\sigma\to \infty} \sigma^2 \left(
    \Phi_{c+} \frac{\partial \hat{w}}{\partial |\y|} -
    \hat{w} \frac{\partial \Phi_{c+}}{\partial |\y|} \right)\Big{\vert}_{
    |\y|=\sigma} \\
  & \qquad +   \int_{\PT} \left(b \hat{w} {\mathcal F}_{+}\right)
  \vert_{\teta=0} \, d\txione\txitwo  \,.
  \end{split}
\end{equation}
Then, by using $\hat{w}\sim -1+{C/|\y|}$ and $\Phi_{c+}\sim
{E_{+}/|\y|}$ as $|\y|\to \infty$, together with
$b\hat{w}=\partial_{\teta}\hat{w}=\partial_{\teta}w$ on $\PT$ when
$\teta=0$, we obtain from (\ref{app+c:green}) that
\begin{equation}\label{app+c:e}
  E_{+} = -\frac{1}{\pi} \int_{\PT} q(\txione,\txitwo;b) \,
  {\mathcal F}_{+}(\txione,\txitwo;b) \, d\txione\txitwo \,.
\end{equation}
Finally, upon representing the solution ${\mathcal F}_{+}$ to
(\ref{app+:F+}) in terms of the 2-D free-space Green's function, we
find that (\ref{app+c:e}) yields (\ref{berg:Ei_general0}) of Lemma
\ref{lemma:monop}.  For $b\ll 1$, we have $q\sim {b/2}$ for an
arbitrary patch shape $\PT$ and so (\ref{berg:Ei_general0}) reduces to
(\ref{berg:Ei_asympt0}).

When $\PT$ is a disk of radius $a$, we observe that ${\mathcal
F}_{+}={\mathcal F}_{+}(s;b)$ satisfies the radially symmetric ODE
$\left(s{\mathcal F}_{+,s}\right)_{s}= q(s;b)\, s$.  Upon integrating
this ODE we determine $E_{+}$ in (\ref{app+c:e}) as the quadrature
\begin{equation}\label{app+c:e_rad}
  E_{+} = -2 \int_{0}^{a}  q(s;b) {\mathcal F}_{+}(s;b)\, s \, 
ds\,; \qquad
  {\mathcal F}_{+,s} = \frac{1}{s} \int_{0}^{s} 
  q(t;b) \, t \, dt \,, \quad 0\leq s \leq a\,,
\end{equation}
with ${\mathcal F}_{+}=\left({C/2}\right)\log{a}$ at $s=a$.  Upon
integrating (\ref{app+c:e_rad}) by parts, we readily obtain
(\ref{berg:Ej_all}) of Lemma \ref{lemma:monop}.

Finally, we obtain the asymptotic behavior for $E_{+}$ in
(\ref{berg:Ej_asy}).  For $b\to \infty$, we use $C(\infty)\sim
{2a/\pi}$ and $q(t;\infty)= \pi^{-1}\left(a^2-t^2\right)^{-1/2}$ from
Lemma \ref{lemma:Cj_kappa} directly in (\ref{berg:Ej_all}).  By
performing the integration, we obtain (\ref{berg:Ej_asy_large}).  For
$b\to 0$, we use $C\sim {b a^2/2}$ from (\ref{berg:Cj_small}) and
$q\sim {b/2}$ in (\ref{berg:Ej_all}).  By integrating the resulting
expression we get (\ref{berg:Ej_asy_small}).

\section{Monopole coefficient $E_{i-}$}\label{app:mono-minus}

In this appendix, we analyze the new problem (\ref{berg:Phi-}) to
identify the monopole coefficient in the refined far-field behavior
(\ref{berg:Phi-_ff}), whose properties were summarized in Lemma
\ref{lemma:monom}.  For clarity we again omit the subscript $i$ below
for the $i$-th patch.

For (\ref{berg:Phi-}), we decompose $\Phi_{-}$ as
\begin{equation}\label{app-:decomp}
  \Phi_{-}= \Phi_{p-} + \Phi_{c-} \,,
\end{equation}
where $\Phi_{p-}$ is taken to satisfy
\bsub \label{app-:Phi-p}
\begin{align}
  \tdelta \Phi_{p-} &= {\mathcal N}_{-}(\y) \equiv
    2 \teta \left( w_{\txione\txione} -w_{\txitwo\txitwo} \right) \,, \quad
    \y \in \R_{+}^{3} \,, \label{app-:Phi-p0} \\
    \partial_{\teta} \Phi_{p-} &=0 \,, \quad \mbox{on} \quad \teta=0 \,.
\end{align}
\esub 
We now verify indirectly that $\Phi_{p-}$ satisfies the intricate
far-field behavior given in (\ref{berg:Phi-_4}).  To do so, we use
$w\sim {C/\rho}$ as $\rho\to\infty$ from the far-field behavior in
(\ref{berg:wc_4}) to estimate that
\begin{equation}\label{app-:R-asy}
  {\mathcal N}_{-} \sim 4\pi C {\mathcal N}_{-0} \,,  \quad
  \mbox{where} \quad {\mathcal N}_{-0} \equiv \frac{3\teta}{2\pi \rho^5}
  \left(\txione^2-\txitwo^2\right) \,,
  \quad \mbox{as} \quad \rho=|\y|\to \infty  \,.
\end{equation}
We substitute (\ref{app-:R-asy}) into the right-hand side of
(\ref{app-:Phi-p0}) and compare the resulting limiting problem with
the corresponding problem (\ref{app_g:ng1-_prob}) for the term $G_{-}$
in our analysis of the surface Neumann Green's function. In this way,
we conclude that
\begin{equation}\label{app-:Phip-_asy}
  \Phi_{p-} \sim 4\pi C G_{-} \,, \quad \mbox{as} \quad \rho=|\y|\to \infty
  \,.
\end{equation}
By substituting the exact solution in (\ref{app_g:g-_final2}) for
$G_{-}$ into (\ref{app-:Phip-_asy}) we conclude that the limiting
behavior of $\Phi_{p-}$ as $\rho=|\y|\to \infty$ agrees precisely with
that required in (\ref{berg:Phi-_4}).  As a result, $\Phi_{p-}$ has no
monopole behavior as $|\y|\to \infty$.

Next, we use the method of images to represent the solution
$\Phi_{p-}$ of (\ref{app-:Phi-p}) in terms of the Green's function
$g(\y^{\prime};\y)$, with
$\y^{\prime}=\big(\txione^{\prime},\txitwo^{\prime},\teta^{\prime}\big)^T$,
satisfying
\begin{equation}\label{app-:gimage}
  \Delta_{\y^{\prime}} g = \delta(\y^{\prime}-\y) \,, \quad
  \y^{\prime}\in \R_{+}^{3} \,; \qquad
  \partial_{\teta^{\prime}} g =0 \,, \quad \mbox{on} \quad \teta^{\prime} =0 \,,
\end{equation}
which has the solution
\begin{equation}\label{app-:gimage_sol}
  g(\y^{\prime};\y) = -\frac{1}{4\pi |\y^{\prime}-\y|} -
  \frac{1}{4\pi |\y^{\prime}-\bar{\y}|}  \,, \quad
  \mbox{with} \quad {\bar \y} \equiv (\txione,\txitwo,-\teta)^T .
\end{equation}
In deriving an integral representation for $\Phi_{p-}$ from Green's
second identity we must take into account that the far-field behavior
for $\Phi_{p-}$ is non-vanishing as
$\rho^{\prime}=|\y^{\prime}|\to\infty$.  To do so, we write the
far-field behavior (\ref{berg:Phi-_4}) with no additional monopole in
terms of spherical coordinates
$(\rho^{\prime},\theta^{\prime},\phi^{\prime})$, with $0\leq
\theta^{\prime}\leq {\pi/2}$ and $0\leq \phi^{\prime}<2\pi$.  We
readily conclude that $\Phi_{p-} \sim \Phi_{p\infty} +
o\left({1/\rho^{\prime}}\right)$ as $\rho^{\prime}\to \infty$, where
\begin{equation}\label{app-:Phip-_ff}
  \Phi_{p\infty}(\theta^{\prime},\phi^{\prime}) \equiv
  -\frac{C}{4} \left( 2\cos\theta^{\prime} +
    \frac{1}{\left(1+\cos\theta^{\prime}\right)^2}
  \right) \sin^{2}(\theta^{\prime}) \, \cos(2\phi^{\prime})\,.
\end{equation}
By applying Green's second identity to $g$ and $\Phi_{p-}$ over a
large hemisphere $\Omega_{\sigma}$ in the upper half-space, and using
$g\sim -{1/\left(2\pi \rho^{\prime}\right)}$ for $\rho^{\prime}\to \infty$
together with (\ref{app-:Phip-_ff}), we obtain for $\sigma\to \infty$
that
\begin{equation}\label{app-:Phip-_green}
 \Phi_{p-}(\y) = \int_{\R_{+}^{3}} g(\y^{\prime};\y) {\mathcal N}_{-}(\y^{\prime}) \,
    d\y^{\prime} + \frac{1}{2\pi} 
    \int_{0}^{\pi/2} \sin(\theta^{\prime})
    \left(\int_{0}^{2\pi} \Phi_{p\infty}(\theta^{\prime},\phi^{\prime})
     \, d\phi^{\prime}\right) \, d\theta^{\prime} \,.
\end{equation}
By using (\ref{app-:Phip-_ff}), we calculate
$\int_{0}^{2\pi} \Phi_{p\infty} \, d\phi^{\prime}=0$, so that
(\ref{app-:Phip-_green}) reduces to
\begin{equation}\label{app-:Phip-_sol}
  \Phi_{p-}(\y) = -\frac{1}{4\pi} \int_{\R_{+}^{3}}
  {\mathcal N}_{-}(\y^{\prime}) \left(
    \frac{1}{|\y^{\prime}-\y|} + \frac{1}{|\y^{\prime}-{\bar \y}|} \right)
  \, d\y^{\prime} \,.
\end{equation}

To determine the complementary solution we substitute (\ref{app-:decomp})
into (\ref{berg:Phi-}) to obtain
\bsub\label{app-c:inn2}
\begin{align}
\tdelta \Phi_{c-} &= 0 \,, \quad \y \in \R_{+}^{3} \,,\label{app-c:inn2_1}\\
  -\partial_{\teta} \Phi_{c-} + b \Phi_{c-} &= -b {\mathcal F}_{-}
                                                 \,, \quad \teta=0 \,,\,
    (\txione,\txitwo)\in \PT\,,  \label{app-c:inn2_2}\\
  \partial_{\teta} \Phi_{c-} &=0 \,, \quad \teta=0 \,,\,
                               (\txione,\txitwo)\notin \PT
    \,, \label{app-c:inn2_3}\\
    \Phi_{c-} & \sim \frac{E_{-}}{\rho}\,,  \quad
    \mbox{as} \quad \rho\to \infty \,, \label{app-c:inn2_4}
\end{align}
\esub where ${\mathcal F}_{-}={\mathcal F}_{-}(\txione,\txitwo)$ is given
by
\begin{equation}\label{app-c:f-val}
  {\mathcal F}_{-}=\Phi_{p-} \vert_{\teta=0}=-\frac{1}{2\pi}
  \int_{\R_{+}^{3}} \frac{ {\mathcal N}_{-}(\y^{\prime})} {|\y^{\prime}-\y|}
  \, d\y^{\prime} \,, \quad \mbox{where} \quad \y=(\txione,\txitwo,0)^T \,.
\end{equation}

To determine our first expression for $E_{-}$ we apply Green's identity to
$\Phi_{c-}$ and $\hat{w}=-1+w$, in a similar way as was done to find
$E_{+}$ in (\ref{app+c:green}) and (\ref{app+c:e}). This yields that
\begin{equation}  \label{berg:Ei_-1}
  E_{-}(b) =  -\frac{1}{\pi} \int_{\PT} q(\txione,\txitwo)\,
  {\mathcal F}_{-}(\txione,\txitwo) \, d\txione d\txitwo \,, \qquad \mbox{with}
  \quad
   q =-\frac{1}{2} w_{\teta}\vert_{\teta=0} \,.
\end{equation}
Upon using (\ref{app-c:f-val}) for ${\mathcal F}_{-}(\txione,\txitwo)$,
with ${\mathcal N}_{-}$ as  defined in (\ref{app-:Phi-p0}), (\ref{berg:Ei_-1})
reduces to
\begin{equation}\label{berg:Ei_-2}
  E_{-}(b) =  \frac{1}{\pi^2} \int_{\PT} q(\txione,\txitwo)
  {\mathcal I}(\txione,\txitwo)
  \, d\txione d\txitwo \,, \quad \mbox{with} \quad
  {\mathcal I} \equiv \int_{\R_{+}^{3}}  \frac{\teta^{\prime}\left(
      w_{\txione^{\prime}\txione^{\prime}} - w_{\txitwo^{\prime}\txitwo^{\prime}}\right)}
  { |\y^{\prime}-\y|} \, d\y^{\prime} \,,
\end{equation}
where $\y^{\prime}=(\txione^{\prime},\txitwo^{\prime},\teta^{\prime})^{T}$ and
$\y=(\txione,\txitwo,0)^{T}$.

Next, we derive an alternative expression to (\ref{berg:Ei_-2}) that,
more conveniently, determines the integral ${\mathcal I}$ not from a
3-D integration, but from an integration only over the
patch $\PT$.  To do so, we introduce an auxiliary function
$\Psi(\y^{\prime})$ satisfying
\begin{equation}\label{app-c:aux_prob}
  \begin{split}
  \Delta_{\y^\prime} \Psi &= \frac{3\teta^{\prime}}{(r^{\prime})^{5}} \left[
    \left(\txione^{\prime}-\txione\right)^2 -
    \left(\txitwo^{\prime}-\txitwo\right)^2 \right] \,,
  \quad \y^{\prime}\in \R_{+}^{3}\,, \\
  \Psi_{\teta^{\prime}} &=0 \,,  \quad  \mbox{on} \quad \teta^{\prime}=0 \,,
\end{split}
\end{equation}
where we have defined $r^{\prime}\equiv |\y^{\prime}-\y|$. In our analysis below
we will use the equivalent representation that
\begin{equation}\label{app-c:aux_force}
  \Delta_{\y^\prime} \Psi = \teta^{\prime} \left(
    \partial_{\txione^{\prime}\txione^{\prime}} \left(\frac{1}{r^{\prime}} \right)
    -\partial_{\txitwo^{\prime}\txitwo^{\prime}} \left(\frac{1}{r^{\prime}} \right)
  \right) \,,   \quad \y^{\prime}\in \R_{+}^{3}\,.
\end{equation}
The explicit solution to (\ref{app-c:aux_prob}) is readily obtained
from a translation and a scaling of the Green's function $G_{-}$ in
(\ref{app_g:g-_final2}), which solved the closely related problem
(\ref{app_g:ng1-_prob}). In this way, we conclude that
\begin{equation}\label{app-c_aux_sol}
  \Psi = - \left( \frac{\teta^{\prime}}{4 (r^{\prime})^3} +
    \frac{1}{8 \left(r^{\prime}+\teta^{\prime}\right)^2} \right)
  \left[ \left(\txione^{\prime}-\txione\right)^2 -
    \left(\txitwo^{\prime}-\txitwo\right)^2 \right] \,,
\end{equation}
so that on the plane $\teta^{\prime}=0$, we have
\begin{equation}\label{app-c_aux_sol_z}
  \Psi\vert_{\teta^{\prime}=0} = - \frac{1}{8}
 \left( \frac{\left(\txione^{\prime}-\txione\right)^2 -
      \left(\txitwo^{\prime}-\txitwo\right)^2 }{
     \left(\txione^{\prime}-\txione\right)^2 +
      \left(\txitwo^{\prime}-\txitwo\right)^2 }\right) \,.
\end{equation}
  
Next, we apply Green's second identity to $\Psi$ and $w$ over a large
hemisphere of radius $\rho^{\prime}\equiv |\y^{\prime}|=\sigma\gg 1$
in the upper half-space $\teta^{\prime}\ge 0$, and use the fact that
$w$ is harmonic, that $\Psi_{\teta^{\prime}}=0$ on $\teta^{\prime}=0$,
and that $w_{\teta^{\prime}}$ is non-vanishing on $\teta^{\prime}=0$
only on the patch $\PT$. In this way, we obtain that
\begin{equation}\label{app-c:final_1}
  \begin{split}
    \int_{\R_{+}^{3}} w \Delta_{\y^\prime} \Psi \, d\y^{\prime} &= \int_{\PT}
    \left(\Psi \partial_{\teta^{\prime}} w \right)_{\teta^{\prime}=0} \,
    d\txione^{\prime} \, d\txitwo^{\prime}  \\
    & \qquad +
\lim_{\sigma\to\infty} \left( \sigma^{2} 
  \int_{0}^{\pi/2} \int_{0}^{2\pi}
  \left( w \partial_{\rho^{\prime}} \Psi - 
      \Psi \partial_{\rho^{\prime}} w \right)
  {\Big\vert}_{\rho^{\prime}=\sigma} \sin\theta^{\prime} d\theta^{\prime}
  \, d\phi^{\prime} \right)\,,
  \end{split}
\end{equation}
where $0\leq \theta^{\prime}\leq {\pi/2}$ and
$0\leq \phi^{\prime}<2\pi$ are the angular spherical
coordinates. Then, by using (\ref{app-c:aux_force}) for
$\Delta_{\y^\prime} \Psi$ on the left-hand side of
(\ref{app-c:final_1}), we integrate by parts and use decay at
infinity. Upon recalling the definition (\ref{berg:Ei_-2}) for
${\mathcal I}$, we conclude that
\begin{equation}\label{app-c:final_2}
  \begin{split}
    {\mathcal I} &= \int_{\R_{+}^{3}}  \frac{\teta^{\prime}\left(
      w_{\txione^{\prime}\txione^{\prime}} - w_{\txitwo^{\prime}\txitwo^{\prime}}\right)}
  { |\y^{\prime}-\y|} \, d\y^{\prime}  = \int_{\PT} 
    \left(\Psi \partial_{\teta^{\prime}} w \right)_{\teta^{\prime}=0} \,
    d\txione^{\prime} \, d\txitwo^{\prime} \,, \\
    & \qquad \quad +
\lim_{\sigma\to\infty} \left( \sigma^{2} 
  \int_{0}^{\pi/2} \int_{0}^{2\pi}
  \left( w \partial_{\rho^{\prime}} \Psi 
    -  \Psi \partial_{\rho^{\prime}} w \right)
  {\Big\vert}_{\rho^{\prime}=\sigma} \sin\theta^{\prime}\, d\theta^{\prime}
  \, d\phi^{\prime} \right) \,.
  \end{split}
\end{equation}
To estimate the integral over the hemisphere, we use
$w\sim {C/\rho^{\prime}}$ and
$\partial_{\rho^{\prime}} w \sim -{C/(\rho^{\prime})^2}$ as
$\rho^{\prime}\to \infty$. In addition, in terms of spherical
coordinates, we calculate that the far-field behavior of $\Psi$ given in
(\ref{app-c_aux_sol}) has the form
\begin{subequations}\label{app-c:Psi_ff}
\begin{equation}
  \Psi \sim g(\theta^{\prime}) \cos(2\phi^{\prime}) +
  {\mathcal O}\left({1/\rho^{\prime}}\right)
  \,, \quad \partial_{\rho^{\prime}} \Psi =
  {\mathcal O}\left({1/(\rho^{\prime})^2}\right) \,,  \quad \mbox{as} \quad
  \rho^{\prime}=|\y^{\prime}|\to \infty \,,
\end{equation}
where $g(\theta^{\prime})$ is defined by
\begin{equation}
 g(\theta^{\prime})\equiv -\frac{1}{4}\cos\theta^{\prime} \sin^{2}(\theta^{\prime})
  -\frac{1}{8}\left(\frac{1-\cos\theta^{\prime}}{1+\cos\theta^{\prime}}\right) \,.
\end{equation}
\end{subequations}
Upon using these far-field estimates of $w$ and $\Psi$ we conclude that
\begin{equation}\label{app-c:Psi_ff_llim}
  \begin{split}
 & \lim_{\sigma\to\infty} \left( \sigma^{2} 
  \int_{0}^{\pi/2} \int_{0}^{2\pi}
  \left( w \partial_{\rho^{\prime}} \Psi - 
     \Psi \partial_{\rho^{\prime}} w \right)
   {\Big\vert}_{\rho^{\prime}=\sigma} \sin\theta^{\prime} \, d\theta^{\prime}
   \, d\phi^{\prime} \right) \\
 & = \lim_{\sigma\to\infty} \left(C \int_{0}^{\pi/2} 
   g(\theta^{\prime}) \, d\theta^{\prime}
   \left(\int_{0}^{2\pi} \cos(2\phi^{\prime})\, d\phi^{\prime}
   \right) +  {\mathcal O}({1/\sigma})\right) =0 \,.
 \end{split}
\end{equation}
As a result, since the limit over the hemisphere in
(\ref{app-c:final_2}) vanishes, we obtain upon using
(\ref{app-c_aux_sol_z}) in (\ref{app-c:final_2}) that
\begin{equation}\label{app-c:final_J}
  {\mathcal I}= -\frac{1}{8} \int_{\PT}
 \left( \frac{\left(\txione^{\prime}-\txione\right)^2 -
      \left(\txitwo^{\prime}-\txitwo\right)^2 }{
     \left(\txione^{\prime}-\txione\right)^2 +
     \left(\txitwo^{\prime}-\txitwo\right)^2 }\right) \,
 \partial_{\teta^{\prime}} w \vert_{\teta^{\prime}=0} \, d\txione^{\prime}\,
 d\txitwo^{\prime} \,.
\end{equation}
Finally, by setting
$\partial_{\teta^{\prime}} w \vert_{\teta^{\prime}=0} =-2q$ in
(\ref{app-c:final_J}) we combine (\ref{app-c:final_J}) and
(\ref{berg:Ei_-2}) to obtain our main result for $E_{-}$ given in
(\ref{berg:Ei_-}) of Lemma \ref{lemma:monom}.

\section{Perfectly reactive elliptically-shaped patches}
\label{app:ellipe}

For a perfectly reactive elliptically-shaped patch of semi-axes $\eps
a_{i1}$ and $\eps a_{i2}$, we now adapt the approach in
\cite{Strieder09} to calculate the reactive capacitance
$C_{i}(\infty)$ in (\ref{berg:wc_charge}) as well as the monopole
coefficients $E_{i+}(\infty)$ and $E_{i-}(\infty)$ in
(\ref{berg:Ei_general0}) and (\ref{berg:Ei_-}), respectively.  This is
done by calculating the charge density $q_i$ from an exact solution
$w_i$ of (\ref{berg:wc}).  In our derivation below, for clarity we
omit the subscript $i$ patch index.

To solve (\ref{berg:wc}) with $b_i=\infty$ when $\PT$ is an ellipse of
semi axes $a_1$ and $a_2$, with $a_1\ge a_2$, we first introduce
ellipsoidal coordinates in the form
\begin{equation}\label{app-ell:coord}
  \frac{\txione^2}{a_{1}^2+\theta} + \frac{\txitwo^2}{a_{2}^{2}+\theta} +
  \frac{\teta^2}{\theta}=1 \,.
\end{equation}
For a given $(\txione,\txitwo,\teta)$ the three roots $\theta$
of (\ref{app-ell:coord}), labeled by $\lambda$, $\mu$, and $\nu$,
satisfy $\lambda\geq 0\geq\mu\ge -a_2^2\ge \nu\ge -a_1^2$. Level sets
of constant $\lambda>0$ are ellipsoids that have the limiting behavior
${\txione^2/a_1^2}+{\txitwo^2/a_2^2}\to 1$ and $\teta\to 0$ as
$\lambda\to 0^{+}$, while for $\lambda\to \infty$ the ellipsoid becomes
a large sphere $\txione^2+ \txitwo^2+\teta^2\to \lambda$ of radius
$\sqrt{\lambda}$.

As a result, we look for a solution to (\ref{berg:wc}) in the form
$w(\txione,\txitwo,\teta)={\mathcal W}(\lambda)$. From separating
variables in the Laplacian written in terms of ellipsoidal coordinates, we
find that ${\mathcal W}(\lambda)$ satisfies the boundary value
problem
\begin{equation}\label{app-ell:Wprob}
  \begin{split}
  & \frac{d}{d\lambda} \left[ \sqrt{\lambda (a_1^2+\lambda)(a_2^2+\lambda)}
    \frac{d {\mathcal W}}{d\lambda} \right] =0 \,, \quad 0<\lambda<\infty \,,\\
  & \quad {\mathcal W}(0^{+})=1 \,, \quad \mbox{(on elliptical patch)} \,;
  \qquad {\mathcal W}\to  0  \quad \mbox{as} \quad \lambda\to \infty\,.
  \end{split}
\end{equation}
The solution is
\begin{equation}\label{app-ell:Wsol_1}
  {\mathcal W}(\lambda) = \frac{1}{A} \int_{\lambda}^{\infty} \frac{d\zeta}{
    \sqrt{\zeta (a_1^2+\zeta)(a_2^2+\zeta)}} \,, \quad
   A \equiv \int_{0}^{\infty} \frac{d\zeta}{
    \sqrt{\zeta (a_1^2+\zeta)(a_2^2+\zeta)}} \,.
\end{equation}    
When $a_1>a_2$, $A$ can be written in terms of the complete elliptic
integral $K(m)$ as
\begin{equation}\label{app-ell:K}
  A = \frac{2}{a_1} K(m) \,, \quad \mbox{with} \quad
  m \equiv \sqrt{1 - \frac{a_2^2}{a_1^2}} \,, \quad \mbox{where} \quad
  K(m)\equiv \int_{0}^{\pi/2} \frac{1}{\sqrt{1-m^2\sin^{2}\theta}} d\theta\,,
\end{equation}
so that (\ref{app-ell:Wsol_1}) becomes
\begin{equation}\label{app-ell:Wsol}
  {\mathcal W}(\lambda) = \frac{a_1}{2K(m)}
  \int_{\lambda}^{\infty} \frac{d\zeta}{
    \sqrt{\zeta (a_1^2+\zeta)(a_2^2+\zeta)}}  \,.
\end{equation}

To determine the reactive capacitance $C(\infty)$ we let $\lambda\to \infty$ in
(\ref{app-ell:Wsol}) and use that $|\y|=\txione^2+\txitwo^2+\teta^2\approx
\lambda$ as $\lambda\to \infty$, so that $|\y|\sim \sqrt{\lambda}$.  This
yields for $|\y|\to \infty$ that
\begin{equation*}
  {\mathcal W} \sim \frac{a_1}{2 K(m)} \int_{\lambda}^{\infty}
  \zeta^{-3/2} \, d\zeta \sim \frac{a_1}{K(m) \sqrt{\lambda}} \sim
  \frac{a_1}{K(m) |\y|} \,, \quad \mbox{as} \quad |\y|\to\infty \,,
\end{equation*}
which identifies that $C(\infty)={a_1/K(m)}$. In this way, we recover
the classical result for the capacitance of an elliptic patch (see,
e.g., \cite{Strieder09}).  If $a_1<a_2$, we simply swap $a_1$ and
$a_2$ in the formula for $C(\infty)$.

Next, we determine the charge density $q(\y;\infty)$ by first
calculating $\partial_{\teta} \lambda$. From an implicit
differentiation of (\ref{app-ell:coord}) we get
\begin{equation}\label{app-ell:lam_z1}
  \frac{\partial\lambda}{\partial\teta} = 2\teta \left[
    \frac{\teta^2}{\lambda} + \lambda
    \left( \frac{\txione^2}{(a_1^2+\lambda)^2} +
      \frac{\txitwo^2}{(a_2^2+\lambda)^2} \right)\right]^{-1} \,.
\end{equation}
To evaluate (\ref{app-ell:lam_z1}) on the elliptical patch we let
$\lambda\to 0$ in (\ref{app-ell:lam_z1}) while using
${\teta^{2}/\lambda}\to 1-{\txione^2/a_1^2} -{\txitwo^2/a_2^2}$ as
$\lambda\to 0$.  This yields that
\begin{equation}\label{app-ell:lam_z}
  \frac{\partial\lambda}{\partial\teta} \to \frac{2\teta}{{\teta^2/\lambda}}
  = \frac{2\lambda}{\teta} = 2\sqrt{\lambda} \left( 1 -\frac{\txione^2}{a_1^2}
    - \frac{\txitwo^2}{a_2^2} \right)^{-1/2} \,, \quad \mbox{as} \quad
  \teta\to 0 \,, \quad \frac{\txione^2}{a_1^2} + \frac{\txitwo^2}{a_2^2} \leq 1.
\end{equation}
Then, we calculate ${\mathcal W}^{\prime}(\lambda)$ from (\ref{app-ell:Wsol})
and estimate it as $\lambda\to 0$ to obtain
\begin{equation}\label{app-ell:wlambda}
  {\mathcal W}^{\prime}(\lambda) = -\frac{a_1}{2K(m)} \frac{1}{
    \sqrt{\lambda (\lambda+a_1^2) (\lambda+a_2^2)}} \sim -\frac{\lambda^{-1/2}}
      {2 a_2 K(m)} \,, \quad \mbox{as} \quad \lambda \to 0 \,.
\end{equation}
From using (\ref{app-ell:lam_z}) and (\ref{app-ell:wlambda}) to
evaluate the chain rule $w_{\teta}={\mathcal W}^{\prime}(\lambda)
\partial_{\teta} \lambda$ as $\lambda\to 0$, we calculate the
charge density on the patch as 
\begin{equation}\label{app-ell:q_temp}
  q(\y;\infty)  =-\tfrac{1}{2}w_{\teta}\vert_{\teta=0} = \frac{1}{2a_2 K(m)}
  \left( 1 -\frac{\txione^2}{a_1^2}- \frac{\txitwo^2}{a_2^2} \right)^{-1/2}
  = \frac{C}{2a_1 a_2} \left( 1 -\frac{\txione^2}{a_1^2}-
    \frac{\txitwo^2}{a_2^2} \right)^{-1/2} ,
\end{equation}
where $C=C(\infty)$ is the reactive capacitance.  Therefore, we have
recovered the result presented in \cite{Strieder09}.

Finally, to determine the monopole coefficients $E_{+}(\infty)$ and
$E_{-}(\infty)$ we substitute the second expression in
(\ref{app-ell:q_temp}) for $q$ into (\ref{berg:Ei_general0}) and
(\ref{berg:Ei_-}), respectively, to yield iterated double integrals
over the elliptical patch $\PT$.  Upon replacing
$\tilde{\txione}={\txione/a_1}$ and $\tilde{\txitwo}={\txitwo/a_2}$ in
these formulae (and dropping the tilde) the integrations can be
expressed in terms of the unit disk $\mathbb{D}$ centered at the origin.
{\clb In this way, we obtain from (\ref{berg:Ei_general0}) that
\begin{equation}\label{app:e+}
    \frac{E_{+}}{C^2}= -\frac{1}{8\pi^2} \int_{\mathbb{D}}
    \frac{1}{\sqrt{1 - |\y^{\prime}|^2}}
      \left(\int_{\mathbb{D}}  \frac{\log\left[ a_{1}^2(\txione-\txione^{\prime})^2 +
            a_2^{2} (\txitwo-\txitwo^{\prime})^2\right]^{1/2}} {
          {\sqrt{1 - |\y|^2}}} \, d\y \right) \,
      d\y^{\prime} \,,
\end{equation}
while from (\ref{berg:Ei_-}) we derive
\begin{equation}\label{app:e-}
    \frac{E_{-}}{C^2}= \frac{1}{16\pi^2} \int_{\mathbb{D}}
    \frac{1}{\sqrt{1 - |\y^{\prime}|^2}}
      \left(\int_{\mathbb{D}}  \left(\frac{ a_1^2(\txione^{\prime}-\txione)^2 -
          a_2^2(\txitwo^{\prime}-\txitwo)^2}{
          a_1^2(\txione^{\prime}-\txione)^2 +
          a_2^2(\txitwo^{\prime}-\txitwo)^2}\right)\frac{1}
          {\sqrt{1 - |\y|^{2}}} \, d\y \right) \,
      d\y^{\prime} \,.
\end{equation}
Here $|\y|^2=\txione^2+\txitwo^2$ and
$|\y^{\prime}|^2=(\txione^{\prime})^2+
(\txitwo^{\prime})^2$. Remarkably, these two iterated double integrals
can be evaluated analytically in closed form. We summarize our results
as follows:}

\begin{lemma}\label{lemma:ellipse} 
Let $\PT$ be a perfectly absorbing elliptical-shaped patch
${\txione^2/a_1^2}+{\txitwo^2/a_2^2}\leq 1$ with semi-axes $a_1>0$ and
$a_2>0$ that are aligned with the two principal directions of
curvature.  Then, the reactive capacitance $C=C(\infty)$ and the
charge density $q=q(\y;\infty)$ defined in (\ref{berg:wc_charge}) are
by
\begin{equation}\label{app:c1}
    C=\frac{a_{>}}{K(m)} \,, \quad q = \frac{C}{2a_2a_1} \frac{1}
    {\sqrt{1 -{\txione^2/a_1^2}-{\txitwo^2/a_2^2}}} \,, \quad m\equiv
    \sqrt{1- \frac{a_{<}^2}{a_{>}^2}} \,,
\end{equation}
where $a_{<}=\min(a_1,a_2)$, $a_{>}=\max(a_1,a_2)$, and $K(m)$ is the
complete elliptic integral of the first kind.  Finally, the integrals
{\clb in (\ref{app:e+}) and (\ref{app:e-}) for the two ratios ${E_{\pm}/C^2}$
are given explicitly by
\begin{equation}\label{app:amazing}
     \frac{E_{+}}{C^2} = -\frac{1}{2}\log\left(\frac{a_1+a_2}{2}\right)
     -\log{2} + \frac{3}{4}  \,, \qquad \frac{E_{-}}{C^2} = \frac{a_1-a_2}
     {4(a_1+a_2)} \,.
\end{equation}}
\end{lemma}

We need only derive the explicit results in (\ref{app:amazing}).  Our
derivation relies on two elementary results from complex analysis.

\begin{lemma}\label{lemma:complex1}
Suppose that $\mathsf{P}(z)$ is a polynomial of degree $M>1$ of the
form $\mathsf{P}(z)=d\prod_{j=1}^{M} (z-z_j)$, with $d\ne 0$, for
which $P(0)\neq 0$ and $P(z)$ is non-vanishing on the boundary
$|z|=1$.  Suppose that $|z_j|>1$ for $j=1,\ldots,k$, while $|z_j|<1$
for $j=k+1,\ldots,M$ for some $k\in\lbrace{0,\ldots,M\rbrace}$.  Then,
\begin{equation}\label{app:jensen}
    \frac{1}{2\pi} \int_{0}^{2\pi} \log| \mathsf{P}\left(e^{i\theta}\right)|\,
    d\theta = \log|d| + \sum_{j=1}^{k} \log|z_j| \,.
\end{equation}
\end{lemma}

\begin{proof} 
By applying Jensen's formula (see \cite{stein}) to the polynomial
$\mathsf{P}(z)$, we obtain in terms of its roots inside the unit disk
that
\begin{equation}\label{app:jensen_thm}
    \frac{1}{2\pi} \int_{0}^{2\pi} \log| \mathsf{P}\left(e^{i\theta}\right)|\,
    d\theta = \log|\mathsf{P}(0)| - \sum_{j=k+1}^{M} \log|z_j| \,.
\end{equation}
We readily obtain (\ref{app:jensen}) by substituting
$|\mathsf{P}(0)|=|d|\prod_{j=1}^{M}|z_j|$ into (\ref{app:jensen_thm}).
When there {\clb is no root} outside the unit disk, the second term on
the right side of (\ref{app:jensen_thm}) vanishes.
\end{proof}

The second result that we need from complex analysis is the following:

\begin{lemma}\label{lemma:complex2} 
Consider the quadratic polynomial $\mathsf{A}(z)=z^2+\zeta z + \beta$,
where $\beta$ is real with $-1<\beta<1$ and $\zeta$ is complex-valued.
Then, the two roots of $\mathsf{A}(z)$ are inside the unit disk
$|z|<1$ if and only if $\zeta=\zeta_{R}+i\zeta_{I}$ satisfies
\begin{equation}\label{app:inside}
      \frac{ \zeta_{R}^2}{(1+\beta)^2} + \frac{\zeta_{I}^2}{(1-\beta)^2}<1\,.
    \end{equation}
\end{lemma}
  
\begin{proof}
We prove this result by calculating the winding number of
$\mathsf{A}$, labeled by $N_{w}(\mathsf{A})$, as $z$ traverses the
unit disk $|z|=1$ once counterclockwise, with parameterization
$z=e^{i\theta}$ for $0\leq \theta\leq 2\pi$.  We write
$\mathsf{A}(e^{i\theta})= e^{i\theta} h(\theta)$, with
$h(\theta)\equiv e^{i\theta}+ \zeta + \beta e^{-i\theta}$, so that
\begin{equation*}
      N_{w}\left(\mathsf{A}(e^{i\theta})\right) = N_{w}\left(e^{i\theta}\right) +
      N_{w}\left(h(\theta)\right)\,.
\end{equation*}
Since $N_{w}\left(e^{i\theta}\right)=1$, we conclude that
$\mathsf{A}(z)$ has its two roots inside the unit disk if and only if
$N_{w}\left(h(\theta)\right)=1$.  By separating $h(\theta)$ into real
and imaginary parts as $h(\theta)=h_{R}+ih_{I}$, we calculate that
\begin{equation}
    h_{R} = \zeta_{R} + (1+\beta)\cos\theta \,, \quad h_{I} \equiv
     \zeta_{I} + (1-\beta) \sin\theta \,,
\end{equation}
where $\zeta=\zeta_{R}+i\zeta_{I}$. Since $-1<\beta<1$, it follows
that the image of the unit disk is
\begin{equation}\label{app:quad_ell}
    \frac{(h_R-\zeta_R)^2}{(1+\beta)^2} + \frac{(h_I-\zeta_I)^2}
    {(1-\beta)^2} =1 \,.
\end{equation}
The necessary and sufficient condition for $N_w(h)=1$ is that the
origin $(h_R,h_I)=(0,0)$ lies strictly within the ellipse
(\ref{app:quad_ell}), which yields (\ref{app:inside}).
\end{proof}

With these two elementary results, we now derive the explicit
formulas in (\ref{app:amazing}).

\begin{proof}{(Proof of (\ref{app:amazing}) of Lemma \ref{lemma:ellipse})}
We will first prove (\ref{app:amazing}) for ${E_+/C^2}$.  We begin by
introducing
\begin{equation*}
    \txione=\rho\cos\theta\,, \quad \txitwo=\rho\sin\theta \,, \quad
    \txione^{\prime}=\rho^{\prime}\cos\psi \,, \quad \txitwo^{\prime}=
    \rho^{\prime}\sin\psi \,,
\end{equation*}
so that the argument of the logarithm in (\ref{app:e+}) becomes
\begin{equation*}
   \mathsf{F}(\theta,\psi) \equiv a_{1}^2(\txione-\txione^{\prime})^2 +
   a_2^{2} (\txitwo-\txitwo^{\prime})^2  = \mathsf{X}^2 +
   \mathsf{Y}^2 = \left(\mathsf{X}+ i \mathsf{Y}\right)
   \left(\mathsf{X}-i \mathsf{Y}\right)\,,
\end{equation*}
where $\mathsf{X}\equiv
a_1\left(\rho\cos\theta-\rho^{\prime}\cos\psi\right)$,
$\mathsf{Y}\equiv
a_2\left(\rho\sin\theta-\rho^{\prime}\sin\psi\right)$, and
$i=\sqrt{-1}$.  By defining $\mathsf{U}\equiv \mathsf{X}+i\mathsf{Y}$,
this leads to the factorization
\bsub \label{app:factor}
\begin{equation}\label{app:factor_1}
 \mathsf{F}(\theta,\psi) = \mathsf{U}(\theta,\psi) \mathsf{U}(-\theta,
  -\psi) \,,
\end{equation}
where $\mathsf{U}=\mathsf{U}(\theta,\phi)$ is defined by
\begin{equation}\label{app:factor_2}
  \mathsf{U}(\theta,\phi)= \rho\left(a_1\cos\theta + i a_2 \sin\theta\right)
  -\rho^{\prime}\left(a_1\cos\psi + i a_2 \sin\psi\right)\,.
\end{equation}
\esub
By performing the angular integrations we observe from symmetry that
\begin{equation}\label{app:ang_1}
  \begin{split}
  \int_{0}^{2\pi}  \int_{0}^{2\pi} \log\left[\mathsf{F}(\theta,\psi)\right]
    d\theta\, d\psi &=  \int_{0}^{2\pi}  \int_{0}^{2\pi}
      \left( \log |\mathsf{U}(\theta,\psi)| +
        \log |\mathsf{U}(-\theta,-\psi)| \right) \, d\theta \, d\psi \\
    & = 2 \int_{0}^{2\pi}\int_{0}^{2\pi} \log|\mathsf{U}(\theta,\psi)|
    \, d\theta \, d\psi \,.
  \end{split}
\end{equation}
Upon substituting (\ref{app:ang_1}) into (\ref{app:e+}) and converting
to polar coordinates we obtain
\begin{equation}\label{app:e+new}
  \frac{E_{+}}{C^2} = -\frac{1}{8\pi^2} \int_{0}^{1} \int_{0}^{1}
  \frac{ {\mathcal J} \rho \rho^{\prime}}{\sqrt{1-\rho^2}\sqrt{1-(\rho^{\prime})^2}}
   \, d\rho \, d\rho^{\prime} \,, \quad
  {\mathcal J}\equiv \int_{0}^{2\pi}\int_{0}^{2\pi}
  \log|\mathsf{U}(\theta,\psi)| \, d\theta \, d\psi \,.
\end{equation}

Next, we will calculate ${\mathcal J}={\mathcal
J}(\rho,\rho^{\prime})$ explicitly. To do so, we first introduce the
complex variables $z=e^{i \theta}$ and $w=e^{i \psi}$, so that after
some algebra we write (\ref{app:factor_2}) as
\begin{equation}\label{app:u_new}
  \mathsf{U} = d_1 \left(z+ \frac{\beta}{z}\right) + d_2
  \left(w+ \frac{\beta}{w}\right)\,,
\end{equation}
where $d_1$, $d_2$ and $\beta$, satisfying $-1<\beta<1$ are defined by
\begin{equation}\label{app:u_new_p}
  d_1 \equiv \frac{\rho}{2} (a_1+a_2) \,, \quad d_2 \equiv
  -\frac{\rho^{\prime}}{2}  (a_1+a_2) \,, \quad  \beta \equiv
  \frac{a_1-a_2}{a_1+a_2}\,.
\end{equation}

{\bf Case I:} Suppose that $|d_1|>|d_2|$, so that
$\rho>\rho^{\prime}$. For this case we write $\mathsf{U}$ in
(\ref{app:u_new}) in terms of a polynomial in $z$ as
\bsub\label{app:int1}
\begin{equation}\label{app:int1a}
  \mathsf{U} = \frac{\mathsf{P}(z)}{z}\,, \quad \mbox{with} \quad
  \mathsf{P}(z) = d_1 \mathsf{A}(z)\,,
\end{equation}
where
\begin{equation}\label{app:q}
  \mathsf{A}(z) = z^2 + \zeta z + \beta \,, \quad \mbox{with} \quad
  \zeta\equiv -\rho_{\star}  \left(w + \frac{\beta}{w}\right) \,,
\end{equation}
\esub
and $\rho_{\star}\equiv {\rho^{\prime}/\rho}<1$. We set $z=e^{i
\theta}$ and integrate $\log|\mathsf{U}|$ first with respect to
$\theta$, for each fixed $w=e^{i\psi}$ with $0<\psi<2\pi$. On the unit
disk we have from (\ref{app:int1a}) that $\log|\mathsf{U}| =
\log|\mathsf{P}(e^{i\theta})|$, and so since $\mathsf{P}(z)=d_1
\mathsf{A}(z)$, we must first determine the location of the zeroes of
$\mathsf{A}(z)$ to apply Lemma \ref{lemma:complex1}.

To do so, we use Lemma \ref{lemma:complex2}.  By setting $w=e^{i\psi}$
in $\zeta$ in (\ref{app:q}) we decompose $\zeta=\zeta_R+i\zeta_I$ to
get $\zeta_R=-\rho_{\star}(1+\beta)\cos\psi$ and
$\zeta_I=-\rho_{\star}(1-\beta) \sin\psi$.  The criterion
(\ref{app:inside}) becomes
\begin{equation*} \frac{\rho_{\star}^2
(1+\beta)^2}{(1+\beta)^2} \cos^{2}\psi + \frac{\rho_{\star}^2
(1-\beta)^2}{(1-\beta)^2} \sin^{2}\psi =
\rho_{\star}^2 < 1 \,,
\end{equation*}
which holds since $|\rho_{\star}|<1$.  As a result, $\mathsf{A}(z)$
and $\mathsf{P}(z)$ have no roots outside the unit disk.  Therefore,
by (\ref{app:jensen}) of Lemma \ref{lemma:complex1} we have
\begin{equation}
  \frac{1}{2\pi} \int_{0}^{2\pi} \log|\mathsf{U}|\, d\theta =
  \frac{1}{2\pi}\int_{0}^{2\pi} \log|\mathsf{P}(e^{i\theta})| \, d\theta=
  \log|d_1| \,.
\end{equation}
Upon trivially performing the second integration over $\psi$, we
conclude that ${\mathcal J}$ in (\ref{app:e+new}) in {\bf Case I} is
\begin{equation}\label{app:j1}
  {\mathcal J} = 4\pi^2 \log|d_1| \,, \quad \mbox{with} \quad
  d_1= \frac{\rho}{2} (a_1+a_2)  \,.
\end{equation}

For {\bf Case II} where $|d_1|<|d_2|$, so that $\rho<\rho^{\prime}$,
we we write $\mathsf{U}$ in (\ref{app:u_new}) in terms of a polynomial
in $w$ as
\bsub\label{app:int2}
\begin{equation}
  \mathsf{U} = \frac{\mathsf{P}(w)}{w}\,, \quad \mbox{with} \quad
  \mathsf{P}(w) = d_2 \mathsf{A}(w)\,,
\end{equation}
where $\mathsf{A}(w)$ is now defined by
\begin{equation}\label{app:qw}
  \mathsf{A}(w) = w^2 + \zeta w + \beta \,, \quad \mbox{with} \quad
  \zeta\equiv \rho_{\star}  \left(z + \frac{\beta}{z}\right) \,,
\end{equation}
\esub
and $\rho_{\star}\equiv {\rho/\rho^{\prime}}<1$.  By repeating the
analysis as in Case I, it follows by Lemma \ref{lemma:complex2} that
$\mathsf{A} (w)$ has no roots in $|w|>1$.  From Lemma
\ref{lemma:complex1} we get ${\mathcal J}= 4\pi^2 \log|d_2|$.
  
By combining these two cases, and by recalling (\ref{app:u_new_p}) for
$d_1$ and $d_2$, we have derived the key result
\begin{equation}\label{app:key}
  {\mathcal J} = 4\pi^{2} \log\left(\max(|d_1|,|d_2|)\right) = 4\pi^2 
    \log\left(\overline{a} \rho_{>} \right) \,; \quad 
  \rho_{>}\equiv \max(\rho,\rho^{\prime})\,,\quad \overline{a}\equiv
  \frac{(a_1+a_2)}{2},
\end{equation}
so that we need only to perform the radial integrations.  Upon
substituting (\ref{app:key}) into (\ref{app:e+new}), and using
$\int_{0}^{1}\left(1-\rho^2\right)^{-1/2}\, d\rho=1$, we obtain that
\begin{equation}\label{app:e+neww}
  \frac{E_{+}}{C^2} = -\frac{1}{2} \log \overline{a} - \frac{1}{2}
  \int_{0}^{1}\int_{0}^{1} \mathsf{K}(\rho,\rho^{\prime}) \, d\rho \,
  d\rho^{\prime}\,; \qquad \mathsf{K} \equiv \frac{\rho \rho^{\prime} \log \rho_{>}} {\sqrt{1-\rho^2}\sqrt{1-(\rho^{\prime})^2}} \,.
\end{equation}
Since $\mathsf{K}$ is symmetric about the diagonal, i.e.
$\mathsf{K}(\rho,\rho^{\prime})=\mathsf{K}(\rho^{\prime},\rho)$,
(\ref{app:e+neww}) becomes
\begin{equation}\label{app:e+newww}
\frac{E_{+}}{C^2} = -\frac{1}{2} \log \overline{a} - 
\int_{0}^{1} \frac{\rho \log\rho}{\sqrt{1-\rho^2}} \left(
  \int_{0}^{\rho} \frac{\rho^{\prime}}{\sqrt{1-(\rho^{\prime})^2}} \, d\rho^{\prime}
\right) d\rho\,.
\end{equation}
The inner integral in (\ref{app:e+newww}) is simply
$1-\sqrt{1-\rho^2}$, so that (\ref{app:e+newww}) reduces to
\begin{equation}\label{app:e+new4}
  \frac{E_{+}}{C^2} = -\frac{1}{2} \log \overline{a} +\int_{0}^{1}\rho\log\rho
  \, d\rho - \int_{0}^{1} \frac{\rho \log\rho}{\sqrt{1-\rho^2}} \, d\rho\,.
\end{equation}
Integration by parts yields $\int_{0}^{1}\rho \log\rho \, d\rho={-1/4}$.  To
evaluate the second integral we set $t=\sqrt{1-\rho^2}$, and then integrate
by parts to get
\begin{equation*}
  \int_{0}^{1} \frac{\rho \log\rho}{\sqrt{1-\rho^2}} \, d\rho =
  \frac{1}{2} \int_{0}^{1} \log\left(1-t^2\right) \, dt = \log{2}-1 \,.
\end{equation*}
In this way, we obtain
${E_{+}/C^2}=-\tfrac{1}{2}\log\overline{a} -{1/4} - \log{2}+1$, with
$\overline{a}={(a_1+a_2)/2}$, which confirms the first result in
(\ref{app:amazing}).

To calculate ${E_-/C^2}$ from (\ref{app:e-}), we simply observe that
\begin{equation*}
  \frac{E_-}{C^2} = -\frac{a_1}{2}\frac{\partial}{\partial a_1}
  \left( \frac{E_+(a_1,a_2)}{C^2(a_1,a_2)} \right) +
  \frac{a_2}{2}\frac{\partial}{\partial a_2}\left(
    \frac{E_+(a_1,a_2)}{C^2(a_1,a_2)} \right) \,.
\end{equation*}
Therefore, by using our formula for ${E_{+}/C^2}$ we get
${E_-/C^2} = {(a_1-a_2)/(4[a_1+a_2])}$, which confirms the second
result in (\ref{app:amazing}). 
\end{proof}

Finally, we remark that since $C$, $q$, and $E_{+}$ are invariant
under rotations, these quantities are independent of the orientation
of the elliptical patch with respect to the principal directions of
curvature. However, in our derivation of $E_{-}$ we assumed that the
semi-axes of the small elliptical patch were aligned with the
principal directions of curvature on the boundary through the center
of the patch. The more general case of an oblique elliptical patch is
readily treated through a rotation matrix.

\bibliographystyle{plain} \bibliography{references_n}

\end{document}